\documentclass{article}

\usepackage{amsmath,amssymb,amsfonts,amsthm,epsf,a4,color,graphicx,eucal,mathrsfs,mathabx,xcolor}
\usepackage{xcolor}
\usepackage[colorlinks=true, pdfstartview=FitV, linkcolor=blue, 
            citecolor=blue, urlcolor=blue]{hyperref}
            \usepackage{todonotes}

\usepackage{bigfoot}

\usepackage{epstopdf}
\usepackage{subfig}
\DeclareNewFootnote{AAffil}[arabic]
\DeclareNewFootnote{ANote}[fnsymbol]

\usepackage{etoolbox}
\makeatletter
\patchcmd\maketitle{\def\@makefnmark{\rlap{\@textsuperscript{\normalfont\@thefnmark}}}}{}{}{}
\makeatother

\makeatletter
\def\thanksAAffil#1{
  \footnotemarkAAffil\protected@xdef\@thanks{\@thanks%
        \protect\footnotetextAAffil[\the \c@footnoteAAffil]{#1}}%
}
\def\thanksANote#1{%
  \footnotemarkANote%
  \protected@xdef\@thanks{\@thanks%
        \protect\footnotetextANote[\the \c@footnoteANote]{#1}}%
}
\makeatother

\newtheorem{theorem}{Theorem}[section]
\newtheorem{lemma}[theorem]{Lemma}

\newtheorem{definition}[theorem]{Definition}

\newtheorem{proposition}[theorem]{Proposition}

\newtheorem{remark}[theorem]{Remark}
\newtheorem{notation}[theorem]{Notation}

\newtheorem{maintheorem}{Theorem}
\newtheorem{condition}{Condition}

\newtheorem{thmx}{Theorem}

\newtheorem{hypx}[thmx]{Hypothesis}

\usepackage{xcolor}
\usepackage{amsmath}
\usepackage{amsthm}
\usepackage{mathtools}
\usepackage{amssymb}
\usepackage{eucal}
\usepackage{epsfig}
\usepackage{psfrag}
\usepackage{graphicx}
\usepackage{paralist}
\usepackage{textcomp}
\usepackage{pgf}   
\usepackage{l3keys2e, array, pgfcore}
\usepackage{nicematrix}

\numberwithin{equation}{section}

\newcommand{\R}{\mathbb{R}}
\newcommand{\N}{\mathbb{N}}

\newcommand\JUMP[1]{\mathchoice
                  {\big[\hspace*{-.3em}\big[#1\big]\hspace*{-.3em}\big]}
                   {[\hspace*{-.15em}[#1]\hspace*{-.15em}]}
                   {[\![#1]\!]}
                   {[\![#1]\!]}}
\newcommand{\GC}{\Gamma_{\!\scriptscriptstyle{\rm C}}}
\newcommand{\gC}{\gamma_{\!\scriptscriptstyle{\rm C}}}
\newcommand{\op}{\omega_+}
\newcommand{\om}{\omega_-}
\newcommand{\GD}{\Gamma_{\!\scriptscriptstyle{\rm D}}}
\newcommand{\gD}{\gamma_{\!\scriptscriptstyle{\rm D}}}
\newcommand{\GN}{\Gamma_{\!\scriptscriptstyle{\rm N}}}

\newcommand{\barGC}{\overline{\Gamma}_{\!\scriptscriptstyle{\rm C}}}
\newcommand{\barGD}{\overline{\Gamma}_{\!\scriptscriptstyle{\rm D}}}
\newcommand\calE{\mathcal E}
\newcommand\calR{\mathcal R}

\newcommand\calG{\mathcal G}
\newcommand\calK{\mathcal K}

\newcommand\calH{\mathcal H}

\newcommand\calV{\mathcal V}
\newcommand\calJ{\mathcal J}

\newcommand\pl\partial

\newcommand{\Var}{\mathrm{Var}}

\newcommand\bbC{\mathbb C}
\newcommand\bbA{\mathbb A}
\newcommand\bbM{\mathbb M}
\newcommand\bbD{\mathbb D}

\newcommand{\surf}{\mathrm{surf}}
\newcommand{\bulk}{\mathrm{bulk}}
\newcommand{\romanC}{\mathrm{C}}

\newcommand{\SBV}{\mathrm{SBV}}

\newcommand{\Omegaone}{\Omega_+}
\newcommand{\Omegatwo}{\Omega_-}
\newcommand{\barOmegaone}{\overline{\Omega}_+}
\newcommand{\barOmegatwo}{\overline{\Omega}_-}

\newcounter{myfigure}
\newenvironment{my-picture}[3]{\refstepcounter{myfigure}\label{#3}\setlength{\unitlength}{\textwidth}\begin{picture}(#1,#2)}{\end{picture}}
\newcommand{\Surf}{\mathcal{H}^{2}}

\newcommand{\Spu}{H^1_{\GD}(\Omega{\setminus}\GC;\R^3)} 

\newcommand{\Spv}{\Spu} 
\newcommand{\Spvr}{H^1_{\GD}(\Omega{\setminus}\GC)} 
\newcommand{\Spw}{L^2(\Omega;\R^3)}
\newcommand{\Spz}{L^1(\GC)} 

\newcommand{\grsysepsk}{(\calV_{\eps_k},\calR_{\eps_k},  \calE_{\eps_k})}
\newcommand{\ingrsysreps}{(\rese K\eps,\rese V\eps, \rese R \eps, \rese E\eps)}
\newcommand{\ingrsyslim}{(\rese K{},\Veve, \rese R {}, \Eve)}
\newcommand{\ingrsysrepsk}{(\rese K{\eps_k},\rese V{\eps_k},\rese R{\eps_k},\rese E{\eps_k})}

\newcommand{\eps}{\varepsilon}

\newcommand{\pairing}[4]{ \sideset{_{ #1 }}{_{ #2 }}  {\mathop{\langle #3 , #4
\rangle}}}

\newcommand{\foraa}{\text{for a.a.\,}}
\newcommand{\aein}{\text{a.e.\ in}}

\newcommand{\dissr}{\calR}

\newcommand{\Msym}{\mathbb{R}_{\mathrm{sym}}^{3\times 3}}

\newcommand{\dd}{\, \mathrm{d}}

\newcommand{\BV}{\mathrm{BV}}
\newcommand{\weakto}{\rightharpoonup}
\newcommand{\weaksto}{\overset{*}{\rightharpoonup}}
\newcommand{\QED}{\mbox{}\hfill\rule{5pt}{5pt}\medskip\par}
\newcommand{\down}{\downarrow}
\newcommand{\Dir}{\GD}
\newcommand{\dir}{\gD}

\newcommand{\res}[1]{\mathfrak{F}[#1]}
\newcommand{\resop}{\mathfrak{F}}
\newcommand{\resh}[1]{\widehat{\mathfrak{F}}[#1]}
\newcommand{\reshop}{\widehat{\mathfrak{F}}}
\newcommand{\resg}[1]{\widehat{\mathfrak{f}}[#1]}
\newcommand{\resgop}{\widehat{\mathfrak{f}}}
\newcommand{\resz}[1]{\mathfrak{f}[#1]}
\newcommand{\reszop}{\mathfrak{f}}
\newcommand{\resw}[1]{\widecheck{\mathfrak{F}}[#1]}

\newcommand{\rese}[2]{\mathsf{#1}^{#2}}
\newcommand{\resed}[2]{\dot{\mathsf{#1}}^{#2}}
\newcommand{\resedd}[2]{\ddot{\mathsf{#1}}^{#2}}

\newcommand{\resei}[3]{\mathsf{#1}^{#2}_{#3}}
\newcommand{\reseid}[3]{\dot{\mathsf{#1}}^{#2}_{#3}}
\newcommand{\reseidd}[3]{\ddot{\mathsf{#1}}^{#2}_{#3}}
\newcommand{\reseiddd}[3]{\dddot{\mathsf{#1}}^{#2}_{#3}}
\newcommand{\cv}[1]{r_{#1}} 
\newcommand{\reso}[2]{\Lambda_{#1}(#2)}
\newcommand{\epl}{e_{\mathrm{plan}}}
\newcommand{\sfu}{\mathsf{u}}
\newcommand{\sfx}{\mathsf{x}}
\newcommand{\sfz}{\mathsf{z}}
\newcommand{\sfw}{\mathsf{w}}
\newcommand{\sff}{\mathsf{f}}
\newcommand{\sfd}{\mathsf{d}}
\newcommand{\sfe}{\mathsf{e}}
\newcommand{\Bigf}{\mathrm{F}}
\newcommand{\loc}{\mathrm{loc}}
\newcommand{\red}{\mathrm{r}}
\newcommand{\sym}{\mathrm{sym}}
\newcommand{\KL}{\mathrm{KL}(\Omega;\R^3)}
\newcommand{\KLGD}{\mathrm{KL}_{\GD}(\Omega;\R^3)}
\newcommand{\Dr}{\mathsf{D}_{\mathrm{r}}}
\newcommand{\Cr}{\mathbb{C}_{\mathrm{r}}}
\newcommand{\weak}{\mathrm{weak}}
\definecolor{ddcyan}{rgb}{0,0.2,0.8}
\definecolor{ddmagenta}{rgb}{0.8,0,0.8}
\definecolor{orange}{rgb}{0.6,0.2,0}
\definecolor{vgreen}{rgb}{0.1,0.5,0.2}
\definecolor{dred}{rgb}{.8,0,0}
\definecolor{Turk}{rgb}{0,0.7,0.4}
\definecolor{purple}{rgb}{0.5,0,0.9}
\definecolor{dcyan}{rgb}{0,0.2,1.0}
\definecolor{dblue}{rgb}{0,0,0.5}

\newcommand{\nc}{\normalcolor}

\newcommand{\EEE}{\color{black}}
\newcommand{\RNEW}{\color{black}} 

\newcommand{\RCO}{\color{black}}

\newcommand{\GIO}{\color{vgreen}}

\newcommand{\Mextname}{\mathcal{M}_{\kern-1pt{\tiny\textsc{ve}}}}
\newcommand{\Eve}{\mathsf{E}_{\kern-1pt{\tiny\textsc{ve}}}}
\newcommand{\Veve}{\mathsf{V}_{\kern-1pt{\tiny\textsc{ve}}}}
\newcommand{\Mext}[1]{\Mextname{#1}}
\newcommand{\mix}[2]{\left[{#1}|{#2} \right]}

\newcommand{\zzeta}{\boldsymbol{\zeta}}
\newcommand{\eeta}{\boldsymbol{\eta}}
\newcommand{\llambda}{\boldsymbol{\lambda}}
\newcommand{\uupsilon}{\boldsymbol{\upsilon}}

\newcommand{\rmA}{\mathrm{A}}

\newcommand{\Temp}{\mathfrak{T}}

\newcommand{\FF}{\boldsymbol{D}}
\newcommand{\EE}{\boldsymbol{E}}
\newcommand{\UU}{\boldsymbol{U}}

    \newcommand{\RREV}{\color{purple}}
\newcommand{\BREV}{\color{purple}}

\newcommand{\RMA}{\color{black}}
\newcommand{\semi}{Semistable }
\newcommand{\SE}{\mathrm{SE}}

\DeclareSymbolFont{extraup}{U}{zavm}{m}{n}
\DeclareMathSymbol{\varheart}{\mathalpha}{extraup}{86}
\DeclareMathSymbol{\vardiamond}{\mathalpha}{extraup}{87}

\NiceMatrixOptions{cell-space-limits = 2pt}
\author{
  Giovanna Bonfanti%
  \thanksAAffil{DICATAM (Department of Civil, Environmental, Architectural Engineering and Mathematics), 
 Universit\`a degli studi di
  Brescia, via Branze 43, I--25133 Brescia, Italy.
Email: {\ttfamily giovanna.bonfanti@unibs.it}}, \ 
  Elisa Davoli%
  \thanksAAffil{ASC (Institute of Analysis and Scientific Computing), TU Wien, Wiedner Hauptstra\ss e 8-10, 1040 Vienna, Austria. Email: {\ttfamily elisa.davoli@tuwien.ac.at}}, \
  Riccarda Rossi%
  \thanksAAffil{DIMI (Department of Mechanical and Industrial Engineering), Universit\`a degli studi di
  Brescia, via Branze 38, I--25133 Brescia, Italy.
Email: {\ttfamily riccarda.rossi\,@\,unibs.it}}%
}

\begin{document}


\title{A coupled rate-dependent/rate-independent system for adhesive contact in Kirchhoff-Love plates}

\date{}
\maketitle
\begin{abstract}
 We perform a dimension reduction analysis for a coupled rate-dependent/rate-independent adhesive-contact model in the setting of visco-elastodynamic plates. We work with a   weak solvability notion inspired by the theory of (purely) rate-independent processes, and accordingly  term the related solutions   `Semistable Energetic'. For Semistable Energetic solutions, 
  the momentum balance holds in a variational sense, whereas the flow rule for the adhesion parameter is replaced by a semi-stability condition coupled with an energy-dissipation inequality.
 \par 
  Prior to addressing the   dimension reduction analysis, we show that Semistable Energetic
 solutions to the three-dimensional damped adhesive contact model converge,
as the viscosity term tends to zero, to three-dimensional Semistable Energetic solutions for
the undamped corresponding system.
\par
\RMA We then perform a dimension reduction analysis, both in the case of a vanishing viscosity tensor, and in the complementary setting in which the damping is assumed to go to infinity as the thickness of the plate tends to zero. In both regimes, the presence of  adhesive contact yields a nontrivial coupling of
 the
  in-plane and out-of-plane contributions.
  \par
  In the vanishing-viscosity case we additionally confine the analysis to the case in which  also inertia is neglected: in the vanishing-thickness limit we thus obtain purely rate-independent
  evolution for the adhesive contact phenomenon, still formulated in terms of the Semistable Energetic solution concept.
  In the second, 
    undamped scenario, inertia is instead encompassed, thus the limiting evolution retains a mixed rate-dependent/rate-independent character, 
     and is again given in terms of  an energy-dissipation inequality and a semistability condition. \EEE
\end{abstract}
%



\section{Introduction}

 The often intrinsic nonconvexity and nonlinearity of most three-dimensional models for inelastic phenomena  lead to \EEE notable hurdles for their numerical approximations and simulations. The variational identification of reduced lower dimensional models has thus thrived in the past thirty years as a 
   valuable \EEE
  modeling tool in continuum mechanics. In this paper we address a  dimension  reduction analysis for a model of adhesive contact
between two bodies, 
in the frame of  visco-elastodynamics.

The mathematical literature on dimension reduction is vast. Starting from the seminal papers \cite{acerbi-buttazzo-percivale, ledret-raoult, friesecke-james-muller} which sparked the effort towards an identification of static reduced models in nonlinear elasticity, limiting models have been deduced in a variety of settings, in static, quasistatic, and dynamic regimes, spanning from elastodynamics \cite{abels-mora-muller, abels-mora-muller2} and visco-elasticity \cite{friedrich-machill} to delamination
 (see \cite{FreParRouZan11, FreRouZan-2013}, as well as 
 \cite{MiRoTh10DDNE}), plasticity \cite{davoli-mora, DavoliMora15, Maggiani-Mora}, and crack propagation \cite{babadjian, freddi-paroni-zanini}. 
 \par
  The analysis in this paper  moves from \cite{FreParRouZan11}, in which the dimension reduction analysis was carried out for a \emph{purely} rate-independent model of delamination and, correspondingly, for  the related  \emph{Energetic} solutions \`a la \textsc{Mielke/Theil} \cite{MieThe04RIHM}. We have instead addressed an adhesive contact model in the setting of visco-elastodynamic plates; the related process has thus a  \emph{mixed}  rate-dependent/rate-independent character, calling for an appropriate weak solution notion. We have focused on how the presence of viscous effects in the momentum balance affects the properties of 
  reduced Semistable Energetic Solutions for the  thin plate model. The setup we have considered and our results are detailed below. 
\EEE

\subsection*{The model}
 We consider a model describing the mechanical evolution during a time interval $(0,T)$ of two viscoelastic bodies $\Omega_+$ and $\Omega_-$ in $\R^3$ that are in contact with adhesion along a prescribed surface portion $\GC$  (see Figure $1$ below for the special case in which $\Omega$ has a cylindrical geometry and the contact surface $\GC$ is vertical). \EEE

In its
\emph{classical formulation},  such evolution is governed by a \EEE momentum balance, with viscosity and inertia,
for the displacement  field $u:(0,T)\times ( \Omega_+ {\cup} \Omega_-)\to \R^3$,  
\EEE
namely
\begin{subequations}
\label{adh-con-intro}
\begin{equation}
\label{mom-balance-intro}
\varrho \ddot{u} - \mathrm{div} \left(\bbD e(\dot u) + \bbC e (u) \right)
= f \quad \text{ in } (0,T) \times ( \Omega_+ {\cup} \Omega_-),
\end{equation}
with \RMA $\varrho \geq 0  \EEE $ the mass density of the body,  $\bbD$ and $\bbC$ the viscosity and the elasticity tensors, \EEE $e(u):=\frac{1}{2}(\nabla u+\nabla u^\top)$ the linearized strain
tensor
and $f$  a  time-dependent applied volume force.  
Equation \eqref{mom-balance-intro} is  supplemented with  time-dependent Dirichlet
boundary conditions on the Dirichlet part  $\GD$  of the boundary $\partial\Omega$, where
$\Omega := \Omega_+ \cup \GC \cup \Omega_-$.  For simplicity we will assume that
the applied traction on the Neumann part $\GN = \partial\Omega \setminus \GD$ is null, namely
\begin{equation}
\label{bc-intro}
u=w \quad \text{  on } (0,T) \times \GD,  \qquad \qquad
\left(\bbD e(\dot u) + \bbC e (u) \right)|_{\GN} n= 0 \quad \text{  on } (0,T) \times \GN,
\end{equation}
with $n$  the outward unit normal to $\partial\Omega$.

 Following the approach by Fr\'emond, cf.\ \cite{Fre12} and the pioneering paper \cite{FreNed96DGDP}, the evolution of 
 adhesion between the two bodies is described in terms of an internal variable $z:(0,T)\times \GC \to [0,1]$  that is in fact \EEE  a surface-damage parameter, as it describes the fraction of fully effective molecular links in the bonding. Namely $z(t,x)=0$ (resp. $z(t,x)=1$) means that the bonding is completely broken (resp. fully intact) at the time $t\in (0,T)$, at the material point $x\in \GC$,  with $z(t,x)\in (0,1)$ for the intermediate states. \EEE
The evolutions of $u$  and 
of the adhesion parameter
$z$ \EEE are coupled through the following boundary
condition on the contact surface $\GC$xf
\begin{equation}
\label{cont-surf-intro}
  \left(\bbD e(\dot u) + \bbC e (u) \right)|_{\GC}  n +  \alpha_\lambda (\JUMP{u})\EEE + \kappa z \JUMP{u} =0
\quad \text{ on } (0,T) \times \GC,
\end{equation}
where, with a slight abuse of notation,
$n$ indicates here the unit normal to $\GC$ oriented from $\Omega_+$
to $\Omega_-$, $\kappa$ is a positive constant  and the symbol
$ \JUMP{u} := u_+ - u_-$ denotes the jump of $u$ across the interface $\GC$,  as $u_\pm$ is  the trace on $\GC$ of the restriction of $u$ to $\Omega_\pm$. 
In the boundary condition  \eqref{cont-surf-intro},  $\alpha_\lambda (\JUMP{u})+\kappa z \JUMP{u}$ represents the contact reaction and the term
 $\alpha_\lambda (\JUMP{u})$ (where the 
 Lipschitz continuous function
 $\alpha_\lambda: \R^3 \to \R^3 $ is the Yosida regularization of the  convex analysis subdifferential of the indicator function $I_{[0,+\infty)}$) penalizes 
the interpenetration between the two parts $\Omega^+$ and $\Omega^-$, yielding an approximation of the unilateral constraint $\JUMP{u}\cdot n\geq 0$ on $\GC$. In fact,
while the original model proposed by Fr\'emond \cite{Fre12} contains the impenetrability condition on the contact surface, in the present analysis, dealing with inertial terms in the momentum balance, we have chosen to keep an approximation of this constraint (cf.\ also Remark  \ref{rmk:few-comments}  below).
Moreover,
the contribution $ \kappa z \JUMP{u}$ in \eqref{cont-surf-intro}, due to adhesive contact, penalizes  
displacement jumps in points with strictly positive $z$ but does not exclude them. 
We observe that the blow-up of the coefficient $\kappa$ would lead to a different model with the {\it brittle constraint} $ z \JUMP{u}=0$ 
that allows for displacement jumps (i.e., $\JUMP{u}\ne 0$) only at points where the bonding is completely broken (i.e., $z=0$), and otherwise 
imposes the transmission condition $\JUMP{u}=0$ on the displacements.
\par
From the principle of virtual power, in which microscopic forces responsible for the degradation of the adhesive substance are included, the evolution of $z$ is ruled by
\EEE
\begin{equation}
\label{adh-flow-rule-intro}
\partial \mathrm{R}(\dot z) + \mathrm{b} \partial\calG(z)  +\partial I_{[0,1]}(z) -a_0 \ni - \tfrac{1}2  \kappa \left| \JUMP{u} \right|^2 \quad \text{ on } (0,T) \times \GC,
\end{equation}
 where   $\partial \mathrm{R}: \R \rightrightarrows \R $ is the subdifferential 
 of the $1$-homogeneous dissipation potential $\mathrm{R}$ defined as follows:
\begin{align}
\mathrm{R}(\dot z):=\left\{
\begin{array}{ll}
    a_1|\dot z|&\text{if } \dot z\leq 0\,,\\
\infty&\text{otherwise}\,,
\end{array}
\right.
\end{align}
with  $a_0$ and $a_1$  positive coefficients.  Indeed,  this choice for $\mathrm{R}$ imparts a rate-independent character to the flow rule for $z$. \EEE
By means of $\mathrm{R}$ we are encompassing the unidirectional evolution condition $\dot z\leq 0$, that is we are taking into account the irreversibility of the damage process on the contact surface. In \eqref{adh-flow-rule-intro}, 
$\partial I_{[0,1]}$ is the subdifferential of the indicator function of the interval $[0,1]$,  which forces $z$ to assume admissible values, and $\mathrm{b}$ is a \emph{nonnegative} coefficient  modulating the regularizing term $\partial\calG$. We emphasize that this regularization will be   active as soon as $\mathrm{b}>0$, but we will also address the case in which $\mathrm{b}=0$. More precisely,
 we will consider a 
 \emph{$\mathrm{BV}$-gradient} contribution along the footsteps of \cite{RosTho12ABDM},
 which tackled the analysis of a system modelling adhesion between two thermo-viscoelastic bodies and in particular addressed \EEE  the limit passage from \emph {adhesive contact} to \emph {brittle delamination}.  
  As in \cite{RosTho12ABDM}, in the definition of $\mathcal{G}$ we will encompass a strengthening of 
 the physical constraint on $z$ by further enforcing $z\in \{0,1\}$. In \cite{RosTho12ABDM},  such restriction brought \EEE  along some crucial analytical advantages in the limit passage procedure.  With the aim of extending the present investigation to the case of \emph {brittle delamination} models, we have   kept \EEE the regularizing term $\partial\calG$ into \eqref{adh-flow-rule-intro}, actually carrying out our analysis both in the case with, and without, such a regularization (see Remark  \ref{rmk:few-comments}   later on).
\par
Finally, we will supplement the above boundary-value problem with the initial conditions
\begin{equation}
\label{init-conds-intro}
u(0) = u_0 \text{ in }   \Omega, \qquad \dot{u}(0) = \dot{u}_0   \text{ in }  \Omega, \qquad z(0) = z_0  \text{ on }  \GC\,.
\end{equation}
\end{subequations}

\par
Due to the expected poor time regularity of the adhesion parameter $z$, the adhesive contact system \eqref{adh-con-intro} will be weakly formulated in a suitable way.
More precisely, we will resort to an Energetic-type solvability notion in which the momentum balance equation will be satisfied in a variational sense while a semi-stability condition, joint with an energy-dissipation inequality,  will hold as weak formulation of the flow rule for the adhesion parameter. 
 This solution concept is due to \textsc{T.\ Roubi\v{c}ek} \cite{Roub08,Roub10TRIP}, see also \cite{RosTho15CEx} from which we borrow the term `Semistable Energetic' for the associated solution curve, cf.\  Definition \ref{def:energetic-sol1} ahead. 

In fact, in the paper we will also work with an \emph{enhanced} version of  \semi Energetic  solutions, for which we will claim the validity of an energy-dissipation \emph{balance},  see
Definition \ref{def:energetic-sol2}.    We will term such solutions \emph{Balanced}   \semi Energetic solutions and  in fact obtain them as soon as the damping term in the momentum balance equation is present and yields additional  spatial regularity for $\dot u$. \EEE

 \subsection*{Our results}
  Our work  sparks from the  asymptotic analysis carried out in \cite{FreParRouZan11}, which we have extended to the case of visco-elasto-dynamics. 
\par
With this aim, preliminarily we have gained insight into the role of the damping term in the momentum balance  for the three-dimensional 
adhesive contact system. Namely, in the case of fixed positive thickness, we have carried out   an asymptotic analysis in the system, 
 as the coefficient of the damping term tends to zero.
 For the damped system we have at our disposal a result guaranteeing the existence of   \emph{Balanced} \semi Energetic solutions.
Now, 
 in \underline{\bf  Theorem \ref{thm:pass-lim-nu}} ahead  we have shown that in this asymptotic regime   \emph{Balanced} \semi Energetic solutions for the damped adhesive contact system \eqref{adh-con-intro} converge to a  \semi Energetic solutions for the  corresponding undamped system, in particular extending a previous existence result proven in \cite{RosThoBRI-INERTIA}. 
 As a matter of fact, the disappearance of the damping term brings about a loss of time regularity for the displacement $u$, which ultimately prevents us from obtaining  an energy balance.

 The influence of  damping is even more apparent in the   dimension reduction analysis for the   adhesive contact problem: because of the viscosity  and the inertial terms
in the momentum balance, the system acquires a mixed rate-dependent/rate-independent character which makes the asymptotic analysis significantly different from the purely rate-independent case considered in \cite{FreParRouZan11}. In this, the roles of viscosity and inertia are tightly related.
\par
In fact,  first of all we will address dimension reduction in  a regime in which the damping term   disappears in the vanishing-thickness limit. 
As we will see, 
in this first case the  inertial term as well needs to be neglected, already for positive thickness.   We will then prove
that, in the vanishing-viscosity limit, (Balanced) \semi Energetic solutions of the 3D adhesive contact system converge to \semi
Energetic solutions of the plate model, in which the displacement variable is in elastic equilibrium and the delamination parameter evolves rate-independently. Thus, the limiting system is purely rate-independent. 
\par
Secondly, we will tackle the vanishing thickness analysis in a 
regime that retains both the damping, and the inertial terms. In this case, we will again obtain convergence   to \semi
Energetic solutions for the limiting system, which preserves a mixed rate-dependent/rate-independent nature. 
\par
To delve into our results, let us  specify the geometry for the dimension reduction analysis. 
We consider  a thin, cylindrical plate $\Omega_\eps$ of height $\eps>0$ where the contact surface $\GC^\eps$ is positioned vertically. We refer to Figure $1$ below for a depiction of the geometry of $\Omega:=\Omega_1$ where all sub- and superscripts are, for simplicity, omitted.

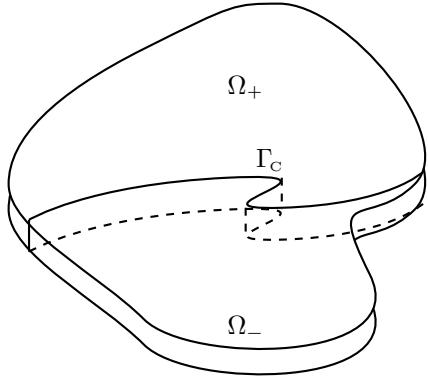
\begin{figure}[h!]
\centering
\begin{tikzpicture}[scale=0.8]
\label{fig:omega}
 \filldraw[thick,fill=white,scale=1.5]
	[xshift=4,yshift=-10] (-1.5,1)  .. controls (-1.5,1.5) and (-1,2) ..
	(0,2.5) .. controls (1,3) and (1,3) ..
	(1.5,3) .. controls (2.25,3) and (3.5,1.5) .. 
	(3,1) .. controls (2.5,0.5) and (2,1) .. 
	(2.5,0)  .. controls (3,-1) and (0.5,-1) ..
	(0,-0.5) .. controls (-0.5,0) and (-1.5,0.5) ..
	(-1.5,1);
	\filldraw[thick,fill=white,scale=1.5]
	[xshift=4,yshift=-2] (-1.5,1)  .. controls (-1.5,1.5) and (-1,2) ..
	(0,2.5) .. controls (1,3) and (1,3) ..
	(1.5,3) .. controls (2.25,3) and (3.5,1.5) .. 
	(3,1) .. controls (2.5,0.5) and (2,1) .. 
	(2.5,0)  .. controls (3,-1) and (0.5,-1) ..
	(0,-0.5) .. controls (-0.5,0) and (-1.5,0.5) ..
	(-1.5,1);
	\draw[black, thick] plot [smooth, tension=2, out=-20, in=-20] coordinates {(-1.7,0.8)(1.6,1.5)(2.5,1)(4.85,1.6)};
        \draw[black, thick] (-1.7,0.8)--(-1.7,0.27);
	\draw[black, thick, dashed] (2.5,1.5)--(2.5,0.97);
 \draw[black, thick, dashed] plot [smooth, tension=2, out=-20, in=-20] coordinates {(-1.7,0.27)(1.6,0.97)(2.5,0.47)(4.85,1.07)};
 \draw[black, thick, dashed] (1.9,1)--(1.9,0.55);
 \node at (1.9, 3) {$\Omega_+$};
 \node at (2.3, 1.8) {$\GC$};
 \node at (1.9, -1) {$\Omega_-$};
	\end{tikzpicture}
 \caption{The set $\Omega=\Omega_+\cup \GC\cup \Omega_-$}
\end{figure}

As  customary in dimension reduction problems, as a preliminary step, we will  perform a  suitable \EEE  rescaling of the variables. While our spatial rescaling will be consistent with that of linearly elastic dimension reduction problems \cite{friesecke-james-muller}, a further time-rescaling will be  needed to cope with possible very slow oscillations occurring in $\Omega_\eps$ and emerging in the limit as $\eps\to 0$. We refer to \cite{abels-mora-muller} and \cite{Maggiani-Mora} for an analogous space-time rescaling in the dynamically  elastic and perfectly plastic settings, respectively.

 In our first convergence result, \textbf{Theorem \ref{mainthm-1}}, 
  we will address the vanishing-thickness analysis for a 3D system featuring a \emph{quasistatic} momentum balance, with no inertial terms, and the viscosity tensor is multiplied by a (positive) coefficient that also vanishes. 
We will prove that the limiting plate model is described by an  elastic equilibrium equation for the displacements (i.e., no inertia and no damping), 
featuring a reduced elasticity tensor on the in-plane directions.
Due to the absence of any additional time-regularizations, in this setting we will solely  deduce an energy-dissipation \emph{inequality}. \EEE

In our second convergence result, see \textbf{Theorem \ref{mainthm-2}}, we will consider the complementary scenario in which  the damping coefficient is assumed to go to infinity as the thickness goes to zero,  inertia is retained, \EEE  and some standard scaling on the adhesive contact term is considered.  We will show that in this case the limiting system exhibits  a viscoelastic behaviour with an adhesive contact condition featuring both the in-plane and the out-of-plane contributions. Despite the additional compactness provided by the persistence of damping effects in the limit, the analysis  in this setup  is quite delicate for, a priori,  the limiting viscous and elastic bilinear forms in the weak momentum balance would depend on the planar strains of the displacement and test functions in a temporally nonlocal way involving an appropriate operator $\Mextname$,  cf.\ Sec.\ \ref{ss:6.1} ahead. 
We consider in our analysis the simplified scenario in which both the elasticity and viscosity tensors keep planar and out-of-plane effects decoupled,  thanks to  a suitable condition,
cf.\  \eqref{condX-1} ahead. \EEE
 Despite the enhanced temporal regularity enjoyed by the limiting displacement, neither in this case  we  have succeeded in proving an energy-dissipation \emph{balance} 
for the \semi Energetic solutions to the plate system, cf.\ Remark \ref{rmk:despite} later on.
\EEE


In both our dimension reduction results, as already observed in other inelastic dimension reduction settings (see, e.g.,\cite{DavoliMora15}), a full decoupling of the limiting in-plane and out-of-plane contributions into two different momentum balance equations for the in-plane and out-of-plane displacements is, in general, not  feasible, \EEE unless further assumptions on the adhesive contact are required. We briefly comment on this point in Remarks \ref{rk:mom-balance-pb1} and \ref{rk:mom-balance-pb2}.

\subsection*{Plan of the paper}
In Section \ref{s:2}  we specify our standing assumptions on the  3D  adhesive contact system, introduce Semistable Energetic (for short, $\SE$)  solutions and Balanced $ \SE$, and prove
our existence result,  Theorem \ref{thm:pass-lim-nu}, for the undamped system. Section   \ref{s:dim-red} sets the stage for the dimension reduction analysis, while in Section 
\ref{s:main-results} we state our two main results, Theorems  \ref{mainthm-1} \& \ref{mainthm-2}. After settling some preliminary results in Section \ref{s:prelim-ests}, we prove the former in Section \ref{s:proof-th1} and the latter in Section \ref{s:proof-th2}. 
\EEE

\medskip

\section{Semistable Energetic solutions for the damped and undamped 3D adhesive contact systems}
\label{s:2}
This section focuses on the 3D adhesive contact system, in that  we 
 introduce the weak solvability notion  of  Semistable Energetic solution,  we also define \emph{Balanced} $\SE$ solutions, and we 
state the existence of  (Balanced) $\SE$ solutions
for the damped 3D system in Theorem \ref{thm:RT}. 
From it, through an asymptotic procedure, we deduce the existence of $\SE$ solutions for the undamped system. 
Prior  to specifying  the setup for our analysis, let us settle some general notation.
\begin{notation}
\label{not:1.1}
\upshape
Let $X$ be a (separable) Banach space. By
$\pairing{}{X}{\cdot}{\cdot}$  we will  denote the duality between $X^*$ and $X$ or between $(X^n)^*$ and $X^n$
(whenever $X$ is a Hilbert space, $\pairing{}{X}{\cdot}{\cdot}$  will be the inner product),  while $\| \cdot \|_{X}$ stands for  
the norm in $X$ and, most often, in $X^n$ as well.
The inner Euclidean product in $\R^n$, $n\geq 1$, will be denoted by $\pairing{}{}{\cdot}{\cdot}$ and the Euclidean norm in $\R^n$  by $|\cdot| $.
\RCO We will also denote by $|A|: = \sqrt{\mathrm{tr}(A^\top A)}$ the Frobenius norm of a matrix $A\in \R^{n\times n}$, and by $:$ the associated scalar product, so that $A:B:= \mathrm{tr}(A^\top B)$. 
 Moreover, $\R_\sym^{n\times n}$ will denote the space of $n\times n$ symmetric real matrices.
\EEE
\par
We  will  write $\| \cdot \|_{L^p}$ for the $L^p$-norm on the space $L^p(O;\R^d)$, with $O$ a measurable subset of $\R^n$ and $1\leq p\leq \infty$,  and similarly $\| \cdot \|_{H^k}$ for the norm of the  space $H^k(O;\R^d)$.
 \par
Given a function $v:(0,T)\times \Omega\to\R$ differentiable,  w.r.t.\ time,  almost everywhere in  $(0,T)\times \Omega$, its (a.e.\ defined) partial time derivative 
will be  indicated by $\dot{v}:(0,T)\times \Omega\to\R$. 
A different notation will be employed when considering $v$ as a (Bochner) function  from $(0,T)$ with values in a Lebesgue or Sobolev space $X$  (with the Radon-Nikod\'ym property):
if $v\in\mathrm{AC} ([0,T];X)$,  then its  (almost everywhere defined)  time derivative is indicated by $v':(0,T) \to X$. 
\RCO Finally, let $X$ be, in addition, reflexive. Given  $(u_n)_n, u \in \romanC^0([0,T];X)$,
 whenever we 
write 
\[
u_n \to u \quad \text{in }  \romanC^0([0,T];X_{\mathrm{weak}})
\]
we will mean the following:  first of all, that $(u_n)_n \subset L^\infty (0,T;X)$ is bounded, hence there exists $R>0$ such that for every $n\in N$
  the image sets $u_n([0,T])$ are contained  in the closed ball $\overline{B}_R$. Let $\mathrm{d}_{\weak}$ be a distance inducing on $\overline{B}_R$
  the weak topology of $X$: convergence in $\romanC^0([0,T];X_{\mathrm{weak}})$ means convergence in $\romanC^0([0,T]; (X,d_{\mathrm{weak}}))$. \nc
\par
Last but not least, the symbols $c,\,c',\, C,\,C'$ will denote positive constants whose precise value may vary from line to line (or even  within the same   line).
We will sometimes employ the symbols $I_i$,  $i = 0, 1,... $,
 as place-holders for  terms appearing in inequalities: also in this case, such symbols may appear in different proofs with different meaning.   
\end{notation}
\paragraph{\bf Setup.}
Throughout the paper, the standing assumptions
 on the 3D domain $\Omega = \Omega_+\cup \GC \cup \Omega_- $, are
\begin{subequations}
\label{ass-domain}
\begin{eqnarray}
\label{ass-domain1}
\hspace*{-3em}
&&    
\text{$\Omega\subset\R^3$ is
bounded, $\Omegatwo,\, \Omegaone,\, \Omega$ are Lipschitz domains,
$\Omegaone \cap \Omegatwo=\emptyset$, $\GC=\barOmegaone\cap\barOmegatwo$}\,,
\\
\label{ass-domain2}
\hspace*{-3em}
&&
\partial \Omega = \overline{\GD}\cup \overline{\GN},
\text{ with $\GD,\,\GN$ open subsets in $\partial\Omega$,} \
\\
\label{ass-domain2+}
\hspace*{-3em}
&&\GD\cap \GN=\emptyset, \ \barGD\cap\barGC=\emptyset, \ 
\mathcal{H}^{2}(\GD\cap\barOmegaone)>0\,,\
\mathcal{H}^{2}(\GD\cap\barOmegatwo)>0\,, \
 \mathcal{H}^{2}(\GC) >0\,, \EEE
\end{eqnarray}
where $\mathcal{H}^{2}$ 
denotes 
the $2$-dimensional Hausdorff measure.  
\end{subequations}
For a given $u\in H^1(\Omega{{\setminus}}\GC;\R^3)$, the symbol
\[
\JUMP{u} := u_+ - u_-, \quad \text{with } u_\pm \text{ the trace on $\GC$ of the restriction of $u$ to $\Omega_\pm$},
\]
will denote the \emph{jump of $u$ across the interface $\GC$}. 
 In what follows, we will use the notation
\[
H^1_{\GD}(\Omega{{\setminus}}\GC;\R^3):= \{ v \in H^1(\Omega{\setminus}\GC;\R^3)\, : \ v=0\text{ a.e.\ on }\GD\}
\] 
(where we again  have omitted the $\GD$-trace
  operator) and simply write $\| \cdot \|_{H^1}$ when no confusion may arise. 
   Moreover, we will denote by  $\langle \cdot,\cdot \rangle_{H^1_{\GD}(\Omega{\setminus}\GC)}$    the duality pairing between $H^1_{\GD}(\Omega{{\setminus}}\GC;\R^3)^*$ and $ H^1_{\GD}(\Omega{{\setminus}}\GC;\R^3)$. \EEE
\par
Throughout the paper we 
shall assume that $\varrho$ is a given positive constant.
We will also suppose that 
 the elasticity tensor $\mathbb{C}$ 
fulfills
\begin{subequations}
\label{assdata} 
\begin{equation}
\label{assCD} 
\begin{split}
&\mathbb{C}\in\R^{3\times 3\times 3\times 3} 
\text{ is  symmetric, i.e. }  \mathbb{C}_{ijk\ell} =  \mathbb{C}_{k\ell ij} = \mathbb{C}_{jik\ell}  =  \mathbb{C}_{ij\ell k}  \quad i,j,k,\ell \in \{1,2,3\} \EEE
 \EEE
\\
&\text{$\mathbb{C}$ is  positive definite, i.e., } \exists\,C_\mathbb{C}^1,C_\mathbb{C}^2>0 
 \ \forall\,
\eta\in\R_{\sym}^{3\times 3}:\; \  \EEE
C_\mathbb{C}^1|\eta|^2\leq\eta:\mathbb{C}\eta\leq C_\mathbb{C}^2|\eta|^2\,.
\end{split} 
\end{equation}
For the damped system, we will also consider
\begin{equation}
\label{viscosityTensor}
\text{
 a viscosity tensor $\bbD \in \R^{3\times 3\times 3\times 3}, 
$ symmetric and positive definite in the sense of \eqref{assCD}.}
\end{equation}
\par
 Finally, we will assume that the volume force $f$ and the 
  Dirichlet loading $w$ fulfill
\begin{align}
\label{assRegf}
&
f \in W^{1,1}(0,T;L^2(\Omega;\R^3))
\\
&
\label{hyp-w}
w\in  W^{2,1}(0,T;H^1(\Omega;\R^3)) \cap  W^{3,1}(0,T;L^2(\Omega;\R^3)) \,,
\end{align}
and that the system is supplemented with initial data\
\begin{equation}
\label{initial-data}
u_0 \in H^1_{\GD}(\Omega{{\setminus}}\GC;\R^3), \qquad \dot{u}_0 \in L^2(\Omega;\R^3), \qquad z_0
 \in 
 \begin{cases}
 L^\infty(\GC;[0,1]), \EEE & 
 \\
 \SBV(\GC;\{0,1\}) & \text{if } \mathrm{b}>0\,. 
 \end{cases}
\end{equation}
\end{subequations}
\begin{remark}
\label{rmk:symmetry tensor}
\upshape We point out that condition \eqref{assCD} is equivalent to the classical symmetry conditions required also in \cite[(2.4)]{FreParRouZan11}. In other words, all four equalities in \eqref{assCD} are either equal to--or can be directly deduced from--the properties in \cite[(2.4)]{FreParRouZan11}.
\end{remark}\EEE
\begin{remark}[Square and square root of fourth order tensors]
\label{rmk:sqrt}
\upshape
Given $A,B\in \R_{\sym}^{3\times 3}$ and  tensors  $  \FF, \,  \EE  \in \R^{3\times 3\times 3\times 3}$, positive definite and  fulfilling
the symmetry condition \eqref{assCD}, 
 recall that
\begin{equation}
\FF A{:}B=\sum_{i,j,k,l}  B_{ij}\FF_{ijkl}A_{kl}\quad\text{and}\quad
\FF \, \EE=\Big(\sum_{m,n}\FF_{ijmn}\EE_{mnlk}\Big)_{i,j,k,l } \,. 
\end{equation}
By symmetry we also observe that
\begin{equation}
\label{abssquare}
\begin{split}
|\FF A|^2&=\sum_{i,j}\Big(\sum_{k,l}\FF_{ijkl}A_{kl}\Big)^2
=\sum_{i,j}\sum_{k,l}\FF_{ijkl}A_{kl}\sum_{m,n}\FF_{ijmn}A_{mn}
=\sum_{k,l,m,n}A_{kl}A_{mn}\sum_{i,j}\FF_{klij}\FF_{ijmn}\\
&=\FF^2A{:}A\,.
\end{split}
\end{equation}
Let now $C_{\FF}^1, \, C_{\FF}^2>0$ fulfill:
\[
 \forall\, A\in  R_{\sym}^{3\times 3}  \colon \quad
C_{\FF}^1 |A|^2 \le \FF A{:} A\le C_{\FF}^2 |A|^2\,.
\]
Then, we have 
\begin{equation}
\label{squaresCD}
 \forall\, A\in  R_{\sym}^{3\times 3}  \colon \quad
(C_{\FF}^1)^2 |A|^2 \le \FF^2 A{:} A\le (C_{\FF}^2)^2 |A|^2\,. 
\end{equation}
To find these relations for the constants, we may argue as follows:
Being a linear mapping on $\R^{3 \times 3}$,
we can fix a notation to rewrite any tensor $A\in\R^{3 \times 3}$ as a vector of $3^2$ components and
$\FF$  as a  $\R^{3^2\times 3^2}$ matrix,
 symmetric and positive definite.
Exploiting the spectral decomposition of this matrix we see that the constants
 $C_{\FF}^1,C_{\FF}^2$ 
are bounds for the smallest, resp.\ largest, eigenvalues of
the $\R^{3^2\times 3^2}$ matrices corresponding to $\FF$.  Then,  \eqref{squaresCD} follows taking into account \eqref{abssquare}.
\par
In a similar manner, exploiting the symmetry and (uniform) positive definiteness  as well as the spectral decomposition of the corresponding
$\R^{3^2\times 3^2}$-matrix,  we may conclude the existence of the \emph{square root} of
$\FF$, i.e.,  there is
\begin{subequations}
\begin{equation}
\label{ass-U1}
\UU \in  \R^{3\times 3\times 3\times 3}
\text{ symmetric and positive definite, s.t.\ }\;\FF=\UU^2\,.
\end{equation}
By symmetry, with calculations similar to those performed in \eqref{abssquare} we thus have
\begin{equation}
\label{4heateq}
\FF A{:}A=\UU A {:}\UU A\,.
\end{equation}
In addition, we may check that 
\begin{equation}
\label{ass-U2}
\forall\, A\in  R_{\sym}^{3\times 3}  \colon \quad
\sqrt{C_{\FF}^1}|A|^2\leq\UU A{:}A
\leq\sqrt{C_{\FF}^2}|A|^2\,.
\end{equation}
\end{subequations}
The existence of a square root for the
positive definite, symmetric fourth order tensor $\FF$
is found again by exploiting the spectral properties of the corresponding $\R^{3^2\times 3^2}$ matrix.
After diagonal transform, for this matrix the entries of its square root matrix are
found by taking the square
root of the eigenvalues. This also yields \eqref{ass-U2}, since, as already mentioned,
the constants
$C_{\FF}^1,C_{\FF}^2$  are bounds for the smallest, resp.\ largest, eigenvalue of
$\FF$.
\end{remark} \EEE
\begin{remark}
\label{rmk:data-f-w}
\upshape
While the existence result for the (damped) adhesive contact system from \cite{RosThoBRI-INERTIA} applies to 
the case in which also a surface traction force $g $ is applied
to  the Neumann part of the boundary,
here we will confine the discussion to the case in which only a volume force is applied. This restriction is in view of the dimensional reduction analysis,
since
a spatial rescaling of  $g $ would involve 
additional technical difficulties.
\par
The time regularity of $f$ and $w$  ensures 
that the partial time derivative of the driving
 energy functional $\calE$ from  \eqref{bulk-surface-contribution} is well defined and satisfies  estimate  \eqref{Gronwall-estimate} below. We could weaken 
  conditions \eqref{assRegf} and \eqref{hyp-w} 	
   if we rewrote  the terms involving the power $\partial_t \calE$ of the external forces in a suitable way, \RCO cf.\ Remark \ref{rmk:suitable-way} ahead. \EEE
\end{remark}
%
\subsection{$\SE$ solutions  for the damped   adhesive contact system}
\label{ss:2.1}
%
Prior to recalling the definition of  $\SE$ solution for the (damped) adhesive contact system
 in the 3D domain,  let us settle its energetics. 
\RCO We mention in advance that, for simplicity, in what follows we will work with a \emph{constant} mass density $\varrho$ and likewise 
we will  not encompass a dependence of the tensors $\bbC$ and $\bbD$ on the spatial variable. \EEE
\paragraph{Dissipation potentials and driving energy   functional for the damped system.} 
The evolution of the  adhesive contact system  in the damped case is 
 governed by the following 
  kinetic energy $\calK$,
 viscous dissipation potential $\calV,$ and  $1$-homogeneous
dissipation $\calR$:
\begin{align}
\label{defEkin}
&\calK: L^2 (\Omega;\R^3) \to[0,\infty)\,, &&  \calK(\dot u):=\int_\Omega\tfrac{\varrho}2 |\dot u|^2\,\mathrm{d}x
\,,\\
\label{defV}
&\calV:\Spu\to[0,\infty)\,, &&
\calV(\dot u):=\int_{\Omega{\setminus}\GC}\tfrac{1}{2}\mathbb{D} e(\dot u): e(\dot u)\,\mathrm{d}x\,,  \\
\label{defRk}
&\calR:\Spz\to[0,\infty]\,, && 
\calR(\dot z):=\int_{\GC} \mathrm{R}(\dot z)\,\mathrm{d}\Surf (x) \quad  \text{ with } 
\mathrm{R}(\dot z):=\left\{
\begin{array}{ll}
a_1|\dot z|&\text{if }\dot z\leq 0\,,\\
\infty&\text{otherwise\EEE}\,. 
\end{array}
\right.
\end{align}
\RMA Hereafter, $\varrho\geq 0$ will be a fixed constant, modulating the presence of inertia  in the momentum balance.  \EEE
The driving energy functional  $\calE:[0,T]\times\Spu\times L^1(\GC)\to \R\cup\{\infty\}$ 
 is given by 
\begin{subequations}
\label{energy-definition}
\begin{equation}
\label{bulk-surface-contribution}
\begin{aligned}
  \calE(t,u,z):=
\calE_\bulk (t,u) + \calE_\surf(u,z)\,.
\end{aligned}
\end{equation}
The \emph{bulk} contribution  is given by 
\begin{align}
\label{defEk}
&
\calE_\bulk (t,u): =  \int_{\Omega{\setminus}\GC}\tfrac{1}{2}\mathbb{C}e(u):e(u) \dd x  - \pairing{}{H^1_{\GD}(\Omega{\setminus}\GC)}{\Bigf(t)}{u},
\intertext{where the function  $\Bigf: [0,T] \to \Spu^*$ encompasses the volume force 
and the contributions involving the time-dependent Dirichlet loading $w$, namely}
\label{BigF}
&
\langle \Bigf(t),u\rangle_{H^1(\Omega{\setminus}\GC)}  : = \int_\Omega f(t) u \dd x 
-\int_{\Omega{\setminus}\GC} \bbC e(w(t)): e(u) \dd x -\int_{\Omega{\setminus}\GC} \bbD e(\dot{w}(t)): e(u) \dd x 
-\int_\Omega \varrho \ddot{w}(t) u \dd x\,,
\intertext{while the surface contribution consists of  }
\label{calG-z}
&
 \calE_\surf(u,z): =  \calH(u) +  \calJ (u,z)  +\int_{\GC}\left( I_{[0,1]}(z)
{-} a_0 z \right)  
\,\mathrm{d}\Surf(x)+\mathrm{b} \calG(z) \qquad \text{with } \mathrm{b} \geq 0\,.
\end{align}
In \eqref{calG-z}, the term $\calH$ features the Yosida approximation
 $\widehat{\alpha}_\lambda$ 
 of the indicator function  of 
\begin{equation}
\label{energy-H}
\begin{aligned}
&
\text{the cone } K = \{ v\in \R^3\, : v \cdot n \geq 0 \}, \text{ i.e.} 
\\
&
\calH(u)  : = \int_{\GC} \widehat{\alpha}_\lambda(\JUMP{u}) \,\mathrm{d}\Surf(x) \qquad \text{with }
\widehat{\alpha}_\lambda(v): = \frac1\lambda \mathrm{dist}^2(v, K)
\end{aligned}
\end{equation}
\RCO for some parameter $\lambda>0$ that will be kept \emph{fixed} in what follows. \EEE
What is more,  the coupling term $\calJ$ accounts for the `adhesive contact energy'
\begin{equation}
\label{JKfunc}
\calJ (u,z)   : = \int_{\GC} \tfrac{\kappa}2 z Q(\JUMP{u}) \dd \Surf(x)  \qquad \text{with } Q(v) : = |v|^2\,.
\end{equation}
Finally, the  regularizing contribution $\calG$, which is active as soon as $\mathrm{b}>0$,  forces $z \in \SBV(\GC;\{0,1\})$ \RCO in that it is $\infty$ for $z\in L^1(\GC){\setminus}  \SBV(\GC;\{0,1\})$. Hence, \EEE
the approximate jump set
$\mathrm{J}_z$	  of $z$ is well defined (cf.\ \cite[Def.\ 3.67]{AmFuPa05FBVF}) and  $\calG$ records its $1$-dimensional Hausdorff measure, namely it is defined by 
\begin{equation}
\label{energy-calG}
  \calG(z) = \begin{cases}
\mathcal{H}^1(\mathrm{J}_z) & \text{if } z \in \SBV(\GC;\{0,1\}),
\\
+\infty & \text{otherwise.}
\end{cases}
\end{equation}
\end{subequations}
\begin{remark}
\label{rmk:few-comments}
\upshape
A few comments on $ \calE_\surf$ are in order:
\begin{enumerate}
\item
The contribution $\calH$ penalizes the distance of $\JUMP{u}$ from the cone $K$, i.e.\ the failure of the 
non-interpenetration constraint $\JUMP{u} \cdot n \geq 0$. We will not enforce such constraint (which would correspond to replacing $\widehat{\alpha}_\lambda$
by the indicator function $\widehat{\alpha} = I_{K}$ in the definition of $\calH$), due to the presence of inertia 
in the displacement equation. It is indeed well known that the simultaneous presence of inertial terms and unilateral constraints in the momentum equation poses significant analytical difficulties. They can be dealt with by either formulating 
 the momentum equation  in terms of a variational inequality  (cf. \cite{CocouSR} and \cite{Cocou}),    \EEE
 or by adopting the novel approach proposed in \cite{ScaSchi17}. 
 \par
  On the one hand,  also in view of the forthcoming dimensional-reduction analysis,   we prefer to work with a stronger formulation than that 
provided by a variational inequality.  On the other hand, we believe that the techniques developed in \cite{ScaSchi17} could be adapted to the present case as well, cf.\ Remark 
\ref{rmk:Scala_Schimpi} ahead.
\item
The regularizing term $\calG$ can be interpreted as a `$\BV$ proxy' of the more standard gradient regularizations often featuring in adhesive contact models. It was proposed in \cite{RosTho12ABDM}, where the limit passage from adhesive contact  to \emph{brittle delamination} was addressed for a system modelling adhesive contact between two visco-elastic bodies further subject to thermal fluctuations. This limiting procedure corresponds to letting $\kappa\to\infty$ in \eqref{JKfunc}, which  formally leads to 
$zQ(\JUMP{u}) =0$, namely the \emph{brittle constraint}
\[
z\JUMP{u} =0 \qquad \Surf\text{-a.e.\ in } \GC.
\] 
By such constraint, the transmission condition $\JUMP u=0$ is imposed where adhesion is active, i.e.\ $z>0$. 
Now, $\calG$ strengthens the physical constraint $z\in [0,1]$ by further enforcing $z\in \{0,1\}$: 
 this makes the brittle delamination model akin to a model for Griffith fracture, with $z$ the characteristic function of a  (finite-perimeter) set $Z$ which may be understood as a  complementary of the crack set. 
 From an analytical viewpoint,  in \cite{RosTho12ABDM}  the contribution 
 \begin{equation}
 \label{explanation-perimeter}
\RCO \mathrm{b}   \EEE \calG(z) =  \RCO \mathrm{b}   \EEE   \mathcal{H}^1(\mathrm{J}_z) =   \RCO \mathrm{b}   \EEE  P(Z,\GC)
 \end{equation}
 (where the latter term is the perimeter of the set $Z$ in $\GC$) played  a key role in the construction of recovery test functions for the limit passage, as $\kappa\to \infty$, in the momentum balance. 
 \par
  In this paper we will not address the dimensional reduction analysis for the
brittle delamination model, but only focus on the adhesive contact system. Hence, 
 for the  upcoming results we would not need to encompass
 the regularizing contribution 
  $\calG$ into the surface energy $\calE_\surf$.
  Nonetheless,  in a forthcoming article  we plan to extend our asymptotic analysis to the brittle delamination system. Therein, we will significant rely on the  regularization via perimeter provided by the term $\calG$. That is why, for future reference   
  we will allow for the presence of this contribution to $\calE_\surf$, often distinguishing the cases $\mathrm{b}>0$ and $\mathrm{b}=0$. 
\end{enumerate}
\end{remark}
%
%
\par
 We are now in a position to  precisely state the  two concepts of \semi Energetic 
solution   we will work with throughout the paper. 
Both notions of semistable solutions
 encompass  the weakly formulated momentum balance and 
a semi-stability condition that provides a
  weak formulation of the flow rule
for the adhesion parameter. Moreover, 
for \semi Energetic solutions we will just require an energy-dissipation inequality. Instead, for \emph{Balanced} \semi Energetic solutions, we will claim the validity of 
an energy-dissipation balance. 
In Def.\  \ref{def:energetic-sol1}  (resp.\ Def.\  \ref{def:energetic-sol2}) \EEE below we will state the momentum balance explicitly in the context of  the damped adhesive contact system, while,  for later reference, 
we prefer to give the semi-stability inequality and the  energy-dissipation inequality (resp.\ energy-dissipation balance) \EEE  in their general form, as in  \cite[Def.\ 3.1]{RosTho15CEx}. 
\RMA The statement of the regularity properties for the displacement variable reflects the fact that, also in view of the forthcoming 
dimension reduction analysis, we encompass both the case in which inertial terms are present ($\varrho>0$) and that of a quasistatic momentum balance ($\varrho=0$). \EEE 
\begin{definition}[Semistable  Energetic solution]
\label{def:energetic-sol1}
We call a pair $(u,z) : [0,T] \to \Spu \times \Spz$  a \emph{\semi Energetic} ($\SE$) solution 
of  the  damped adhesive contact system
 $(\mathcal{K}, \mathcal{V}, \mathcal{R}, \mathcal{E})$ \EEE
 if
\begin{subequations}
\label{regularity}
\begin{align}
&
\label{reg-u}
u \in 
  H^1(0,T;\Spu)\,, \quad  \RMA \varrho \dot u \EEE  \in L^\infty (0,T; \Spw)\EEE\,,
 \quad  \RMA \varrho \ddot u\EEE \in L^2(0,T;\Spu^*)\,,
\\
&
\label{reg-z}
z \in  
 L^\infty((0,T){\times} \GC; [0,1]) \EEE \cap   \mathrm{BV}([0,T];\Spz), \qquad z \in L^\infty (0,T; \SBV(\GC; \{0,1\})) \text{ if } \mathrm{b}>0; 
\end{align}
\end{subequations}
fulfill 
\begin{compactitem}
\item[-] the weak momentum balance
\begin{equation} \label{weak-mom}
\begin{aligned}
&
 \langle  \varrho \ddot{u}(t), v \rangle_{H^1(\Omega{\setminus}\GC)} \EEE
+ \int_{\Omega{\setminus}\GC} \left\{ \mathbb{D}e(\dot{u}(t)) {:} e(v) 
{+} \mathbb{C}e(u(t)):e(v) \right\} \,\mathrm{d}x
\\
& \quad 
+\int_{\GC}\alpha_\lambda(\JUMP{u(t)})\cdot \JUMP{v} \dd \Surf(x)
+\int_{\GC} \kappa z\JUMP{u(t)}\JUMP{v}\,\mathrm{d}\Surf (x)
\\
&
= \int_\Omega f(t) v \dd x 
-\int_{\Omega\setminus \GC} \bbC e(w(t)) : e(v) \dd x - \int_{\Omega\setminus \GC} \bbD e(\dot w(t)) : e(v) \dd x
- \int_\Omega \varrho \ddot{w}(t) v \dd x 
\end{aligned}
\end{equation}
for every $v \in H^1_{\GD}(\Omega{{\setminus}}\GC;\R^3)$ and for almost all $t\in (0,T)$, with $\alpha_\lambda$ the G\^ateaux  derivative of $\widehat \alpha_\lambda$,
\item[-] the  semistability condition 
\begin{equation}
\label{semistab-z} 
\calE(t,u(t),z(t)) \leq \calE(t,u(t),\tilde z) +
\calR(\tilde z {-}z(t)) \qquad \text{for all } \tilde z \in \Spz \
\text{and for all } t \in [0,T],\EEE
 \end{equation}
\item[-]  the energy-dissipation inequality 
 \begin{equation}
 \label{enineq}
 \begin{aligned}
\calK(\dot u(t))& 
+  \int_0^t 2 \calV(\dot u(s))  \dd s
+ \Var_{\calR}(z, [0,t])+
 \calE(t,u(t),z(t)) 
 \\
&  \leq  \calK(\dot u(0))
+  \calE(0,u(0),z(0)) + \int_0^t \partial_t\calE(s,u(s),z(s))\,\mathrm{d}s 
\qquad \text{for all } t \in [0,T]\,,
 \end{aligned}
 \end{equation}
 with 
$\Var_{\calR}$ the 
total variation  induced by 
$\calR$, i.e.,  for a given 
 subinterval $[s,t]\subset [0,T],$
\[
\Var_{\dissr}(z; [s,t]) := \sup\left\{ \sum_{j=1}^{N}
\dissr(z(r_j) - z(r_{j-1}))\, : \quad s= r_0<r_1<\ldots<r_{N-1}<r_N=t
\right\}\,. 
\] 
\end{compactitem}
\end{definition}
\begin{definition}[Balanced Semistable  Energetic solution]
\label{def:energetic-sol2}
We call a \semi Energetic solution $(u,z)$ \emph{Balanced} if, in addition, it complies with the  
energy-dissipation \emph{balance} 
 \begin{equation}
 \label{endissbal}
 \begin{aligned}
\calK(\dot u(t))& 
+  \int_0^t 2 \calV(\dot u(s))  \dd s
+ \Var_{\calR}(z, [0,t])+
 \calE(t,u(t),z(t)) 
 \\
& =  \calK(\dot u(0))
+  \calE(0,u(0),z(0)) + \int_0^t \partial_t\calE(s,u(s),z(s))\,\mathrm{d}s 
\qquad \text{for all } t \in [0,T]\,.
 \end{aligned}
 \end{equation}
\end{definition}
\EEE
\begin{remark}[Time-dependent Dirichlet conditions]
\upshape
Observe that, 
for a solution $u$ of \eqref{weak-mom} the function
\begin{equation}
\label{tilde-u}
\tilde{u} : [0,T]\to H^1(\Omega{\setminus}\GC;\R^3) \qquad \tilde{u}(t): = u(t) + w(t)
\end{equation}
solves the momentum equation \eqref{mom-balance-intro} and fulfills the time-dependent Dirichlet condition in  \eqref{bc-intro}. 
\end{remark}
\begin{remark}[Reformulation of the semistability condition]
\upshape
Since the bulk contribution to $\calE$ only depends on the variable $u$ (cf.\ \eqref{defEk}), which is kept fixed in the semistability condition, 
\RCO also in view of \eqref{explanation-perimeter} \EEE
inequality \eqref{semistab-z}  reduces to 
\begin{equation} \label{reduced-semistab}
\begin{aligned}
&
 \int_{\GC}  \tfrac \kappa 2   z(t)|\JUMP{u}(t)|^2 \dd \Surf(x)   -\int_{\GC} a_0 \, z(t) \dd \Surf (x) \EEE + \mathrm{b} P(Z(t),\GC)\\
 &\qquad  \leq  \int_{\GC}   \tfrac \kappa 2  \widetilde{z}|\JUMP{u}(t)|^2 \dd \Surf(x)   -\int_{\GC} a_0 \widetilde{z} \,  \dd \Surf (x) \EEE
+\mathrm{b} P(\widetilde{Z},\GC)\
+\int_{\GC} a_1 |\widetilde{z}{-}z(t)| \dd \Surf (x) \\ & \text{for all } \widetilde{z} \in
\RCO L^1(\GC)  \ 
 (\widetilde{z} \in
 \SBV(\GC;\{0,1\}) \text{ if } \mathrm{b}>0), \EEE  \text{ with } 0\leq \widetilde{z}\leq z(t)  \ \aein\, \GC,  \text{  for every } t \in [0,T].
\end{aligned}
\end{equation}
In \eqref{reduced-semistab}, $\widetilde{Z}$ is a finite-perimeter set such that 
$  \mathrm{b}  \calG(\widetilde{z}) =   \mathrm{b}   P(\widetilde{Z},\GC)$, cf.\ Remark \ref{rmk:few-comments}. 
\end{remark}
\begin{remark}[`Explicit' energy-dissipation balance]
\upshape
For later use we record here that
\[
\begin{aligned}
\partial_t \calE(t,u,z)  & =  -\pairing{}{H^1(\Omega{\setminus}\GC)}{\dot{\Bigf}(t)}{u} 
\\
 &  =
-\int_\Omega \dot{f}(t) u \dd x 
+\int_{\Omega{\setminus}\GC} \bbC e(\dot{w}(t)): e(u) \dd x +\int_{\Omega{\setminus}\GC} \bbD e(\ddot{w}(t)): e(u) \dd x 
+\int_\Omega \varrho  \dddot{w}(t) u \dd x
\end{aligned}
\]
Therefore, 
also taking into account that  $ \Var_{\calR}(z;[0,t]) = \int_{\GC} a_1 (z(0){-}z(t)) \dd \Surf(x)$, 
the energy-dissipation balance
\eqref{endissbal} 
 translates into
 \begin{equation}
 \label{enineq-explicit}
 \begin{aligned}
& \int_\Omega \tfrac{\varrho}2 |\dot u(t)|^2 \dd x 
+\int_0^t \int_{\Omega{\setminus}\GC}\mathbb{D} e(\dot u): e(\dot u)\,\mathrm{d}x \dd s 
+ \int_{\GC}a_1 (z(0){-}z(t)) \dd \Surf(x)
+
 \calE(t,u(t),z(t)) 
 \\
& =  \int_\Omega \tfrac{\varrho}2 |\dot u(0)|^2 \dd x 
+  \calE(0,u(0),z(0)) 
\\
& \quad 
- \int_0^t  \int_\Omega \dot{f}(s) u(s) \dd x 
\dd s 
+\int_0^t \int_{\Omega{\setminus}\GC} \bbC e(\dot{w}(s)): e(u(s)) \dd x  \dd s 
\\
& \quad +
\int_0^t \int_{\Omega{\setminus}\GC} \bbD e(\ddot{w}(s)): e(u(s)) \dd x  \dd s
+ \int_0^t \int_\Omega \varrho \dddot{w}(s) u(s)  \dd x \dd s 
\qquad \text{for all } t \in [0,T]\,. 
 \end{aligned}
 \end{equation}
\end{remark}
\RCO
\begin{remark}[Rewriting the work of the external forces]
 \label{rmk:suitable-way} 
 \upshape
 Relying on the time regularity of $u$ from \eqref{reg-u}, it would be possible to 
 rewrite the work of the  external forces
 $\mathrm{Work}([0,t]):= \int_0^t \partial_t \calE(s,u(s),z(s)) \dd s$ 
 as 
 \begin{equation}
 \label{work-rewritten}
 \begin{aligned}
 &
 \mathrm{Work}([0,t])  = 
  \int_0^t  \int_{\Omega} f \dot u\dd x  \dd s+   \int_\Omega f(0) u(0) \dd x  -    \int_\Omega f(t) u(t) \dd x \EEE 
 \\
 & \quad 
-\int_0^t \int_{\Omega{\setminus}\GC} \bbC e(w): e(\dot u) \dd x \dd s 
+\int_{\Omega{\setminus}\GC} \bbC e(w(t)) : e(u(t)) \dd x - 
\int_{\Omega{\setminus}\GC} \bbC e(w(0)) : e(u(0)) \dd x 
\\ & \quad 
-\int_0^t \int_{\Omega{\setminus}\GC} \bbD e(\dot{w}): e(\dot u) \dd x  \dd s 
+ \int_{\Omega{\setminus}\GC} \bbD e(\dot{w}(t)): e(u(t)) \dd x 
- \int_{\Omega{\setminus}\GC} \bbD e(\dot{w}(0)): e(u(0)) \dd x 
\\ & \quad 
- \int_0^t \int_\Omega  \varrho \ddot{w}\dot{u} \dd x \dd s + \int_\Omega \varrho  \ddot{w}(t) u(t) \dd x - 
 \int_\Omega  \varrho \ddot{w}(0) u(0) \dd x \,.
\end{aligned}
 \end{equation}
 Inserting this in \eqref{enineq} would lead to cancellations with some of the terms 
 in $\mathcal{E}(0,u(0),z(0)) - \mathcal{E}(t,u(t),z(t)). $ Clearly, with this reformulation,
  the dimensional reduction analysis that we are going to perform in 
 Sec.\ \ref{s:dim-red} would involve scaling conditions on $w$ alternative to those assumed in Hypothesis \ref{hyp:data} ahead. 
 \par
 Nonetheless, in what follows, we will stick to the formulation \eqref{enineq} of the energy-dissipation balance for easier reference to the general theory of rate-independent systems. 
\end{remark}
\EEE

The existence of Balanced $\SE$ solutions for the damped adhesive contact system in the 3D body $\Omega$ was proved in
\cite{RosThoBRI-INERTIA} (in fact, therein the case of null Dirichlet boundary conditions was considered, but the result
could be easily extended to the case of time-dependent conditions).
\begin{theorem}{\cite[Thm.\ 2.1]{RosThoBRI-INERTIA}}
\label{thm:RT}
Assume \RCO  \eqref{ass-domain}
 and \EEE  \eqref{assdata}. 
\RMA Let $\varrho\geq 0$. \EEE 
Then,  
there exist  Balanced $\SE$   solutions to the damped adhesive contact system 
satisfying the initial conditions
\begin{equation}
\label{init-conds}
u(0) = u_0 \ \aein \  \Omega, \qquad \dot{u}(0) = \dot{u}_0 \ \aein \ \Omega, \qquad z(0) = z_0  \ \aein \ \GC\,,
\end{equation}
 where the initial data $({u}_0 ,\dot{u}_0 ,z_0)$ fulfill  the semistability condition \eqref{semistab-z} at  $t=0$, i.e.
\begin{equation}
\label{init-conds2}
\calE(0,u_0,z_0) \leq \calE(0,u_0,\tilde z) +
\calR(\tilde z {-}z_0) \qquad \text{for all } \tilde z \in \Spz \,.
\end{equation}
\end{theorem}
\begin{remark}[Encompassing non-interpenetration]
\label{rmk:Scala_Schimpi}
\upshape
In \cite{ScaSchi17}, the authors proved  the existence of global-in-time solutions to the Cauchy problem for an adhesive contact system with inertia and the non-interpenetration  constraint on the displacement variable. Their analysis relies on 
a novel formulation of the momentum balance, originally proposed in \cite{BoRoScaSchi17}, for which \emph{time-dependent}  test functions are chosen in a suitable ``parabolic'' space $\mathscr{V}$, consisting of the intersection of  Bochner-Sobolev spaces. 
In that setting,  the unilateral constraint  is rendered by means of a cyclically monotone operator with graph in $\mathscr{V}\times\mathscr{V}^*$. This weak formulation  of the momentum balance can be in fact re-stated in terms of  a variational inequality. Nonetheless,  it  allows for a clear identification, in the displacement equation, of the reaction force due to the  non-interpenetration constraint in terms of the abovementioned maximal monotone operator.
\par
We believe that this approach could be adopted in the present setup, as well. 
Nevertheless, we have chosen not to detail it to avoid overburdening the exposition of the dimensional reduction analysis. 
\end{remark}
\subsection{From the damped to the undamped system}
We now address the limit passage in the notion of $\SE$  solutions when the  viscosity tensor
is of the form 
\begin{equation}
\label{viscosity-tensor-nu}
\bbD=  \bbD_\mu = \mu \overline{\bbD} \text{ with } \mu \downarrow 0 \text{ and } \overline{\bbD} \text{ a fixed viscosity tensor as in \eqref{viscosityTensor}}.
\end{equation}
Accordingly, we will denote by $\calV_{\mu}$ the associated dissipation potentials, and by $(u_\mu, z_\mu)_{\mu}$ a family of solutions to the damped adhesive contact system.
 We have the following result; we stress that it is valid both for the case with the perimeter regularization  (i.e., with $\mathrm{b}>0$), and without. 
  In the undamped limit, we will only be able to obtain an energy-dissipation inequality, see  Remark \ref{rmk:failure-ENID}  ahead. 
  \RMA Our result encompasses both the case in which inertia is present in the momentum balance ($\varrho>0$), and the quasistatic case $\varrho=0$. \EEE
\begin{theorem}
\label{thm:pass-lim-nu}
Let $\mathrm{b}\geq 0$ be fixed. 
Let $({\mu}_j)_j$ be a null sequence
and, correspondingly, let  $(u_{{\mu}_j}, z_{{\mu}_j})_j$ be a sequence of  Balanced   $\SE$ solutions to the adhesive contact systems
 $(\mathcal{K}, \mathcal{V}_{{\mu}_j}, \mathcal{R}, \mathcal{E})$,  with viscosity tensors $(\bbD_{{\mu}_j})_j$ as in \eqref{viscosity-tensor-nu}, and supplemented with initial data $(u_0,\dot{u}_0, z_0)$ as in \eqref{initial-data}  and fulfilling \eqref{init-conds2}. \EEE Then, 
there exist a (not relabeled) subsequence and a pair
$(u,z) $, with   
\[
\begin{cases}
\displaystyle u\in  L^\infty(0,T;\Spu), 
\\
\displaystyle \RMA \varrho u \EEE  \in 
 W^{1,\infty} (0,T; L^2(\Omega;\R^3)) \cap  H^{2} \EEE (0,T;\Spu^*),
 \end{cases}
\]
and $z$ as in \eqref{reg-z}, 
such that 
\begin{enumerate}
\item
the following convergences hold
\begin{subequations}
\label{convs-nuj-uz}
\begin{align}
&
u_{{\mu}_j} \weaksto u && \text{in } L^\infty(0,T;\Spu),
\\
& \RMA \varrho u_{{\mu}_j} \weaksto \varrho u  \EEE
&& \text{in } 
 W^{1,\infty} (0,T; L^2(\Omega;\R^3)) \cap H^2(0,T;\Spu^*),
\\
&
z_{{\mu}_j}(t) \weaksto z(t) && \text{in } 
\begin{cases}
L^\infty(\GC)  & 
\\
\SBV(\GC;\{0,1\}) \subset L^\infty(\GC)  & \text{if } \mathrm{b}>0
\end{cases} \quad 
\text{ for all } t \in [0,T];
\end{align}
\end{subequations}
\item $(u,z)$
is  a $\SE$ solution of the undamped system for adhesive contact, namely it fulfills
\begin{itemize}
\item[-] the momentum balance equation for almost all $t\in (0,T)$ and for every $v \in \Spu$ 
\begin{equation} 
\label{weak-mom-undamped}
\begin{aligned}
&
 \langle \varrho \ddot{u}(t), v \rangle_{H^1(\Omega{\setminus}\GC)} 
+ \int_{\Omega{\setminus}\GC}  \mathbb{C}e(u(t)):e(v)  \,\mathrm{d}x
+\int_{\GC}\alpha_\lambda(\JUMP{u})\cdot \JUMP{v} \dd \Surf(x)
+\int_{\GC} kz\JUMP{u}\JUMP{v}\,\mathrm{d}\Surf (x)
\\
&
=  \int_\Omega f(t) v \dd x  -\int_{\Omega\setminus \GC} \bbC e(w(t)) : e(v) \dd x 
- \int_\Omega\varrho  \ddot{w}(t) v \dd x; 
\end{aligned}
\end{equation}
\item[-] the semistability condition \eqref{semistab-z} (equivalently, \eqref{reduced-semistab});
\item[-] the energy-dissipation inequality for all $t\in [0,T]$ 
 \begin{equation}
 \label{enineq-undamped}
 \begin{aligned}
 &
\calK(\dot u(t)) 
+\Var_{\calR}(z, [0,t])+
 \calE(t,u(t),z(t))
 \\
 &
\leq \calK(\dot u(0))
+  \calE(0,u(0),z(0))
\\
&  \quad  -\int_0^t \int_\Omega \dot{f}(s) u(s) \dd x \dd s
+\int_0^t \int_{\Omega{\setminus}\GC} \bbC e(\dot{w}(s)): e(u(s)) \dd x  \dd s
+\int_0^t \int_\Omega\varrho  \dddot{w}(s) u (s) \dd x \dd s  \,.
 \end{aligned}
 \end{equation}
 \end{itemize}
\end{enumerate}
\end{theorem}
\par In order to prove   Thm.\ \ref{thm:pass-lim-nu}  we will first of all derive a series of a priori estimates on the sequences $(u_{{\mu}_j}, z_{{\mu}_j})_j$, relying on the following coercivity and power-control estimates for $\calE$. 
\begin{lemma}
\label{l:properties-en}
Assume \eqref{assdata}.
Then,
\begin{subequations}
\label{est-lemma}
\begin{align}
&
\label{est-coer-E}
\begin{aligned}
& 
\exists\, c_0,\, C_0>0 \ \forall\, (t,u,z) \in [0,T]\times \Spu \times \Spz\, :  
\\
 & \qquad \qquad 
  \calE(t,u,z) \geq c_0 \left(\|u\|_{\Spvr}^2 {+} \| z \|_{L^\infty(\GC)} +\mathrm{b}  \|z\|_{\SBV(\GC)} \right) -C_0,
  \end{aligned}
\\
&
\label{Gronwall-estimate}
\begin{aligned}
&
 \exists\, L_0 \in L^1(0,T)   \ \forall\, (u,z) \in \Spu \times \Spz\, {\text {and for a.a. }} t\in (0,T): 
 \\
 &   |\partial_t \calE (t,u,z)| \leq |L_0(t)|  \left( \calE(t,u,z){+}1\right)\,.
 \end{aligned}
\end{align}
\end{subequations}
\end{lemma}
\begin{proof}
It follows from  \eqref{assRegf}  and \eqref{hyp-w} that 
\[
\exists\, C_F>0 \ \forall\, (t,u) \in [0,T]\times \Spu \, : \quad |\pairing{}{\Spvr}{\Bigf(t)}{u} | \leq C_F \|u \|_{\Spvr}\,.
\]
Therefore, 
by the positive-definiteness of $\bbC$, Korn's inequality, and the fact that, if  $\calE(t,u,z)<\infty$ then $z\in [0,1]$  a.e.\ in $\GC$,  we find that 
\[
 \calE(t,u,z) \geq c \|u\|_{\Spvr}^2 - C_F \|u \|_{\Spvr}+ \| z \|_{L^\infty(\GC)} +\mathrm{b}  \|z\|_{\SBV(\GC)} -a_0|\GC|  -1 -\mathrm{b}, \EEE
\]
whence we immediately deduce \eqref{est-coer-E}.
\par
Then, \eqref{Gronwall-estimate} follows upon taking into account that  for almost all $t\in (0,T)$ and for every \EEE $(u,z) \in \Spu \times \Spz$
\[
|\partial_t \calE(t,u,z)| \leq C\left(\|\dot{f}(t)\|_{L^2(\Omega)}{+} \|\dot{w}(t)\|_{H^1(\Omega)}{+} \|\ddot{w}(t)\|_{H^1(\Omega)}  {+} \|\dddot{w}(t)\|_{L^2(\Omega)}\right)
\|u \|_{\Spvr}\,.
\] 
\end{proof}
We are now in a position to derive a series of a priori estimates on the solutions  $(u_{{\mu}_j}, z_{{\mu}_j})_j$.
\begin{proposition}
\label{prop:2.5}
Assume \eqref{assdata}. Then,
there exists a constant $C>0$ such that  the following estimates hold for all $j\in \N$
 \begin{subequations}
 \label{est-nu}
 \begin{align}
 &
 \label{est-nu-u-1}
 \| \RMA \varrho \EEE  \dot{u}_{{\mu}_j}\|_{L^\infty(0,T;L^2(\Omega;\R^3))} +\mu_j^{1/2} \|e(\dot{u}_{{\mu}_j})\|_{L^2(0,T;L^2(\Omega;\R^{3\times3}))} +
 \|u_{{\mu}_j}\|_{L^\infty (0,T; \Spv)} \leq C,
 \\
 &
  \label{est-nu-u-2}
 \| \RMA \varrho \EEE \ddot{u}_{{\mu}_j}\|_{L^2(0,T;\Spv^*)} \leq C,
 \\
 & 
  \label{est-nu-z-1}
 \|z_{{\mu}_j}\|_{L^\infty((0,T)\times \GC) \cap \BV([0,T];L^1(\GC))} \leq C,
 \\
 & 
  \label{est-nu-z-2}
 \|z_{{\mu}_j}\|_{L^\infty(0,T;\SBV(\GC))} \leq C \qquad\text{if } \mathrm{b}>0.
 \end{align}
 \end{subequations}
\end{proposition}
\begin{proof}
From the energy \emph{balance} \eqref{endissbal}, also taking into account \eqref{initial-data}, we derive that
\[
 \calE(t,u_{{\mu}_j}(t),z_{{\mu}_j}(t))
\leq C + \int_0^t \partial_t\calE(s,u_{{\mu}_j}(s),z_{{\mu}_j}(s))\,\mathrm{d}s  \leq C + \int_0^t  |L_0(s)|  \left( \calE(s,u_{{\mu}_j}(s),z_{{\mu}_j}(s)){+}1\right) \dd s \,.
\]
Then, via the Gronwall Lemma we obtain that 
\begin{equation}
\label{energy-bound-nu}
\sup_{t\in [0,T]} |\calE(t,u_{{\mu}_j}(t),z_{{\mu}_j}(t))|\leq C.
\end{equation}
On account of \eqref{est-coer-E},
we then infer the estimate for  $\|u_{{\mu}_j}\|_{L^\infty (0,T; \Spv)} $ as well as estimates \eqref{est-nu-z-1} \& \eqref{est-nu-z-2}
(indeed, since $z_{{\mu}_j}(\cdot,x)$ is non-increasing,  $\|z_{{\mu}_j}\|_{\BV([0,T];L^1(\GC))} = \|z_{{\mu}_j}(T) {-}z_{{\mu}_j}(0)\|_{L^1(\GC)}$).  
Combining \eqref{energy-bound-nu} and  \eqref{Gronwall-estimate} we infer that 
\[
\left|   \int_0^T \partial_t\calE(t,u_{{\mu}_j}(t),z_{{\mu}_j}(t))\,\mathrm{d}t   \right|\leq C,
\]
hence \eqref{endissbal} yields that
\[
\sup_{t\in (0,T)}\|\RMA \varrho \EEE  \dot{u}_{{\mu}_j}(t)\|_{L^2(\Omega)}^2  
+ \int_0^T   2\calV_{\mu_j}(\dot{u}_{\mu_j}(t)) \EEE \dd t \leq C,
\]
whence the first two bounds in \eqref{est-nu-u-1},  in view of \eqref{viscosity-tensor-nu}.\EEE
\par
 Finally, \eqref{est-nu-u-2} follows from a comparison in the momentum balance \eqref{weak-mom}, 
taking into account 
the previously obtained \eqref{est-nu-u-1} and \eqref{est-nu-z-1}, as well as 
\eqref{assRegf} and \eqref{hyp-w}.
\end{proof}
\noindent
We can now carry out the \underline{\textbf{proof of Theorem \ref{thm:pass-lim-nu}}}. \RMA To avoid overburdening the 
exposition, from now on we shall suppose that inertia contributes to the momentum balance. In the case without inertia, $\rho=0$, the proof can be adapted by suitably modifying the 
compactness arguments for the displacement variable: without entering into details, we may only mention that, after taking the limit in the momentum balance, the weak convergence in $ L^\infty(0,T; \Spv) $ improves to a strong one, with an argument similar to that in  Step $3$ of the proof of  Theorem \ref{mainthm-1}. \EEE
\par
In the case $\varrho>0$, the proof will be 
split in the following steps:
\begin{itemize}
\item[\textbf{Step $0$: compactness.}]
Resorting to the  
compactness results from 
\cite{Simon87}   as well as to Helly's theorem,
from estimates \eqref{est-nu} we gather that there exist $u :[0,T] \to \Spv $ and $z: [0,T]
\to L^\infty(\GC)$ such that, up to a not relabeled subsequence, the following convergences hold as $j\to \infty$
\begin{subequations}
\label{convs-nu}
\begin{align}
&
\label{convs-nu-1}
u_{{\mu}_j} \weaksto u && \text{in } L^\infty(0,T; \Spv) \cap W^{1,\infty} (0,T;L^2(\Omega;\R^3)) \cap H^2(0,T;\Spv^*),
\\
&
\label{convs-nu-3}
{\mu}_j  e(\dot{u}_{{\mu}_j}) \to 0 && \text{in } {L^2(0,T;L^2(\Omega;\R^{3\times3}))},
\\
&
\label{convs-nu-5}
  z_{{\mu}_j}(t) \weaksto z(t) && \text{in } L^\infty(\GC) \  \text{ (and in  $\SBV(\GC;\{0,1\})$ if $\mathrm{b}>0$)}   \text{ for all } t \in [0,T],
\\
 &
\label{convs-nu-4}
z_{{\mu}_j} \weaksto z && \text{in } L^\infty(0,T;\SBV(\GC;\{0,1\})) \quad \text{if } \mathrm{b}>0,
\\
& 
\label{convs-nu-6}
 z_{{\mu}_j}(t) \to z(t) && \text{in } L^q(\GC) \text{ for every } 1\leq q <\infty \text{ and  for all } t \in [0,T] \quad \text{if } \mathrm{b}>0,
\end{align}
where \eqref{convs-nu-6} also follows from the compact embedding $\SBV(\GC;\{0,1\})\Subset L^q(\GC)$ for all $1\leq q<\infty$.
\RCO In addition,
thanks to, e.g., the Ascoli-Arzel\`a type compactness result from \cite[Prop.\ 3.3.1]{AGS08},
 we have
\begin{align}
&
\label{convs-nu-2}
u_{{\mu}_j} \to u && \text{in } \mathrm{C}^0([0,T];  \Spv_{\mathrm{weak}}), 
\\
&
\label{convs-nu-2-bis}
\dot{u}_{{\mu}_j} \to \dot{u} && \text{in } \mathrm{C}^0([0,T];  L^2(\Omega;\R^3)_{\mathrm{weak}})
\end{align}
(where the above convergences have to be understood in the sense specified in Notation \ref{not:1.1}). \nc
\end{subequations}
\item[\textbf{Step $1$: limit passage in the momentum balance.}]
From \eqref{convs-nu-2}
and well-known trace theorems
 we gather, in particular, that 
$\JUMP{u_{{\mu}_j}} \to \JUMP{u}$ in   $\mathrm{C}^0([0,T]; L^{\rho}(\GC))$ for every $1\leq \rho<4$. Hence, by  the Lipschitz continuity of $\alpha_\lambda$
we find that 
\[
\alpha_\lambda (\JUMP{u_{{\mu}_j}}) \to \alpha_\lambda(\JUMP{u}) \text{ in   $\mathrm{C}^0([0,T]; L^{\rho}(\GC))$ for every $1\leq \rho<4$.}
\]
Combining this with convergences \eqref{convs-nu} we can pass to the limit in the weak formulation 
\eqref{weak-mom} of the momentum equation, integrated on any time interval $[s,t]\subset [0,T]$, thus obtaining the integrated version of \eqref{weak-mom-undamped}. By the arbitrariness of the time-interval, we thus conclude the limiting momentum balance at almost every $t\in (0,T)$. 
\par
A comparison in the momentum balance shows that, indeed,  $\ddot{u} \in L^\infty(0,T;\Spu^*)$. 
\item[\textbf{Step $2$: limit passage in the semistability condition.}]
  We will first discuss this in the case $\mathrm{b}>0$. 
In order to prove \eqref{reduced-semistab}
 at any fixed time $t\in (0,T]$,
 following a well-established procedure \RCO (cf.\ \cite{MRS06}), \EEE it is sufficient to exhibit for every $\widetilde{z} \in \SBV(\GC;\{0,1\})$ 
  (with associated finite-perimeter set $\widetilde{Z}$), \EEE 
 a `mutual recovery sequence' 
 $(\widetilde{z}_{j})_{j}$ fulfilling 
 \begin{equation}
 \label{recovery-seq-constr}
 \begin{aligned}
 &
\limsup_{j\to\infty} \Big(  \int_{\GC} \tfrac \kappa 2   (  \widetilde{z}_j{-} z_{{\mu}_j}(t)) |\JUMP{u_{{\mu}_j}(t)}|^2 \dd \Surf(x) + \mathrm{b} \left( P(\widetilde{Z}_j,\GC){-} 
P(Z_{{\mu}_j}(t),\GC)\right)
\\
& \qquad \qquad 
+\int_{\GC}  (a_0{+}a_1) \EEE  |\widetilde{z}_j{-}z_{{\mu}_j}(t)| \dd \Surf(x) \Big)\\
&
 \leq  \int_{\GC}  \tfrac \kappa 2 (  \widetilde{z}{-} z(t)) |\JUMP{u(t)}|^2 \dd \Surf(x) + \mathrm{b} \left( P(\widetilde{Z},\GC){-} P(Z(t),\GC)\right)
+\int_{\GC}  (a_0{+}a_1) \EEE |\widetilde{z}{-}z(t)| \dd \Surf(x)\,.
\end{aligned}
\end{equation}
We borrow the construction of $(\widetilde{z}_j)_j$ from 
\cite{RosTho12ABDM}
and set 
\begin{equation}
\label{MRS}
\widetilde{z}_j: = \widetilde{z} \chi_{A_j} + z_{{\mu}_j}(t) \EEE (1{-}\chi_{A_j}) \qquad \text{with } A_j: = \{x \in \GC \, : \ 0 \leq \widetilde{z}(x) \leq z_{{\mu}_j}(t,x)\}\,.
\end{equation}
It can be easily checked that $\widetilde{z}_j \in [0,1]$ with 
$\widetilde{z}_j \leq z_{{\mu}_j}(t)$ a.e. on $\GC$. The arguments from \cite[Sec.\ 5.2]{RosTho12ABDM} in fact  show that, in the case 
$\mathrm{b}>0$, since   $ z_{{\mu}_j}(t) \EEE \in  \SBV(\GC;\{0,1\})$ for all $j \in \N$ and the test function $\widetilde{z}$ is also picked in  $\SBV(\GC;\{0,1\})$, then
$\widetilde{z}_j \in \SBV(\GC;\{0,1\})$ as well. Furthermore, in \cite{RosTho12ABDM} it has been checked that $\widetilde{z}_j \weakto \widetilde z$
(strongly, if $\mathrm{b}>0$)
 in $L^q(\GC)$ 
for every $1\leq q<\infty$. Hence,
also taking into account convergences \eqref{convs-nu}, one finds that 
\[
\begin{aligned}
&
 \lim_{j\to\infty} \int_{\GC}  (a_0{+}a_1) \EEE |\widetilde{z}_j{-}z_{{\mu}_j}(t)| \dd \Surf(x)=  \lim_{j\to\infty} \int_{\GC}  (a_0{+}a_1) \EEE(z_{{\mu}_j}(t){-}\widetilde{z}_j ) \dd \Surf(x) \\
 &\qquad= \int_{\GC}  (a_0{+}a_1) \EEE  (z(t){-}\widetilde{z}) \dd \Surf(x) = 
  \int_{\GC}  (a_0{+}a_1) \EEE |\widetilde{z}{-}z(t)| \dd \Surf(x), 
\\
&
\lim_{j\to\infty}  \int_{\GC}   \tfrac \kappa 2 ( \widetilde{z}_j{-} z_{{\mu}_j}(t)) |\JUMP{u_{{\mu}_j}(t)}|^2 \dd \Surf(x)  = 
   \int_{\GC} \tfrac \kappa 2  ( \widetilde{z}{-} z(t)) |\JUMP{u(t)}|^2 \dd \Surf(x) \,.
  \end{aligned}
  \]
For the limit passage in the  perimeter term $\mathrm{b} \left( P(\widetilde{Z}_j,\GC){-} 
P(Z_{{\mu}_j}(t),\GC)\right)$ 
we refer to the arguments from the proof of   \cite[Prop.\ 5.9]{RosTho12ABDM}. 
\par
 A direct computation shows that the above construction of the  recovery  sequence  
$(\widetilde{z}_j )_j$
also works for the case $\mathrm{b}=0$. 
\item[\textbf{Step $3$: limit passage in the energy-dissipation inequality.}]
It follows from convergences  \eqref{convs-nu} that for every $t\in [0,T]$
\[
\begin{aligned}
& \|\dot u(t)\|_{L^2(\Omega)}^2 \leq  \liminf_{j\to\infty} \|\dot{u}_{{\mu}_j}(t)\|_{L^2(\Omega)}^2, 
\\
&  \Var_{\calR}(z_{{\mu}_j};[0,t]) = \int_{\GC}a_1  (z_0{-}z_{{\mu}_j}(t)) \dd \Surf (x) \longrightarrow 
\int_{\GC}a_1 (z_0{-}z(t)) \dd \Surf (x) = 
 \Var_{\calR}(z;[0,t]),
 \\
 & 
\calE(t,u(t),z(t))  \leq \liminf_{j\to\infty}\calE(t,u_{{\mu}_j}(t),z_{{\mu}_j}(t)),  
\\
&
\lim_{j\to\infty} \int_0^t \partial_t\calE(s,u_{{\mu}_j}(s),z_{{\mu}_j}(s))\,\mathrm{d}s  =  \\
&
-\int_0^t \int_\Omega \dot{f}(s) u(s) \dd x \dd s
+\int_0^t \int_{\Omega{\setminus}\GC} \bbC e(\dot{w}(s)): e(u(s)) \dd x  \dd s
+\int_0^t \int_\Omega \varrho  \dddot{w}(s) u (s) \dd x \dd s \,.
\EEE
\end{aligned}
\]
Therefore, passing to the limit in the energy-dissipation balance \eqref{endissbal} we obtain the energy-dissipation inequality 
\eqref{enineq-undamped}. This concludes the proof.
\end{itemize}
\begin{remark}[Missing Energy-Dissipation balance]
\label{rmk:failure-ENID} 
\upshape A standard procedure 
for proving  the validity of the  lower energy-dissipation inequality (namely,  the  converse  of  \eqref{enineq-undamped}  in the present context)
for $\SE$ solutions 
 adapts the well known `Riemann-sum technique' for obtaining the lower energy estimate for fully rate-independent systems; we refer, e.g., to the arguments in
\cite{Roub10TRIP},  as well as the proof of \cite[Thm.\ 5.1.2]{MieRouBOOK}, for \emph{general} coupled rate-dependent/rate-independent systems. 
\par
In the present situation, one would argue in the very same way as in the proof of \cite[Prop.\ 4.7]{RosTho12ABDM} and deduce, from the semistability condition, via a Riemann-sum type argument, the following estimate
\[
\begin{aligned}
& 
 \int_0^t  \int_{\GC}  \kappa  z(s) \JUMP{\dot{u}(s)} \JUMP{u(s)} \dd \Surf(x)  \dd s 
 \\
 & 
  \leq   \int_{\GC}   \tfrac \kappa 2 z(t) |\JUMP{u(t)}|^2 \dd \Surf(x)  - 
\int_{\GC}   \tfrac \kappa 2 z(0) |\JUMP{u(0)}|^2 \dd \Surf(x) 
+ \mathrm{b}  P(Z(t),\GC)  - \mathrm{b}  P(Z(0),\GC) \\
&\quad+\int_{\GC}  (a_0{+}a_1) \EEE |z(t){-}z(0)| \dd \Surf(x)
\end{aligned}
\]
for every $t\in (0,T)$. 
One would then combine the above inequality 
  with the limiting momentum balance equation tested by the rate $\dot u$ of 
 the limiting
displacement (see Step 5 in the proof of Theorem \ref{mainthm-2}). 
However, \EEE observe that this test is not admissible in the present context:  in fact, due to the lack of the damping term, we no longer possess the information that
$\dot u \in H^1(\Omega{\setminus}\GC;\R^3)$. Therefore, we cannot resort to the previously described method, and proving
the energy-dissipation balance remains an open problem. 
\par
Seemingly, this is not just a technical issue, as the validity of the energy balance appears to be tightly related to the validity of a chain-rule formula
for the functional $u\mapsto \calE(t,u,z)$ which, in turn, cannot be proved here,   again due to the lack of spatial regularity for $\dot u$.
\end{remark}
\section{The  dimension \EEE reduction}
\label{s:dim-red}
In this section, we specify the framework in which we will perform the dimension reduction analysis
for the \emph{damped} adhesive contact system. 
After settling  the geometric  and energetic setup, we will proceed to  formulate the 3D rescaled problems by means of suitable space and time rescalings, and state our convergence results, Theorem \ref{mainthm-1} and \ref{mainthm-2} below, in which we will  provide the existence of \semi Energetic solutions for two Kirchoff-Love plate models.

\paragraph{\bf Geometric setup.} 
To avoid overburdening the exposition of the dimension reduction analysis, we particularize the geometry of the 3D adhesive contact problem \eqref{adh-con-intro} to a {\it cylindrical} case where the contact surface $\GC$ is positioned vertically. More precisely, 
following \cite{FreParRouZan11} we consider a bounded open Lipschitz subset $\omega$ of $\R^2$ such that 
\[
\omega = \op \cup \gC \cup \om
\]
where $\omega_\pm$ are two disjoint open connected Lipschitz subsets with a non-empty simply connected common boundary $\gC$. We will denote by $(\gD)_\pm\subset\partial\omega_\pm$ the part of the boundary of $\omega_\pm$, on which a time-dependent Dirichlet boundary condition will be imposed. We will assume that 
\[
\calH^1((\gD)_\pm)>0, \qquad \overline{\gC} \cap \overline{(\gD)_\pm} = \emptyset. 
\] 
We set $\gD:= (\gD)_+ \cup (\gD)_-$, and we denote by 
\[
\RCO \Omega^\eps_\pm: = \omega_\pm \times \left(-\frac{\eps}2, \frac{\eps}2 \right), \qquad \GC^\eps: = \gC \times \left(-\frac{\eps  }2, \frac{\eps  }2 \right), \qquad \GD^\eps = \gD \times \left(-\frac{\eps }2, \frac{\eps  }2 \right).
\]
 \EEE

\noindent Let
\[
\Omega^\eps: = \Omega^\eps_+ \cup \GC^\eps \cup \Omega^\eps_-. 
\]
\paragraph{\bf Energetic setup for the adhesive systems.}   We mention in advance that we will address 
the dimension reduction analysis  for a family of adhesive contact systems 
$(\mathcal{K}_\eps, \mathcal{V}_\eps, \mathcal{R}_\eps, \mathcal{E}_\eps)_\eps$ for which all  the material coefficients and  constitutive functions, 
\emph{with the exception of  the elasticity
tensor}, depend on
 the thickness parameter \EEE
 $\eps$. More precisely, for every $\eps>0$ we \EEE
 consider the  damped system for adhesive contact between the 3D bodies 
$\Omega^\eps_+$ and $\Omega^\eps_-$, 
with 
dissipation potentials\
\begin{align}
\label{Veps}
&
\calV_\eps : H_{\Dir^\eps}^1(\Omega^\eps{\setminus}\GC^\eps;\R^3) \to [0,\infty), \qquad \calV_\eps (\dot u): = \int_{\Omega^\eps{\setminus}\GC^\eps}\tfrac12 \bbD_\eps e(\dot u) {:} e(\dot u) \dd x\,,
\\
\label{Reps}
& 
\calR_\eps:L^1(\GC^\eps)\to[0,\infty]\,,\;
\calR_\eps(\dot z):=\int_{\GC^\eps} \mathrm{R}_\eps(\dot z)\,\mathrm{d}\Surf(x)\,,\;
\mathrm{R}_\eps(\dot z):=\left\{
\begin{array}{ll}
a_1^\eps|\dot z|&\text{if }\dot z\leq 0\,,\\
\infty&\text{otherwise\EEE},
\end{array}
\right.
\intertext{and kinetic energy}
&
\calK_\eps: L^2(\Omega^\eps;\R^3) \to [0,\infty), \qquad \calK_\eps(\dot u): = \int_{\Omega^\eps} \tfrac {\varrho_\eps}2 |\dot u|^2 \dd x \,.
\end{align}
\RCO We will specify our conditions on the families $(\varrho_\eps)_\eps$, $(\bbD_\eps)_\eps$, and $(a_1^\eps)_\eps$,  for the different convergence results, \EEE
in Section 
\ref{s:main-results} ahead. \EEE
 The system is 
supplemented with  
volume forces 
 \begin{subequations}
\label{forces-half-line}
\begin{align}
\label{forces-half-line-1}
&
(f_\eps)_\eps \subset W_{\mathrm{loc}}^{1,1}(0,\infty;L^2(\Omega^\eps;\R^3)),
\intertext{and time-dependent Dirichlet loadings}
&
\label{forces-half-line-2}
(w_\eps)_\eps \subset W_{\mathrm{loc}}^{2,1}(0,\infty;H^1(\Omega^\eps;\R^3)) \cap  W_{\mathrm{loc}}^{3,1}(0, \infty;L^2(\Omega^\eps;\R^3))
\end{align} 
 \end{subequations} \EEE
 (cf.\ 
 \eqref{assRegf}--\eqref{hyp-w}),
 and a family $(u_{0,\eps}, \dot{u}_{0,\eps}, z_{0,\eps})_\eps$ of initial data as in \eqref{initial-data}  and fulfilling
 the semi-stability condition 
  \eqref{init-conds2},
 with $\calR_\eps$ and $\calE_\eps$ given by \eqref{Reps} and \eqref{defEk-eps} below.
  \EEE
  Accordingly, we introduce the functionals
   $\Bigf_\eps:   [0,+\infty)  \to {H^1_{\GD^\eps}(\Omega^\eps{\setminus}\GC^\eps;\R^3)}^*$ \EEE defined by 
 \[
\begin{aligned}
 \langle \Bigf_\eps(t),u\rangle_{ H^1_{\GD^\eps}(\Omega^\eps{\setminus}\GC^\eps;\R^3) }  : = &  
  \int_{\Omega^\eps} {f_\eps(t)}{u}  \dd x 
-\int_{\Omega^\eps{\setminus}\GC^\eps} \bbC e(w_\eps \EEE(t)){:}  e(u) \dd x 
\\
 & -\int_{\Omega^\eps{\setminus}\GC^\eps} \bbD_\eps e({\dot{w}_\eps \EEE}(t)){:} e(u) \dd x 
-\int_{\Omega^\eps} \varrho_\eps  \ddot{w}_\eps(t) \EEE u \dd x\,.
\end{aligned}
\]
The system for adhesive contact is then driven by the energy functional $\calE^\eps:  [0,+\infty)  \times H^1_{\GD^\eps}(\Omega^\eps{\setminus}\GC^\eps;\R^3)  \times L^1(\GC^\eps) 
\to\R\cup\{\infty\}$ defined for all $t\in [0,+\infty)$ by 
\[
\calE^\eps (t,u,z) = \calE_\bulk^\eps (t,u) + \calE_\surf^\eps(u,z),
\]
with 
\begin{equation}
\label{defEk-eps}
\begin{aligned}
&
\calE_\bulk^\eps (t,u): =  \int_{\Omega^\eps{\setminus}\GC^\eps}\tfrac{1}{2}\mathbb{C}e(u):e(u) \dd x -  
\langle \Bigf_\eps(t),u\rangle_{ H^1_{\GD^\eps}(\Omega^\eps{\setminus}\GC^\eps;\R^3) }
\\
&
 \calE_\surf^\eps(u,z): =  \RCO \calH^\eps(u)   +  \RCO \calJ^\eps (u,z)  \EEE +\int_{\GC^\eps}\left( I_{[0,1]}(z)
{-} a_0^\eps z \right) \dd\Surf (x)+  \mathrm{b}_\eps \calG(z)  \qquad \text{with }  \mathrm{b_\eps} \geq 0\,,
\\
& \text{ and } 
\begin{cases} 
\RCO \calH^\eps(u)    = \nu_\eps \int_{\GC^\eps} \widehat{\alpha}_\lambda(\JUMP{u}) \dd \Surf(x), \EEE 
\\
 \calJ^\eps \EEE(u,z) : = \int_{\GC^\eps} \tfrac{\kappa}2 z Q^\eps ( \JUMP{u}) \dd \Surf(x) \quad \text{ with  } \qquad
Q^\eps ( \JUMP{u}) : = \left| \JUMP{u_1} \right|^2 + \left| \JUMP{u_2} \right|^2 + \eps^2\left| \JUMP{u_3} \right|^2\,,
\end{cases}
\end{aligned}
\end{equation}
(where, with slight abuse of notation, we now denote by $\JUMP{\cdot}$ the jump across the interface $\GC^\eps$), and where 
 $\widehat{\alpha}_\lambda$ and $\calG$ are defined as in \eqref{energy-H} and \eqref{energy-calG}. \EEE
 Note that in the above formulas 
the elasticity tensor $\mathbb{C}$ is taken independent of $\eps$ and also the parameters
$\lambda$ and $\kappa$   are fixed,  whereas  $(a_0^\eps)_\eps\subset (0,+\infty) $, \EEE
$(\nu_\eps)_\eps \subset (0,+\infty) $,
$(\mathrm{b}_\eps)_\eps \subset [0,+\infty) $  are given sequences  whose behavior will be specified 
  in Section 
\ref{s:main-results},  again.  Here, let us only highlight the coefficients $\nu_\eps$, whose role is to tune the term penalizing interpenetration between the  two bodies. 
We also mention  that  the different scaling of the third term 
in $Q^\eps$, which accounts for a different rigidity in the out-of-plane direction, 
is assumed along the footsteps of  \cite{FreParRouZan11}  in order to ensure that, for the limiting problem, both the in-plane and  the out-of-plane directions  contribute to the 
adhesive contact energy. 
\par
\begin{remark}
\label{comparisonFPRZ}
    \upshape
Unlike \cite{FreParRouZan11},  we will not work under the assumption that the material 
has monoclinic symmetry w.r.t.\ the $(x_1,x_2)$-plane,
which corresponds to requiring that  the elasticity tensor $\mathbb{C}$ has entries $\mathbb{C}_{ijk3}=0$ and $\mathbb{C}_{i333}=0$ for $i,j,k=1,2$. \EEE 
In fact, this condition was adopted in  \cite{FreParRouZan11} only in that, in the context of
their analysis,  it allowed for a more 
transparent reformulation of the limit problem in the absence of delamination.
\par
A  different assumption on $\bbC$ and $\bbD$, i.e.\ \eqref{condX-1}, will be instead required, albeit only for Theorem \ref{mainthm-2}. \EEE
\end{remark}
\EEE 
\subsection{The rescaled problems}
\label{subs:rescaled}
We perform a suitable change of variables in order to set the problem in a domain independent of $\eps$.
Thus,
we let
\[
 \Omega_\pm: =\EEE \omega_\pm \times \left(-\frac{1}2, \frac{1}2 \right), \qquad  \GC: =  \EEE \gC \times \left(-\frac{1}2, \frac{ 1}2 \right), \qquad  \GD: =  \gD \times \left(-\frac{1}2, \frac{1}2 \right)
\]
and set 
\[
\Omega: = \Omega_+\cup\GC\cup\Omega_-\,.
\]
We proceed by fixing some notation.
We consider the rescaling function
\[
\cv\eps: \overline{\Omega}\to \overline{\Omega^\eps} \qquad \cv\eps(x_1,x_2,x_3): = (x_1,x_2,\eps x_3)
\]
along with the  operators
\begin{equation}
\label{rescaling-spatial}
\begin{aligned}
&
\reszop:  
\RCO L^1(\GC^\eps) \EEE 
 \to  
 L^1(\GC) \EEE 
\qquad 
 \resz z(x) := 
z(r_\eps(x))\,,
\\
& 
\resgop:  H_{\GD^\eps}^1(\Omega^\eps{\setminus}\GC^\eps;\R^3)  \EEE \to  
H_{\GD}^1(\Omega{\setminus}\GC;\R^3), \qquad 
 \resg u(x) := 
  \left(u_1 \left(\cv\eps(x) \right), u_2 \left( \cv\eps(x) \right), \eps u_3 \left(\cv\eps(x) \right) \right).
  \end{aligned}
  \end{equation}
We define the linear operator
\[
\Lambda_\eps: \Msym \to \Msym \qquad \reso \eps { \Xi} : = \left(
\begin{array}{ c c c }
\xi_{11} &  \xi_{12} &  \tfrac1\eps \xi_{13}
\medskip
\\
\xi_{12} & \xi_{22} & \tfrac1\eps \xi_{23}
\medskip
\\
\tfrac1\eps \xi_{13} &  \tfrac1\eps \xi_{23} & \tfrac1{\eps^2} \xi_{33}
\end{array}
 \right)\,, \qquad  \text{for all $\Xi = (\xi)_{ij}\in \Msym $.}
\]
  \EEE
Hereafter, whenever applying $\Lambda_\eps$ to the tensor $\xi = e(v)$, we will also adopt the notation
\begin{equation}
\label{e-plan}
e^\eps(v) : = \reso\eps{e(v)}, \text{ and   use  the short-hand }  
\epl(v):= \left(\begin{array}{cc}
\partial_1 v_1 &  \tfrac12 (\partial_2 v_1+\partial_1 v_2) 
\medskip
\\
 \tfrac12 (\partial_2 v_1+\partial_1 v_2)  & \partial_{2} v_2 
 \end{array}
  \right),
\end{equation}
so that 
\begin{equation}
\label{e-eps-explicit}
e^\eps(v)=  
\begin{pNiceArray}{cw{r}{2.2cm}c}[margin]
\Block{2-2}{\epl(v)} & &  \tfrac1{2\eps} (\partial_1 v_3+\partial_3 v_1) 
\\
& &   \tfrac1{2\eps} (\partial_2 v_3+\partial_3 v_2) 
\medskip
\\
\tfrac1{2\eps} (\partial_1 v_3+\partial_3 v_1) & \tfrac1{2\eps} (\partial_2 v_3+\partial_3 v_2) & \tfrac1{\eps^2}  \partial_3 v_3 
\end{pNiceArray}
\end{equation}
\EEE
In what follows, we will use the crucial identity
\begin{equation}
\label{key-identity}
e(u)  {\circ} r_\eps \EEE = e^\eps (\resg u) \qquad \text{for all } u \in H_{\GD^\eps}^1(\Omega^\eps{\setminus}\GC^\eps;\R^3) \,.
\end{equation}
\noindent
\textbf{Space and time-rescaling of $\SE$ solutions.}
Along the lines of \cite{Maggiani-Mora} 
we consider both a spatial \emph{and  a temporal} rescaling of the 
 $\SE$ solutions to the adhesive contact system,
\emph{considered on the whole positive  half-line $(0,+\infty)$}.
 Indeed, the existence  of  \emph{Balanced}  $\SE$ solutions defined on $(0,+\infty)$ follows from Theorem \ref{thm:RT} as soon as the data $f$ and $w$ comply with 
\eqref{assRegf} and \eqref{hyp-w} \emph{locally} on $(0,+\infty)$.  In what follows, rescaled solutions will be denoted by the \textsf{sans} font. \EEE
 
 \par
 To be precise, we introduce the rescaling operators 
\[
\begin{aligned}
& \resop:   \RCO L_\loc^\infty(0,\infty; L^1(\GC^\eps)) \EEE 
 \to  \RCO L_\loc^\infty(0,\infty; L^1(\GC)) \EEE 
  \qquad  \res {z}(t,x): =
z\left(\frac t\eps, \cv\eps(x) \right)\,,
\\
&
 \reshop: L_\loc^\infty(0,\infty; H^1(\Omega^\eps{\setminus}\GC^\eps;\R^3))  \to  
L_\loc^\infty(0,\infty; H^1(\Omega{\setminus}\GC;\R^3)), \EEE
\\ 
&
 \qquad \resh u(t,x) := 
   \left(u_1 \left(\frac t\eps, \cv\eps(x) \right), u_2 \left(\frac t\eps, \cv\eps(x) \right), \eps u_3 \left(\frac t\eps, \cv\eps(x) \right) \right) \,.
\end{aligned}
\]
Then, with a family 
\[
\begin{aligned}
&
(u_\eps)_\eps \subset H_\loc^1(0,\infty;H_{\GD^\eps}^1(\Omega^\eps{\setminus}\GC^\eps;\R^3)) {\cap}W_\loc^{1,\infty}(0,\infty;L^2(\Omega^\eps;\R^3)) {\cap} H_\loc^2(0,\infty;H_{\GD^\eps}^1(\Omega^\eps{\setminus}\GC^\eps;\R^3)^*)
\\
&
(z_\eps)_\eps \subset   L_\loc^\infty(0,\infty; L^\infty(\GC^\eps;[0,1])), \EEE   \text{ and }  
(z_\eps)_\eps \subset  L_\loc^\infty (0,\infty;\SBV(\GC^\eps;\{0,1\})) \text{ if } \mathrm{b}>0, \EEE
 \end{aligned}
\]
   of $\SE$ solutions to the (damped) system for adhesive contact, we associate the 
rescaled  functions 
\[
\rese u\eps: = \resh{u_\eps}, \qquad \rese z\eps: = \res{z_\eps}\,.
\]
We will also rescale the initial data $(u_{0,\eps},   \dot{u}_{0,\eps}, \EEE z_{0,\eps})_\eps$, the  Dirichlet loading $w_\eps$, and the force $f_\eps$, by setting
\[
\begin{aligned}
&
\rese {u}\eps_0: = \resg{u_{0,\eps}}, 
\qquad  \resed u\eps_0: = \frac1\eps\resg{\dot{u}_{0,\eps}}, 
\qquad \rese {z}\eps_0: = \resz {z_{0,\eps}},
\\
&
\rese w\eps: = \resh{w_\eps}, 
\\
&
 \rese f\eps: =  \resw{f_\eps} \quad \text{ with } \quad \resw{f}(t,x): = 
\left(f_1 \left(\frac t\eps, \cv\eps(x) \right), f_2 \left(\frac t\eps, \cv\eps(x) \right), \frac1\eps f_3 \left(\frac t\eps, \cv\eps(x) \right) \right)\,.
\end{aligned}
\]
 We postpone to Remarks \ref{rmk:resc-u} \& \ref{rmk:resc-forces}  some comments on the rescalings of the displacements and of the forces. \EEE
\paragraph{\RCO Adhesive contact system for  $(\rese u\eps,\rese z\eps)$.}
With Proposition \ref{prop:rescaled-adh-cont} below  we are going to show that the rescaled functions $(\rese u\eps,\rese z\eps)$ are $\SE$ solutions of the
 adhesive contact system
with the kinetic energy 
\begin{equation}
\label{kin-cont-eps}
\rese K\eps(\dot{\sfu}): = \frac{\varrho_\eps}{2}
\int_\Omega \left( \eps^2 |\dot{\sfu}_1 |^2 {+}  \eps^2 |\dot{\sfu}_2|^2 {+}    |\dot{\sfu}_3|^2 \right) \dd x,
\end{equation}
and
driven by the $1$- and $2$-homogeneous dissipation potentials
\begin{align}
&
\label{R-resc}
\rese R \eps \colon \Spz\to[0,\infty]\,,\;
{\rese R{\eps}}(\dot \sfz):=\int_{\GC} R_\eps(\dot \sfz)\,\mathrm{d}\Surf(x)\,,\;
R_\eps(\dot \sfz):=\left\{
\begin{array}{ll}
a_1^\eps|\dot \sfz |&\text{if }\dot  \sfz\leq 0\,,\\
\infty&\text{otherwise}\,,
\end{array}
\right.
\\
&
\label{V-resc}
 \rese V\eps \colon H_{\GD}^1(\Omega{\setminus}\GC;\R^3) \to [0,\infty)\,,\; \EEE
\rese V\eps (\dot{\sfu}) : = \frac\eps 2 \int_{\Omega{\setminus}\GC} 
  \mathbb{D}_\eps e^\eps(\dot{\sfu}) {:} e^\eps(\dot{\sfu}) \dd x,
  \end{align}
\EEE
as well as  by 
 the energy functional
 \begin{subequations}
 \label{defEk-resc}
\begin{align}
\nonumber
&
\rese{E}\eps:[0,\infty) \times H_{\GD}^1(\Omega{\setminus}\GC;\R^3) \times L^1(\GC)  \to\R\cup\{\infty\}, 
 \text{ defined as }  \rese E\eps =  \rese{E}\eps_{\mathrm{bulk}} +  \rese{E}\eps_{\mathrm{surf}} \text{ with }
\\
&
 \label{defEk-resc-bulk}
\text{the bulk energy } \rese{E}\eps_\bulk \EEE (t,\sfu) :=
 \int_{\Omega{\setminus}\GC}\tfrac{1}{2}\mathbb{C}e^\eps(\sfu):e^\eps(\sfu) \dd x  - \pairing{}{\Spvr}{\rese F \eps(t)}{\sfu} 
 \intertext{with $\rese F\eps$ given by }
 &
 \nonumber
 \begin{aligned}
&
 \pairing{}{\Spvr}{\rese F \eps(t)}{\sfu} :=   \int_{\Omega} \rese f\eps(t) \sfu \dd x -
 \int_{\Omega{\setminus}\GC} \bbC e^\eps(\rese w\eps(t)): e^\eps(\sfu) \dd x 
-\eps\int_{\Omega{\setminus}\GC} \bbD_\eps e^\eps(\resed w\eps(t)): e^\eps(\sfu) \dd x 
\\  & \quad\qquad \qquad \qquad \qquad \qquad
- \eps^2 \int_\Omega \varrho_\eps \sum_{i=1}^2 \reseidd w\eps i(t)  \sfu_{i}  \dd x
-\int_\Omega \varrho_\eps   \reseidd w\eps 3(t)  \sfu_3 \dd x
\end{aligned}
\intertext{and the surface energy }
& 
 \label{defEk-resc-surf}
\begin{aligned}
&
 \rese{E}\eps_{\mathrm{surf}} (\sfu,\sfz) = 
\rese{H}\eps(\sfu) +\rese {J}\eps (\sfu,\sfz)  + \int_{\GC} \left(I_{[0,1]}(\sfz){-}a_0^{\eps} \sfz \right) \dd \Surf(x) +\mathrm{b}_\eps\calG(\sfz) \text{ with }
\\
&
\qquad  \rese{H}\eps(\sfu) =   \nu_\eps \EEE \int_{\GC}\widehat{\alpha}_\lambda(\JUMP{\sfu_1,\sfu_2,0}) \dd \Surf(x), 
\\
& 
\qquad 
 \rese {J}\eps (\sfu,\sfz) = \int_{\GC}  \tfrac \kappa 2 \EEE  \sfz Q(\JUMP{\sfu})\dd \Surf(x)\,,   \qquad \text{ with }  Q(\JUMP{\sfu})=|\JUMP{\sfu}|^2 \,. 
 \end{aligned}
  \end{align}
\end{subequations}

\begin{remark}[Space and time rescaling of the displacements]
\label{rmk:resc-u}
\upshape    The space-rescaling of the displacements is consistent with classical dimension reduction results in the elasticity setting. In particular, the ratio $\eps$ between the tangential and vertical displacements is motivated by the seminal linearization results in \cite{friesecke-james-muller}.
 As a consequence of this spatial rescaling, it is natural to introduce the operator $e^\eps$ 
  from \eqref{e-eps-explicit}, \EEE
 for which the key identity \eqref{key-identity} holds. 
From the  viewpoint of analysis, we remark that the a priori estimates for ($e^\eps(\resed u\eps))_\eps$ to be  obtained later on  will ultimately lead to the Kirchhoff-Love structure of the limiting displacements, cf.\ \eqref{KL-definition}. We also point out that an explicit dependence of the limiting displacement on the $x_3$-variable is not new in dimension reduction studies for inelastic problems. We refer to \cite[Section 5]{DavoliMora15}, for an example in the setting of perfect plasticity.

The time rescaling corresponds to assuming that oscillations in the set $\Omega^\eps$ occur at a slow time scale, so that a reparametrization is needed to see them in the limit. With the rescaling adopted in this paper, our result is consistent with the classical ones obtained in dimension reduction problems for elastodynamics,  as well as in \cite{Maggiani-Mora} for a dynamical model of perfectly plastic plates. 
    In particular, our reduced model coincides with the dynamic Von K\'arm\'an plate model justified in nonlinear elasticity in \cite{abels-mora-muller,abels-mora-muller2} (see also the references therein for alternative formal derivations by asymptotic expansions). Concerning wave propagation, in the limit  only the inertial contribution affects the component $u_3$, thus solely  allowing for wave propagation in the normal component of the displacement. 
\end{remark}
\begin{remark}[Space and time rescaling of the forces]
\label{rmk:resc-forces}
\upshape
The space and time rescaling for the 
Dirichlet loadings $w_\eps$ obviously needs to
be the same as that for the displacements. In turn, the spatial rescaling of 
$f_\eps$ needs to be ``compatible'' with that of
$u_\eps$ and thus 
 features a factor $\tfrac1\eps$ in the vertical component. Analytically, this guarantees the validity of the key identity \eqref{compatibility-rescaling} ahead. 
\par
We point out that the very same rescaling of the body forces was adopted in \cite{Maggiani-Mora}, where the dimension reduction was carried out in the context of dynamical perfect plasticity. While these  specific choices seem to 
be necessary for our analysis, the time rescaling  of $f_\eps$ and $w_\eps$ brings about significant limitations in the applicability of our results, see Remark \ref{rmk:criticism} ahead. 
\end{remark}
\EEE

 The rescaled conditions satisfied by the pair $(\rese{u}\eps,\rese{z}\eps)$ are collected below.\EEE
\begin{proposition}
\label{prop:rescaled-adh-cont}
For every $\eps>0$ the functions 
\[
\begin{aligned}
&
\rese u\eps  \in  H_\loc^1(0,\infty;H_{\GD}^1(\Omega{\setminus}\GC;\R^3)), \RMA \text{ and, additionally, }  
\\
&\RMA  \quad
\varrho_\eps \rese u\eps \in   W_\loc^{1,\infty}(0,\infty;L^2(\Omega;\R^3)) {\cap} H_\loc^2(0,\infty;H_{\GD}^1(\Omega{\setminus}\GC;\R^3)^*), \EEE
\\
& \rese z\eps \in    L_\loc^\infty (0,\infty;L^\infty(\GC;[0,1])), 
\text{ and } 
\rese z\eps \in     L_\loc^\infty (0,\infty;\SBV(\GC;\{0,1\})) \text{ if } \mathrm{b}>0,
\end{aligned}
\]
are (Balanced) $\SE$  solutions of the damped  inertial system 
$\ingrsysreps$, namely they fulfill 
\begin{compactitem}
\item[-] the rescaled weak momentum balance
 for almost all $t\in (0,\infty)$ and all $\varphi \in H^1_{\Dir}(\Omega{\setminus}\GC;\R^3)$
\begin{equation} \label{weak-mom-eps}
\begin{aligned}
&
 \eps^2 \langle \varrho_\eps \reseidd u\eps1(t), \varphi_1 \rangle_{H^1(\Omega{\setminus}\GC)} + \eps^2  \langle    \varrho_\eps \reseidd u\eps2(t), \varphi_2 \rangle_{H^1(\Omega{\setminus}\GC)}
+
\langle \varrho_\eps  \reseidd u\eps3(t), \varphi_3 \rangle_{H^1(\Omega{\setminus}\GC)} \EEE
\\
& \quad + \int_{\Omega{\setminus}\GC}  \left( \eps \mathbb{D}_\eps e^\eps(\resed u\eps(t)) {:} e^\eps(\varphi) 
{+} \bbC e^\eps(\rese u\eps (t))  {:} e^\eps(\varphi)  \right) \dd x 
\\
& \quad 
+  \nu_\eps \EEE \int_{\GC}\alpha_\lambda(\JUMP{\rese u\eps_1(t), \rese u\eps_2(t), 0})\cdot \JUMP{\varphi_1,\varphi_2,0} \dd \Surf(x)
+\int_{\GC} \kappa\rese z\eps(t)\JUMP{\rese u\eps(t)}\JUMP{\varphi}\,\mathrm{d}\Surf (x)
\\
&
= \int_\Omega \rese f\eps(t)\varphi \dd x 
-\int_{\Omega\setminus \GC} \bbC e^\eps(\rese w\eps(t)) : e^\eps(\varphi) \dd x - \eps\int_{\Omega\setminus \GC} \bbD_\eps
 e^\eps(\resed w\eps(t)) : e^\eps(\varphi) \dd x
\\
& \quad
- \eps^2 \int_\Omega \left(\varrho_\eps \reseidd w\eps1(t) \varphi_1 {+}   \varrho_\eps \reseidd w\eps2(t) \varphi_2 \right) \dd x 
-
\int_\Omega \varrho_\eps  \reseidd w\eps3(t) \varphi_3 \dd x;
\end{aligned}
\end{equation}
\item[-] the rescaled semistability condition
 for every $t \in [0,\infty)$, \EEE
\RCO featuring the set $\rese Z\eps(t): = r_\eps^{-1} (Z_\eps(\eps t))$, \EEE 
\begin{equation} \label{reduced-semistab-eps}
\begin{aligned}
&
\int_{\GC}\RNEW \tfrac \kappa 2 \EEE \rese z\eps(t) |\JUMP{\rese u\eps(t)}|^2 \dd \Surf(x) + 
\RCO \mathrm{b}_\eps \EEE P(\rese Z\eps(t),\GC)  - \int_{\GC} a_\eps^0  \rese z\eps(t)  \dd \Surf(x) \EEE
\\ &  \leq  \int_{\GC}  \RNEW \tfrac \kappa 2 \EEE  \widetilde{z} |\JUMP{\rese u\eps(t)}|^2 \dd \Surf(x) 
+ \RCO \mathrm{b}_\eps \EEE P(\widetilde{Z},\GC)  - \int_{\GC} a_\eps^0  \widetilde{z}  \dd \Surf(x) \EEE
+\int_{\GC} a_\eps^1 |\widetilde{z}{-}\rese z\eps(t)| \dd \Surf(x)  \\ & \qquad  \text{for all }
 \RCO \widetilde{z} \in L^1(\GC)  \  (\widetilde{z} \in
 \SBV(\GC;\{0,1\}) \text{ if } \mathrm{b_\eps}>0), \EEE  \text{ with } 0\leq\widetilde{z}\leq \rese z\eps(t)  \ \aein \  \GC;
\end{aligned}
\end{equation}
\item[-] the rescaled  energy-dissipation balance along any interval $[s,t]\subset [0,\infty)$
\begin{equation}
\label{EDB-eps}
 \begin{aligned}
 & 
\frac{\varrho_\eps}{2}
\int_\Omega \left( \eps^2 |\reseid u\eps1 (t)|^2 {+}  \eps^2 |\reseid u\eps2 (t)|^2 {+}    |\reseid u\eps3 (t)|^2 \right) \dd x 
+ \eps \int_s^t \int_{\Omega\setminus \GC} \bbD_\eps e^\eps (\resed u\eps) {:} e^\eps(\resed u\eps) \dd x  \dd r
\\
  & 
\quad + \Var_{\mathsf{R^\eps}}(\rese z\eps, [s,t])+
 \RCO \rese E\eps \EEE (t,\rese u\eps(t),\rese z\eps(t))\\ 
&  =
\frac{\varrho_\eps}{2}
\int_\Omega \left( \eps^2 |\reseid u\eps1 (s)|^2 {+}  \eps^2 |\reseid u\eps2 (s)|^2 {+}    |\reseid u\eps3 (s)|^2 \right) \dd x 
+   \RCO \rese E\eps \EEE (s,\rese u\eps(s),\rese z\eps(s)) 
\\
& \quad + \int_s^t \partial_t
 \rese E\eps(r,\rese u\eps(r),\rese z\eps(r)) \,  \mathrm{d}r\,.
 \end{aligned}
 \end{equation}
\end{compactitem}
\end{proposition}
\begin{proof}
Throughout the proof, to avoid overburdening notation we will write the duality pairings involving the inertial terms in the momentum balance as integrals. \EEE
\medskip

\noindent
{\bf $\vartriangleright $ \RCO Momentum balance \EEE \eqref{weak-mom-eps}:} In the weak momentum balance 
\eqref{weak-mom} satisfied by the $\SE$ solutions $(u_\eps, z_\eps)$  and with $\mathbb{D}$ and $Q$ replaced by $\mathbb{D}_\eps$ and $Q_\eps$, cf. \eqref{defEk-eps},  we choose test functions 
$v_\eps \in  H^1_{\Dir^\eps}(\Omega^\eps{\setminus}\GC^\eps;\R^3)$ of the form $v_\eps = (v_{1,\eps}, v_{2,\eps}, v_{3,\eps})$
with
\begin{equation}
\label{choice-of-test-functions}
v_{i,\eps}(x) = \varphi_i(r_\eps^{-1}(x)) \text{ for } i =1,2, \qquad
v_{3,\eps}(x) =\frac1\eps \varphi_3(r_\eps^{-1}(x))
 \qquad \text{for an arbitrary } \varphi \in H^1_{\Dir}(\Omega{\setminus}\GC;\R^3)\,,
\end{equation}
namely $\varphi = \resg {v_\eps}$. 
For later use, we record here that, due to \eqref{key-identity}, 
\begin{equation}
\label{e-eps-later}
 e(v_\eps) \RCO {\circ} r_\eps\EEE  =  e^\eps ( \resg {v_\eps})  = e^\eps (\varphi) \,.
\end{equation}
 We then divide \eqref{weak-mom} by $\eps$ and write it at the time  $\tfrac t\eps$ \EEE for almost all $t\in (0,\infty)$, thus obtaining 
\begin{equation} \label{weak-mom-eps-detail}
\begin{aligned}
&
\frac{1}\eps \int_{\Omega^\eps}  \varrho_\eps  \ddot{u}_\eps(\eps^{-1} t) {\cdot}  v_\eps \dd x 
+\frac1\eps \int_{\Omega^\eps{\setminus}\GC^\eps} \left( \mathbb{D}_\eps e(\dot{u}_\eps (\eps^{-1} t)) {:} e(v_\eps) 
{+} \mathbb{C}e(u_\eps(\eps^{-1} t)){:}e(v_\eps) \right) \,\mathrm{d}x
\\
& \quad 
+  \RCO \nu_\eps \EEE \int_{\GC^\eps} \tfrac1\eps \alpha_\lambda(\JUMP{u_\eps(\eps^{-1} t)}){\cdot} \JUMP{v_\eps} \dd \Surf(x)
\RNEW +\frac1\eps \sum_{i=1}^2 \int_{\GC^\eps} \kappa  z(\eps^{-1} t) \EEE \,\JUMP{u_{i,\eps}(\eps^{-1} t)}\JUMP{v_{i,\eps}}\,\mathrm{d}\Surf  (x)
\\
& \quad 
+ \eps  \int_{\GC^\eps} \kappa z(\eps^{-1} t) \EEE  \,\JUMP{u_{3,\eps}(\eps^{-1} t)}\JUMP{v_{3,\eps}}\,\mathrm{d}\Surf  (x) \EEE
\\
&
=\frac1\eps  \int_{\Omega^\eps} f_\eps (\eps^{-1} t)v_\eps \dd x 
-\frac1\eps \int_{\Omega^\eps\setminus \GC^\eps} \bbC e(w_\eps(\eps^{-1} t)) {:} e(v_\eps) \dd x
 - \frac1\eps \int_{\Omega^\eps\setminus \GC^\eps} \bbD_\eps e(\dot{w}_\eps(\eps^{-1} t))
 {:} e(v_\eps) \dd x
 \\
 & \quad 
-\frac{1}\eps \int_{\Omega^\eps} \varrho_\eps  \ddot{w}_\eps(\eps^{-1} t) v_\eps \dd x,
\end{aligned}
\end{equation}
\RNEW where we have used that $\nabla Q_\eps(y) = 2(y_1,y_2,\eps^2 y_3)$ for every $y \in \R^3$\EEE, see \eqref{defEk-eps}.
Let us now examine each of the above integral terms separately. The first one equals
\[
\begin{aligned}
&
\frac{1}\eps \sum_{i=1}^2 \int_{\Omega^\eps}  \varrho_\eps  \ddot{u}_{i,\eps}(\eps^{-1} t,x) v_{i,\eps}(x) \dd x  + 
\frac{1}\eps  \int_{\Omega^\eps}  \varrho_\eps  \ddot{u}_{3,\eps}(\eps^{-1} t,x) v_{3,\eps}(x) \dd x
\\
&
\stackrel{(1)}{=}
 \sum_{i=1}^2 \int_{\Omega}  \varrho_\eps  \ddot{u}_{i,\eps}(\eps^{-1} t,r_\eps(x)) v_{i,\eps}(r_\eps(x))  \dd x  \EEE
+  \int_{\Omega}  \varrho_\eps \ddot{u}_{3,\eps}(\eps^{-1} t,r_\eps(x)) v_{3,\eps}(r_\eps(x))  \dd x \EEE
\\
& 
\stackrel{(2)}{=}  \sum_{i=1}^2 \int_{\Omega}  \varrho_\eps \eps^2  \reseidd u\eps i(s,x) \varphi_i(x) \dd x + 
\int_{\Omega}  \varrho_\eps   \reseidd u\eps 3(s,x) \varphi_3(x) \dd x,
\end{aligned}
\]
with (1) following from the spatial change of variables $x \mapsto r_\eps(x)$ 
and (2) from the temporal change of variables $s = \eps^{-1} t$, taking into account that,  
for $i=1,2$,
  $  \ddot{u}_{i,\eps}(\eps^{-1} t,r_\eps(x)) = \eps^2  \reseidd u\eps i(s,x)$ and $v_{i,\eps}(r_\eps(x)) = \varphi_i(x)$, 
  while $\ddot{u}_{3,\eps}(\eps^{-1} t,r_\eps(x))  = \eps \reseidd u\eps 3(s,x)$ and  $v_{3,\eps}(r_\eps(x)) = \
  \frac1\eps \varphi_3(x)$.
  With the same change of variables we find that 
  \[
  \frac1\eps
   \int_{\Omega^\eps{\setminus}\GC^\eps} \mathbb{D}_\eps e(\dot{u}_\eps (\eps^{-1} t,x)) {:} e(v_\eps(x)) \dd x  
   =  \eps  \int_{\Omega{{\setminus}}\GC} \mathbb{D}_\eps e^\eps (\resed u\eps (s, x)) {:} e^\eps (\varphi(x)) \dd x,    
   \]
  where we have also  used that 
\RCO  $ e(\dot{u}_\eps (\eps^{-1} t, r_\eps(x))) = \eps e^\eps (\resed u\eps (s, x))$ \EEE  and \eqref{e-eps-later}. Relying again on \eqref{key-identity} we also find that 
  \[
  \frac1\eps \int_{\Omega^\eps{\setminus}\GC^\eps} \mathbb{C}e(u_\eps(\eps^{-1} t,x){:} e(v_\eps(x))  \,\mathrm{d}x  = 
   \int_{\Omega{{\setminus}}\GC}  \mathbb{C} e^\eps (\rese u\eps (s, x)) {:} e^\eps (\varphi(x)) \dd x. 
  \]
   Recall that $\alpha_{\lambda}$ is the subdifferential of $\widehat{\alpha}_\lambda$, and hence
   $\alpha_\lambda(v)=\tfrac1\lambda (v-\Pi_K(v))$ for every $v\in \mathbb{R}^3$,
   \RCO with $\Pi_K$ the projection on the cone $K = \{ v\in \R^3\, : v \cdot n \geq 0 \}$.
    In particular, since  
    $\Pi_K$
    only acts on the first two components of its arguments, it follows that the third component of the vector $\alpha_\lambda(\JUMP{u_\eps(\eps^{-1} t)})$ is null. Thus, from \EEE
 a further change of variables, we find 
\[
\begin{aligned}
\RCO \nu_\eps \EEE  \int_{\GC^\eps}\tfrac1\eps\alpha_\lambda(\JUMP{u_\eps(\eps^{-1} t)}){\cdot} \JUMP{v_\eps} \dd \Surf(x) & = 
\RCO \nu_\eps \EEE
\sum_{i=1}^2  \int_{\GC}\alpha_{i,\lambda}(\JUMP{u_\eps(\eps^{-1} t,r_\eps(x))}) {\cdot} \JUMP{v_{i,\eps}(r_\eps(x))} \dd \Surf(x)\\ &  = 
\RCO \nu_\eps \EEE \int_{\GC}\alpha_\lambda(\JUMP{\rese u\eps_1(s), \rese u\eps_2(s), 0}) {\cdot} \JUMP{\varphi_1,\varphi_2,0} \dd \Surf(x)\,.
\end{aligned}
\]
 The last \RNEW two terms \EEE on the left-hand side of \eqref{weak-mom-eps-detail} become
\[
\begin{aligned}
&
 \int_{\GC}   \left( \sum_{i=1}^2  \kappa  z(\eps^{-1} t) \EEE \JUMP{u_{i,\eps}(\eps^{-1} t,r_\eps(x))} {\cdot} \JUMP{v_{i,\eps}(r_\eps(x))} 
  {+}  \kappa  z(\eps^{-1} t) \EEE \eps^2 \JUMP{u_{3,\eps}(\eps^{-1} t,r_\eps(x))}{\cdot} \JUMP{v_{3,\eps}(r_\eps(x))} 
 \right) 
\mathrm{d}\Surf(x) \\ & = \int_{\GC} \kappa\rese z \eps(s)\JUMP{\rese u\eps(s)} {\cdot} \JUMP{\varphi}\,\mathrm{d}\Surf  (x)
\end{aligned}
\]
where the last identity follows taking into account that $u_{3,\eps}(\eps^{-1} t,r_\eps(x)) = \frac1\eps \resei u\eps3 (s,x)$ and, again, $v_{3,\eps}(r_\eps(x))= \frac1\eps \varphi_3(x)$. 
Finally, repeating the very same calculations as in the above lines we find that the right-hand side of \eqref{weak-mom-eps-detail} equals 
\[
\begin{aligned}
\text{r.h.s.\ of  \eqref{weak-mom-eps-detail}} &  = 
\int_\Omega \rese f\eps(s)\varphi \dd x 
-\int_{\Omega\setminus \GC} \bbC e^\eps(\rese w\eps(s)) : e^\eps(\varphi) \dd x - \eps\int_{\Omega\setminus \GC} \bbD_\eps
 e^\eps(\resed w\eps(s)) : e^\eps(\varphi) \dd x
\\
& \quad
- \eps^2 \int_\Omega \left(\varrho_\eps \reseidd w\eps1(s) \varphi_1 {+}  \varrho_\eps \reseidd w\eps2(s) \varphi_2 \right) \dd x 
-
\int_\Omega \varrho_\eps  \reseidd w\eps3(s) \varphi_3 \dd x.
\end{aligned}
\]
 In particular, we point out  that 
the identity 
\begin{equation}
    \label{compatibility-rescaling}
\frac1\eps  \int_{\Omega^\eps} f_\eps (\eps^{-1} t)v_\eps \dd x =\int_\Omega \rese f\eps(s)\varphi \dd x 
\end{equation}
holds thanks to the chosen rescaling of the data $f_\eps$. \EEE
All  in all, we conclude the validity of \eqref{weak-mom-eps}.
\medskip

\noindent
\paragraph{$\vartriangleright$ Semistability condition \eqref{reduced-semistab-eps}:}
\RCO Again, it is convenient to prove \eqref{reduced-semistab-eps} directly in the case $\mathrm{b}_\eps>0$. Thus, \EEE
 we write the semistability condition \eqref{reduced-semistab} satisfied by 
the curves $(u_\eps, z_\eps)$ at the time $\eps^{-1} t$, with $t\in (0,\infty) $ arbitrary, and divide it by $\eps$. Performing the spatial change of variable $x\to r_\eps(x)$ leads to
\[
\begin{aligned}
&
\int_{\GC} \tfrac{\kappa}{2}   z_\eps \left( \eps^{-1} t, r_\eps(x)\right) 
\left( \left|\JUMP{u_{1,\eps} \left(\eps^{-1} t , r_\eps(x) \right)} \right|^2 {+} 
\left|\JUMP{u_{2,\eps} \left(\eps^{-1} t , r_\eps(x) \right)} \right|^2 {+} \eps^2 \left|\JUMP{u_{3,\eps} \left(\eps^{-1} t , r_\eps(x) \right)} \right|^2 \right)  \dd \Surf(x) 
\\
&\qquad  + 
\RCO \mathrm{b}_\eps \EEE   P(\RCO r_\eps^{-1} ( Z_\eps(\eps^{-1} t) ) \EEE,\GC)
\\ &  \leq 
 \int_{\GC} \tfrac{\kappa}{2}   \widetilde{z}(r_\eps(x)) 
\left( \left|\JUMP{u_{1,\eps} \left(\eps^{-1} t , r_\eps(x) \right)} \right|^2 {+} 
\left|\JUMP{u_{2,\eps} \left(\eps^{-1} t , r_\eps(x) \right)} \right|^2 {+} \eps^2 \left|\JUMP{u_{3,\eps} \left(\eps^{-1} t , r_\eps(x) \right)} \right|^2 \right) \dd \Surf(x)
\\
& \qquad 
+\RCO \mathrm{b}_\eps \EEE P( \RCO  r_\eps^{-1} (\widetilde{Z})\EEE),\GC)\
+\int_{\GC} (a_\eps^0{+}a_\eps^1) |\widetilde{z}(r_\eps(x)){-}z_\eps(\eps^{-1} t, r_\eps(x))| \dd \Surf(x) \\ & \qquad  \text{for all }
 \widetilde{z} \in \SBV(\GC;\{0,1\}) \text{ with } 0\leq\widetilde{z}\leq \rese z\eps(\eps^{-1} t)  \ \aein \  \GC  \text{ and for every } t \in [0,\infty),
 \end{aligned}
 \]
 whence we immediately infer \eqref{reduced-semistab-eps}.
 \medskip
 \noindent
 
\paragraph{$\vartriangleright$ Energy-dissipation balance \eqref{EDB-eps}:} 
Recall that, by Thm.\ \ref{thm:RT} the  \emph{Balanced} $\SE$ solutions $(u_\eps,z_\eps)$ fulfill \eqref{enineq} as an energy-dissipation \emph{balance} 
on any arbitrary sub-interval of $[0,\infty)$. We consider it 
on the interval $[\eps s, \eps t] $ 
 for all $0\leq s \leq t <\infty$
and divide it by $\frac1\eps$.  We write explicitly the single contributions to the energy functional $\calE$,  we perform the  change of variables $x \to r_\eps(x)$ and,   for  the temporal variable $\tau$,  the change $\tau\to \eps^{-1} \tau$. \EEE Repeating the same calculations as in the previous lines, we end up with the following identity 
\begin{equation}
\label{enid-eps-expl}
\begin{aligned}
&
\frac{\varrho_\eps}2  \eps^2 \int_{\Omega} \sum_{i=1}^2 |\reseid u\eps i(t)|^2 \dd x +
\frac{\varrho_\eps}2 \int_{\Omega} |\reseid u\eps 3(t)|^2 \dd x 
+\eps \int_s^t \int_{\Omega{\setminus}\GC}  \bbD_\eps e^\eps (\resed u\eps) {:} e^\eps (\resed u\eps) \dd x \dd r
+
\Var_{\mathsf{R^\eps}}(\rese z\eps, [s,t])
\\
&
\quad +\int_{\Omega{\setminus}\GC} \tfrac12 \bbC e^\eps(\rese u\eps(t)) {:} e^\eps(\rese u\eps(t))  \dd x +
\RCO \nu_\eps \EEE
\int_{\GC} \widehat{\alpha}_\lambda (\JUMP{\resei u \eps1(t), \resei u \eps2(t), 0}) \dd \Surf (x)
- \pairing{}{\Spvr}{\rese F \eps(t)}{\rese u\eps(t)}
\\
& \quad 
+ b_\eps \calG(\rese z\eps(t))+\int_{\GC} \tfrac \kappa 2 \rese z\eps(t) |\JUMP{\rese u\eps(t)}|^2 \dd \Surf(x)
\\
&
=
\frac{\varrho_\eps}2  \eps^2 \int_{\Omega} \sum_{i=1}^2 |\reseid u\eps i(s)|^2 \dd x +
\frac{\varrho_\eps}2 \int_{\Omega} |\reseid u\eps 3(s))|^2 \dd x 
\\
& \quad +
\int_{\Omega{\setminus}\GC} \tfrac12 \bbC e^\eps(\rese u\eps(s)) {:} e^\eps(\rese u\eps(s))  \dd x 
+\RCO \nu_\eps \EEE\int_{\GC} \widehat{\alpha}_\lambda (\JUMP{\resei u \eps1(s), \resei u \eps2(s), 0}) \dd \Surf (x)
- \pairing{}{\Spv}{\rese F \eps(s)}{\rese u\eps(s)}
\\
& \quad 
+ \mathrm{b}_\eps   \calG(\rese z\eps(s))+\int_{\GC} \tfrac \kappa 2 \rese z\eps(s) |\JUMP{\rese u\eps(s)}|^2 \dd \Surf(x)
\\
& 
\quad
- \int_s^t \int_\Omega \resed f\eps \rese u\eps \dd x \dd r 
+  \int_s^t  \int_{\Omega{\setminus}\GC}\bbC e^\eps (\resed w\eps) {:} e^\eps (\rese u\eps) \dd x \dd r 
+ \eps \int_s^t \int_{\Omega{\setminus}\GC}  \bbD_\eps e^\eps (\resedd w\eps) {:} e^\eps (\rese u\eps) \dd x \dd r 
\\
& \quad 
+ \eps^2 \int_s^t   \int_\Omega \varrho_\eps  \sum_{i=1}^2 \reseiddd w\eps i  \resei u\eps i \dd x \dd r
+ \int_s^t  \int_\Omega \varrho_\eps \reseiddd w\eps 3  \resei u\eps 3 \dd x\dd r  \,,
\end{aligned}
\end{equation}
namely  \eqref{EDB-eps}.  This finishes the proof. 
\end{proof}
\section{Our  dimension \EEE reduction results}
\label{s:main-results}
Prior to stating the main results of this paper, we need to introduce some notation.
We denote by $\KL$ the  Kirchhoff-Love space
\begin{equation}
\label{KL-definition}
 \KL: = \{ u \in H^1(\Omega;\R^3)\, : (e(u))_{i,3} =0 \text{ for all } i=1,2,3\}. 
\end{equation}
 We will also use the notation
\begin{equation}
\label{KLGD}
\KLGD: = \{ u \in H_{\Dir}^1(\Omega{\setminus}\GC;\R^3)\, : (e(u))_{i,3} =0 \text{ for all } i=1,2,3\}\,.
\end{equation}
We recall (see, e.g., \cite[Theorem 1.7-1]{ciarlet2}) \EEE 
that,   for a given $w \in H^1(\Omega;\R^3)$, we
have $ w \in \KL$ if and only if
  $w_3 \in  H^2(\omega)$
 and
 there exists $\overline{w} = (\overline{w}_1,\overline{w}_2) \in  H^1(\omega;\R^2)$ \EEE
 such that
 \begin{equation}
 \label{KL-decomposition}
  w (x',x_3)= \left(
  \begin{array}{ll}
  &
  \overline{w}_1(x') - x_3   \partial_{1}   w_3(x')
  \smallskip
  \\
  &
  \overline{w}_2(x') - x_3  \partial_{2}  w_3(x')
    \smallskip
  \\
  &
w_3(x')
  \end{array} 
   \right) \qquad \text{for almost all } (x',x_3) \in \Omega.
 \end{equation}
  Clearly, the very same characterization holds for 
 $\KLGD$, with the spaces $H_{\dir}^2(\omega{\setminus}\gC)$ 
 and $H_{\dir}^1(\omega{\setminus}\gC;\R^2)$. \EEE
  For later convenience we also introduce the space
 \begin{equation}
 \label{H1KL}
 \begin{aligned}
 W_\loc^{1,2}(0,\infty; \KL) := \{ v \in  W_\loc^{1,2}(0,\infty; 
  H^1(\Omega;\R^3)) \EEE \,:   &\,    v(t) \in \KL, 
 \\
 &   \dot{v}(t) \in \KL \ \foraa t \in (0,T)\}\,,
 \end{aligned}
 \end{equation}
  and, analogously, the space $W_\loc^{1,2}(0,\infty; \KLGD)$. \EEE
 We stress that the second condition in the above definition is redundant, but we have preferred to state it this way just for clarity.
In view of the above characterization of $\KL$, we have that  $v \in  W_\loc^{1,2}(0,\infty; \KL)$
if and only if 
there exist functions  $\overline v \in W_\loc^{1,2}(0,\infty;    H^1(\omega;\R^2))$  \EEE  and 
$v_3 \in  W_\loc^{ 1,2}(0,\infty;   H^2(\omega))$ \EEE   such that \eqref{KL-decomposition} holds. \EEE
 \par
 We also introduce the operator 
 $
 \mathbb{M}:  \R_\sym^{2 \times 2}\to \R_\sym^{3\times 3} 
 $
 defined by 
 \begin{equation}
 \label{def-oper-M}
 \begin{gathered}
 \Xi =  \left( \xi_{ij}\right)_{i,j=1,2}\EEE \mapsto  \mathbb{M} \Xi = \left( 
 \begin{array}{llll}
 \xi_{11} & \xi_{12} & \lambda_1(\Xi)
 \smallskip
 \\
 \xi_{12} \EEE& \xi_{22} & \lambda_2(\Xi) 
  \smallskip
 \\
  \lambda_1(\Xi) &  \lambda_2(\Xi) &  \lambda_3(\Xi)
 \end{array}
 \right)
 \end{gathered}
 \end{equation}
with 
\[
 ( \lambda_1(\Xi), \lambda_2(\Xi), \lambda_3(\Xi)) : = \mathrm{Argmin}_{(\lambda_1,\lambda_2,\lambda_3) \in \R^3}    \Lambda_{\bbC}\EEE\left( 
 \begin{array}{llll}
 \xi_{11} & \xi_{12} & \lambda_1
  \smallskip
 \\
  \xi_{12} \EEE & \xi_{22} & \lambda_2
  \smallskip
 \\
  \lambda_1 &  \lambda_2  &  \lambda_3
 \end{array}
 \right) 
 \]
  where  $ \Lambda_{\bbC}: \R_\sym^{3\times 3}\to [0, \infty)$ is the quadratic form associated with $\bbC$, defined by
 $ \Lambda_{\bbC}(A): = \tfrac12 \bbC A{:} A$  for every $A\in \R_\sym^{3\times 3}$\,.  \EEE
 Following \cite{Maggiani-Mora}, we observe that the triple $ ( \lambda_1(\Xi), \lambda_2(\Xi), \lambda_3(\Xi)) $ can be characterized as the unique solution of the linear system
\begin{equation}
 \label{properties-CM1}
 \bbC\mathbb{M}\Xi  :  \left( 
 \begin{array}{llll}
 0 & 0 & \zeta_1
 \smallskip
 \\
 0 & 0 & \zeta_2
  \smallskip
 \\
  \zeta_1 &  \zeta_2 &  \zeta_3
 \end{array}
 \right) =0 \qquad \text{for all } (\zeta_1,\zeta_2,\zeta_3) \in \R^3.
 \end{equation}
  Equivalently, $\Xi\mapsto \mathbb{M}\Xi$ 
\EEE is a linear map, fulfilling
 \begin{equation}
 \label{properties-CM2}
 (\mathbb{C}\mathbb{M}\Xi)_{i3} =  (\mathbb{C}\mathbb{M}\Xi)_{3i} =0 \qquad \text{for all } i =1,2,3. 
 \end{equation} \EEE
We then define the \emph{reduced} elasticity tensor $\bbC_\red:   \R_\sym^{2 \times 2} 
\to \R_\sym^{3\times 3} $  by
 \begin{equation}
 \label{reduced-elasticity-tensor}
 \bbC_\red \Xi: = \bbC \bbM \Xi \qquad\text{for all } \Xi \in \R_\sym^{2 \times 2} \,.\EEE
 \end{equation}
  We note that, taking into account  \eqref{properties-CM2}, we can identify $\bbC_\red \Xi$ with an element of $\R_\sym^{2 \times 2} $\,.\EEE
  \par
 For 
  Theorem \ref{mainthm-2} ahead, \EEE
  it will be expedient to require that, either the elasticity tensor $\bbC$ (cf. 
 Remark \ref{rk:expl-hyp-w}) or \emph{both} the elasticity and 
viscosity tensors $\bbC$ and $\bbD$ (cf. Hypothesis E) comply with the additional property  
\begin{align}
&
\label{condX-1}
\bbA_{i3kl} = 0  &&   \text{for all } i \in \{1,2,3\} \   \text{and } k,l \in \{1,2\},
\end{align}
where  $\bbA \in \R^{3\times 3\times 3\times 3}$ denotes a symmetric tensor
in the sense of \eqref{assCD}.
A crucial outcome of \eqref{condX-1} is that, if we multiply by $\bbA$ a  symmetric  matrix \EEE  $E=(e_{ij})$   that is,  additionally, \EEE `only planar', i.e.
\begin{equation}
\label{corollary-of-condX}
\begin{cases}
E \in \R_\sym^{3 \times 3},
\\
e_{k3} = e_{3\ell} =0 \quad k, \ell \in \{1,2,3\}
\end{cases}
\text{ then  $\bbA E$ is also planar, i.e. }
(\bbA E)_{i3} =  (\bbA E)_{3j} =0 \quad i, j \in \{1,2,3\}\,.
\end{equation}
Indeed,  it suffices to use that 
\[
(\bbA E)_{ij}  = \sum_{k,\ell =1}^3 \bbA_{ijk\ell} e_{k\ell} 
=\sum_{k,\ell =1}^2 A_{ijkl} e_{k\ell}, \quad \text{ whence }\quad 
(\bbA E)_{i3} \stackrel{\eqref{condX-1}}{=}    \sum_{k,\ell =1}^2  \bbA_{i3kl} e_{k\ell}  =0.
\]
\par
 Suppose now that $\bbC$ complies with condition \eqref{condX-1}, too (we emphasize that we shall require it for Theorem 
\ref{mainthm-2}, only). In that case, we have the following identification for $\bbM$. \EEE
\begin{lemma}
\label{lm:identif}
Assume that, in addition to  \eqref{assCD},  the elasticity tensor   $\bbC$ satisfies 
condition {\eqref{condX-1}}. 
Then, for every $v \in \KL$ we have
\begin{equation}
\label{identification-M-KL}
 \bbM \epl(v) = e(v) = 
\begin{pNiceArray}{cw{r}{0.3cm}c}[margin]
\Block{2-2}{\epl(v)} & &  0 \\
& &  0 \\
0 & 0& 0 
\end{pNiceArray} \EEE
\end{equation}
\end{lemma}
\begin{proof}
It suffices to recall that $\bbM \epl(v)$ is characterized as the  $(3{\times}3)$-matrix  whose planar part coincides with $\epl(v) $, and satisfying \eqref{properties-CM2}. Now, thanks to 
\eqref{corollary-of-condX} we have that 
$
(\bbC  e(v))_{i3} = (\bbC  e(v))_{3j} =0
$ for $i,j\in \{1,2,3\}$, 
and then \eqref{identification-M-KL} ensues. 
\end{proof}


\begin{remark}
\upshape
 Under condition \eqref{condX-1}, the characterization of $\bbM$ obtained in Lemma \ref{lm:identif} leads  \EEE
%
to a structure of the reduced elasticity tensor coherent with that of the tensor $\mathbb{C}^0$ in \cite[Section 4]{FreParRouZan11}. 
\end{remark}
\subsection{Our first convergence result: removing the damping}
\label{ss:Thm1}
Let $(\eps_k)_k \subset (0,\infty)$ be a sequence converging to zero as $k \to\infty$.
 For our first dimension reduction result we will confine the analysis to the following setup.
\begin{condition}
\label{cond:1}
The coefficients $(\varrho_{\eps_k})_k$ and the tensors $(\mathbb{D}_{\eps_k})_k$ satisfy
\begin{equation}
\label{e:condition-one}
\varrho_{\eps_k}=0 
\text{ for all } k\in \N \,, \qquad \eps_k^{\beta} \mathbb{D}_{\eps_k} \to 0  \text{ for some } \beta \in (0,1)\,.
\end{equation}
\end{condition}
\noindent
The scaling condition on $(\mathbb{D}_{\eps_k})_k$ is compatible with the case
 $\mathbb{D}_{\eps_k} \equiv \mathbb{D}$ but clearly allows for more general situations, including
a (controlled) blow-up of $(\mathbb{D}_{\eps_k})_k$. Anyhow, it
will lead to the disappearance of the damping term in the momentum balance. As we will  explain 
in more detail in Remark \ref{rmk:why-we-cant} ahead, the vanishing-thickness analysis can be carried out under this condition only if no inertial terms are present in the original momentum balance for fixed $\eps_k>0$;  that is why, in \eqref{e:condition-one} we require $\varrho_{\eps_k} \equiv 0 $. 
The related PDE system thus ceases to be an inertial system; 
we shall refer to system 
$\grsysepsk$ from \eqref{R-resc}--\eqref{defEk-resc}
  as a \emph{gradient system}.  \EEE
\par
\par 
 In Hyp.\ \ref{Hyp-B} ahead
we   specify  our conditions on the 
 constants $(a_0^{\eps_k})_k$ and $(a_1^{\eps_k})_k$, and more prominently on the \EEE
parameters $(\mathrm{b}_{\eps_k})_k$, $(\nu_{\eps_k})_k$ featuring in the expression \eqref{defEk-resc-surf}  of the energies 
$(\rese E{\eps_k}_{\mathrm{surf}})_k$. \EEE
Since the scaling from Condition \ref{cond:1} leads to 
an \emph{undamped} system
 in  the vanishing-thickness limit,  \EEE
  our conditions on the  sequences $(\mathrm{b}_{\eps_k})_k$ and  $(\nu_{\eps_k})_k$
  in \eqref{hypB-1} below are meant to   \EEE
   somehow compensate the lack of compactness information due to the missing viscosity in the momentum balance. In particular,
 with \eqref{hypB-1} we require that the parameters $\mathrm{b}_{\eps_k}$ are  strictly positive from a certain $\bar k$ on: in fact, we will rely on the $\SBV$ regularizing term to gain extra spatial compactness for the adhesion variable. At the same time, we  will need to impose that the sequence  $(\nu_{\eps_k})_k$ is null because we will not be able to handle the term penalizing the failure of the non-interpenetration  constraint $\JUMP{u}{\cdot} n \geq 0 $. 
 Likewise, 
  the limit surface energy \eqref{defE-resc-surf-limit} reflects the fact that  $\nu = 0$; 
instead, in  the alternative Hypothesis \ref{Hyp-D} (cf.\ Sec.\ \ref{ss:mainthm-2} ahead)  we will 
allow $\nu$ to be positive and the term penalizing the interpenetration  will feature in \eqref{defEve-resc-surf-limit}. 
\EEE 
\begin{hypx}[{Material parameters}]
\label{Hyp-B}
We suppose that 
\begin{subequations}
\label{hypB}
\begin{align}
&
\label{hypB-1}
\exists\, \lim_{k \to \infty}\mathrm{b}_{\eps_k}=\mathrm{b}>0, \quad \exists\, \lim_{k\to\infty}\nu_{\eps_k} =\nu=0, 
\intertext{and that}
&
\label{parameters-below}
\exists\, \lim_{k \to \infty}
a_0^{\eps_k}= a_0>0, \quad \exists\, \lim_{k\to\infty}a_1^{\eps_k} =a_1>0\,.
\EEE
\end{align}
\end{subequations}
%
\end{hypx}
\par
 Let us now    specify the
conditions on the data $(\rese f{\eps_k})_k, \, (\rese w{\eps_k})_k $
under which we will perform our asymptotic analysis as ${\eps_k}\down0$ for  the systems
$\grsysepsk$.
We mention in advance that \eqref{convergences-f} will be the same for Thms.\ \ref{mainthm-1} and  \ref{mainthm-2}. 
\begin{hypx}[{External forces}]
\label{hyp:data}
We suppose that
 \EEE there exists $\sff \in W_\loc^{1,1}(0,\infty;L^2(\Omega;\R^3))$ such that 
 \begin{equation}
 \label{convergences-f}
 \rese f{\eps_k} \to \sff \text{ in } W_\loc^{1,1}(0,\infty;L^2(\Omega;\R^3)).
 \end{equation}
 Additionally, we assume that
\begin{subequations}
\label{cond-rescal-data}
\begin{align}
& 
\label{bounds-w}
\begin{aligned}
\exists\, C_w>0 \ \  \forall\, k \in \N\, : \qquad   & \|e^{\eps_k}(\rese w{\eps_k})\|_{W^{1,2}_\loc(0,\infty;L^2(\Omega;\R^{3\times 3}))} +   \eps_k^{1-\beta}  \|e^{\eps_k}(\rese w{\eps_k})\|_{W_\loc^{2,2}(0,\infty;L^2(\Omega;\R^{3\times 3}))}   \leq C_w \,,\EEE
 \end{aligned}
\end{align}
with $\beta 
\in (0,1)$ from Condition \ref{cond:1}. \EEE
 Further, we require that  there exists a function $\sfw \in W_\loc^{1,2}(0,\infty; \KL)$, cf.\ \eqref{H1KL}, 
 such that,
  as ${\eps_k}\down0$
 \begin{equation}
 \label{convergences-w}
  \begin{cases}
 \rese w{\eps_k} \to \sfw  & \text{in }    W_\loc^{1,2}(0,\infty; H^1(\Omega;\R^3)),  \EEE 
 \\
  \rese w{\eps_k}(0) \to \sfw(0)  & \text{in }  H^1(\Omega;\R^3), \EEE
  \end{cases}
 \end{equation}
 as well as 
 \begin{equation}
 \label{further-cvg-w}
e^{\eps_k}(\rese w{\eps_k}) \to  \bbM \epl(\sfw) \quad \text{in } W_\loc^{1,2}(0,\infty;L^2(\Omega;\R^{3\times 3}))\,. 
 \end{equation}
 \EEE
\end{subequations}
\end{hypx}
\begin{remark}[On conditions \eqref{convergences-f} \& \eqref{cond-rescal-data}]
    \label{rmk:criticism}
\upshape
Along the footsteps of \cite{Maggiani-Mora}, we have chosen to state our convergence and integrability conditions for the body forces and the Dirichlet data, on the rescaled level, only. To translate our requirements in terms of the unrescaled data $f_\eps$ and $w_\eps$, we may adapt the ansatz that the unrescaled forces have themselves the structure
\[
f_\eps (x,t) = \mathrm{F}(x,\eps t), \quad 
w_\eps (x,t) = \mathrm{W}(x,\eps t)
\]
for some $\mathrm{F} \in W_\loc^{1,1}(0,\infty;L^2(\Omega;\R^3))$ and $\mathrm{W} \in 
W_\loc^{2,1}(0,\infty;H^1(\Omega;\R^3)){\cap}
W_\loc^{3,1}(0,\infty;L^2(\Omega;\R^3))$. 
 
 On the other hand, \eqref{convergences-f} \& \eqref{cond-rescal-data} seem to be necessary for the analysis and reflect the dynamical features of the model, just like in \cite{Maggiani-Mora}. 
\end{remark} \EEE
\begin{remark}[More on condition \eqref{further-cvg-w}]
\label{rk:expl-hyp-w}
\upshape
A few comments on  \eqref{further-cvg-w} are in order:
 if, \emph{in addition}, 
the elasticity tensor complies with condition \eqref{condX-1}, 
then by  Lemma \ref{lm:identif}  
we have that 
$ \bbM \epl(\sfw)  = e(\sfw)$.  In this setting, \EEE a sufficient condition for \eqref{further-cvg-w} is that the loads  $(\rese w{\eps_k})_k$ are themselves Kirchhoff-Love, i.e.\ 
$ \rese w{\eps_k} \in  W_\loc^{1,2}(0,\infty; \KL)$. Then, from  the convergence $\rese w{\eps_k} \to \sfw $  in $    W_\loc^{1,2}(0,\infty;  H^1(\Omega;\R^3)) \EEE  $ 
we have that 
\[
 e^{\eps_k}(\rese w{\eps_k})  =
 \begin{pNiceArray}{cw{r}{0.7cm}c}[margin]
\Block{2-2}{\epl(\rese w{\eps_k})} & &  0 \\
& &  0 \\
0 & 0& 0 
\end{pNiceArray} \EEE
\, \longrightarrow \, e(\sfw) = \bbM \epl(\sfw) 
\]
 in $W_\loc^{1,2}(0,\infty;L^2(\Omega;\R^{3\times 3}))$.
 Nevertheless, we emphasize that, in Theorem 
 \ref{mainthm-1} we shall not require 
 condition  \eqref{condX-1}  for the elasticity tensor. \EEE
\end{remark}
\EEE
\par
 Finally,  in  Hypothesis  \ref{Hyp-C}  
 we impose suitable convergence conditions for  the initial data 
 $(\rese u{\eps_k}_0, \rese z{\eps_k}_0)_k$
 (since  we have dropped inertia in the momentum balance, it is not relevant to consider a sequence of initial velocities $(\resed u{\eps_k}_0)_k$). \EEE
 We mention in advance that the requirements on the limit $\sfz_0 $ of the sequence  $(\rese z{\eps_k}_0)_k$ are formulated in such a way 
 as to encompass both the case $\sfz_0 \in  \SBV(\GC; \{0,1\}) $, and the case in which $\sfz_0$ is just in $L^\infty(\GC)$, 
 even though with Hyp.\ \ref{Hyp-B} we clearly  envisage the presence of the $\SBV$-regularizing term in the 
 limiting system. Nonetheless, we 
 have chosen to formulate Hyp.\ \ref{Hyp-C} in a more flexible way 
   in view of the forthcoming Theorem  \ref{mainthm-2}, for which we will allow the  parameters $(\mathrm{b}_{\eps_k})_k$ to converge to
   $\mathrm{b}=0$, as well.  
    Furthermore,
  in accordance with the fact that the limiting displacement 
  $\sfu$ will satisfy $\sfu(t) \in   \KLGD$  \EEE for almost all $t\in (0,T)$, we will suppose that 
  $\sfu_0 \in   \KLGD$, \EEE  too. 
   \EEE
\begin{hypx}[{Initial data}]
\label{Hyp-C}
We suppose 
there exist   $(\sfu_0,\sfz_0) \in 
\Spu{\times} L^\infty(\GC)
$, 
with 
\begin{equation}
\label{KL+SBV}
\begin{cases}
\sfu_0 \in   \KLGD,\EEE 
\\
 \sfz_0  \in \SBV(\GC; \{0,1\})  \qquad \text{ if $\mathrm{b}>0$},
 \end{cases}
 \end{equation}
\EEE  fulfilling 
the semistability condition for $t=0$
\begin{equation}
\label{semistab-t-0}
\begin{aligned}
&
 \int_{\GC} \tfrac \kappa 2  \sfz_0 |\JUMP{\sfu_0}|^2 \dd \Surf(x) + 
\mathrm{b} P(\rese Z{}_0,\GC)  - \int_{\GC} a_0  \sfz_0 \dd \Surf(x)
\\
&
 \leq  \int_{\GC} \tfrac \kappa 2   \widetilde{z} |\JUMP{\sfu_0}|^2 \dd \Surf(x) 
+\mathrm{b} P(\widetilde{Z},\GC) -\int_{\GC}  a_0  \widetilde{z} \dd \Surf(x) 
+\int_{\GC} a_1|\widetilde{z}{-}\sfz_0| \dd \Surf(x) \EEE
\end{aligned}
 \end{equation}
 for all $ \widetilde{z} \in L^1(\GC)  $ with $\widetilde{z}\leq \sfz_0 $  a.e.\ in $\GC$,  ($\widetilde{z} \in \SBV(\GC;\{0,1\})$ if $\mathrm{b}>0$),
 and satisfying 
 \begin{subequations}
 \label{convergences-initial-data}
 \begin{align}
 &
  \label{convergences-initial-data-a}
 \rese u{\eps_k}_0 \weakto \rese u{}_0  \text{  in }   \Spu \EEE ,  \qquad  \rese z{\eps_k}_0 \weaksto \rese z{}_0   \text { in } L^\infty(\GC), 
 \intertext{(with $\rese z{\eps_k}_0 \weakto \rese z{}_0  $ in $\SBV (\GC;\{0,1\})$ if $\mathrm{b}>0$), and }
 &
  \label{convergences-initial-data-b}
 \rese  E{\eps_k}(0,\rese u{\eps_k}_0, \rese z{\eps_k}_0)\to   \rese E{} (0,\sfu_0, \sfz_0) \,,\EEE \quad 
 \end{align}
\end{subequations}
 where the energy functional $\rese E{}$ will be defined by the bulk and surface contributions in \eqref{defE-resc-limit} below.
\end{hypx}
\noindent \EEE
\par
Let us now introduce the
 limiting 
  $1$-homogeneous dissipation potential
     associated with the constant  $a_1$ from 
\eqref{parameters-below}:
\begin{align}
&
\label{R-resc-limit}
\rese R {}:\Spz\to[0,\infty]\,,\;
\qquad  \rese R{}(\dot \sfz):=\int_{\GC} \mathrm{R}(\dot \sfz)\,\mathrm{d}\Surf(x)\,,\;
\quad \mathrm{R}(\dot \sfz):=\left\{
\begin{array}{ll}
a_1|\dot \sfz|&\text{if }\dot \sfz\leq 0\,.\\
\infty&\text{otherwise}
\end{array}
\right.
\end{align}
(clearly, $\rese R{}$ coincides with the dissipation potential from  \eqref{defRk}, but here we are  using the \textsf{sans} font for  notational consistency). 
Finally, we will denote by $\rese E{} : [0,\infty) \times H_{\GD}^1(\Omega{\setminus}\GC;\R^3) \times L^1(\GC)  \to\R\cup\{\infty\}$ the energy functional 
given by $ \rese E{} =  \rese{E}{}_{\mathrm{bulk}} +  \rese{E}{}_{\mathrm{surf}} $,  where the bulk energy is 
 \begin{subequations}
 \label{defE-resc-limit}
\begin{align}
 \label{defE-resc-bulk-limit}
 &
 \rese{E}{}_\bulk  (t,\sfu) :=
 \int_{\Omega{\setminus}\GC}\tfrac{1}{2}\mathbb{C}_{\mathrm{r}} \, \epl(\sfu){:} \epl(\sfu) \dd x  - \pairing{}{\Spvr}{\rese F{} (t)}{\sfu} 
 \intertext{with 
 $\mathbb{C}_{\mathrm{r}}$ the \emph{reduced} elasticity tensor from \eqref{reduced-elasticity-tensor},  the operator $\epl$ from \eqref{e-plan}, and 
 $\rese F{}: [0,T]\to \Spu^*$ given by }
 &
 \nonumber
 \begin{aligned}
&
 \pairing{}{\Spvr}{\rese F {}(t)}{\sfu} :=   \int_{\Omega} \rese f{}(t) {\cdot} \sfu \dd x -
 \int_{\Omega{\setminus}\GC} \bbC_{\mathrm{r}}  \epl (\rese w{}(t)){:} \epl(\sfu) \dd x, 
\end{aligned}
\intertext{and the surface energy  is}
& 
 \label{defE-resc-surf-limit}
 \rese{E}{}_{\mathrm{surf}} (\sfu,\sfz) = 
\rese {J}{} (\sfu,\sfz)  + \int_{\GC} \left(I_{[0,1]}(\sfz){-}a_0 \sfz\right) \dd \Surf (x) +\mathrm{b}\calG(\sfz) \quad 
\text{with }
 \rese {J}{} (\sfu,\sfz) = \int_{\GC}  \tfrac \kappa 2   \sfz Q(\JUMP{\sfu})\dd \Surf(x)\,.
 \end{align}
\end{subequations}
Indeed, the definition of $\rese{E}{}_{\mathrm{surf}} $  reflects the fact that $\nu_{\eps_k} \to \nu=0$ by Hypothesis \ref{Hyp-B}. 
\par
 We are now in a position to state our first convergence result.  As previously mentioned, we will prove convergence to a Kirchhoff-Love displacement (cf.\ \eqref{KL-decomposition}): in particular, 
  the third component $\sfu_3$ will only depend on the planar variable $x'\in \omega$. 
\begin{maintheorem}
\label{mainthm-1}
Under Condition \ref{cond:1}, 
assume Hypotheses \ref{Hyp-B},
\ref{hyp:data}, 
  \ref{Hyp-C}. 
 Let 
$(\sfu_{\eps_k}, \sfz_{\eps_k})_k$ be a sequence of Balanced $\SE$ solutions 
to the  gradient systems $\grsysepsk_k$,  fulfilling \eqref{weak-mom-eps}--\eqref{EDB-eps}  with 
$\varrho_{\eps_k} 
\equiv 0$ for all $k\in \N$. 
\par
Then, there exist a (not relabeled) subsequence, and a pair $(\sfu,\sfz)$ with 
\begin{equation}
\label{reg-uz-DR}
\begin{aligned}
&
\sfu \in L_{\loc}^\infty (0,\infty; \Spu) 
  \\
  &  \text{ with } 
  \sfu(t) \in   \KLGD   \ \foraa t \in (0,\infty) \text{ and }   \sfu_3 \in  L_\loc^\infty(0,\infty;  H_{\dir}^2(\omega{\setminus}\gC\EEE))  \EEE
\\
&
\sfz \in L_{\loc}^\infty(0,\infty; \SBV(\GC;\{0,1\})) \cap  \BV_{\loc}(0, \infty ;L^1(\GC))\,, \EEE 
\end{aligned}
\end{equation}
such that the following convergences hold as $k\to\infty$:
\begin{equation}
\label{pointwise-convergences}
\begin{aligned}
& \rese u{\eps_k}(t)  \weakto \EEE  \sfu(t) &&  \text{in }  \Spu  && \foraa\, t   \in (0,\infty), \EEE
\\
&
\rese z{\eps_k}(t) \weakto \sfz(t)  &&  \text{in }   \SBV(\GC;\{0,1\}) && \text{for all } t   \in [0,\infty), \EEE
\end{aligned}
\end{equation}
and 
$(\sfu,\sfz)$ comply with
\begin{subequations}
\label{prob-form-Thm32}
\begin{enumerate}
\item
the momentum balance equation for almost all $t\in (0,\infty)$ 
\begin{equation}
\label{mom-bal-DR} 
\begin{aligned}
&
 \int_{\Omega{\setminus}\GC} \bbC_{\mathrm{r}} \epl(\rese u{}(t))  {:} \epl(\varphi)     \dd x   
+ \int_{\GC} \kappa\rese z{}(t)\JUMP{\rese u{}(t)}{\cdot}\JUMP{\varphi}\,\mathrm{d}\Surf (x) 
\\
&
=   \int_\Omega \rese f{}(t){\cdot}\varphi \dd x 
- \int_{\Omega{\setminus}\GC} \bbC_{\mathrm{r}} \epl(\rese w{}(t))  {:} \epl(\varphi)     \dd x   \EEE  \quad \text{for all } \varphi \in  \KLGD; 
\end{aligned}
\end{equation}
\item the semistability condition
 for almost all $t\in (0,\infty)$ and for $t=0$
\begin{equation}
\label{semistab-DR}
\rese E{}(t,\sfu(t),\sfz(t)) \leq \rese E{}(t,\sfu(t),\tilde z) +
\rese R{}(\tilde z {-}\sfz(t)) \qquad \text{for all } \tilde z \in \Spz ;
\end{equation}
\item 
 the energy-dissipation  inequality
\begin{equation}
\label{EDI-lim}
\begin{aligned}
\Var_{\mathsf{R}}(\sfz, [0,t])+
 \rese E{}(t,\sfu(t),\sfz(t))
  \leq
  \rese E{}(0,\sfu_0,\sfz_0) 
+ \int_0^t \partial_t
 \rese E{}(r,\sfu(r),\sfz(r)) \,  \mathrm{d}r 
 \end{aligned}
\end{equation}
 for almost all  $t \in (0,\infty)$.  
\end{enumerate}
\end{subequations}
\end{maintheorem}
 \begin{remark}[Analysis of problem formulation \eqref{prob-form-Thm32}]\label{rk:pb1}\EEE \upshape In fact, the pair $(\sfu,\sfz)$ from Theorem \ref{mainthm-1} is a $\SE$ solution of the 
 \emph{purely rate-independent}  
  adhesive contact system
$( \rese R{}, \rese E{})$,
with the 
$1$-homogeneous dissipation potential $\mathsf{R}$ 
from \eqref{R-resc-limit}, and the driving  energy $\mathsf{E}$ \eqref{defE-resc-limit}).
 Observe that, in the  rate-independent setup, the semistability condition \eqref{semistab-DR}, 
which rephrases as 
\begin{subequations}
\label{vorrei-ma-non-posso-ahime}
\begin{equation}
\label{minimality-4-z}
\sfz(t) \in \mathrm{Argmin} \big\{ \rese E{}(t,\sfu(t),\tilde z) +
\rese R{}(\tilde z {-}\sfz(t)) \,: \ \tilde z \in  \Spz\},
\end{equation}
is coupled to the \emph{static} momentum balance, which is equivalent, by convexity of $\rese E{} (t,\cdot,z)$, to the minimality condition
\begin{equation}
\label{minimality-4-u}
\sfu (t) \in \mathrm{Argmin} \big\{ \rese E{}(t,\tilde u,\sfz(t)) \,: \ \tilde u \in   \KLGD \}\,.
\end{equation}
\end{subequations}
Still, it is easy to realize that \eqref{vorrei-ma-non-posso-ahime} are not sufficient to obtain the full stability condition
\[
\rese E{}(t,\sfu(t),\sfz(t)) \leq \rese E{}(t,\tilde u,\tilde z) +
\rese R{}(\tilde z {-}\sfz(t)) \qquad \text{for all } \tilde u \in  \KLGD, \   \tilde z \in \Spz. 
\]
We refer to 
\cite{RoThPa15SDLS} for an analysis of semistable Energetic solutions for the purely rate-independent evolution of brittle delamination.
\par
We emphasize that, unlike in the standard definition (cf.\ also Def.\ \ref{def:energetic-sol1}), 
in Theorem \ref{mainthm-1} we  obtain the semistability condition and the energy-dissipation inequality for almost all $t\in (0,\infty)$, and analogously for the momentum balance.
Nonetheless,
since \eqref{mom-bal-DR}  is equivalent to \eqref{minimality-4-u}, and for fixed $t\in [0,T]$ the functional $u\mapsto \rese E{}(t,u, \sfz(t)) $ has a unique minimum,
exploiting the fact that $\sfz(t)$ is defined at \emph{every} $t$
  it would be possible to extend $\sfu$ to a function defined on the whole of $[0,T]$ (and indeed continuous at the continuity points of the $\BV$ function $\mathsf{z}$),
   and satisfying \eqref{mom-bal-DR}  \emph{everywhere} in $[0,T]$. 
However, we would not be able to obtain \eqref{semistab-DR} and \eqref{EDI-lim} for all $t\in [0,T]$, cf.\ Remark \ref{rmk:why-we-cant} ahead.
\end{remark} 
 \begin{remark}[On the momentum balance  equation 
\eqref{mom-bal-DR}]
\label{rk:mom-balance-pb1}
\upshape Due to the Kirchhoff-Love structure of the test functions, 
\eqref{mom-bal-DR} 
cannot, a priori, be rewritten in terms of two separate momentum balance equations for the in-plane and out-of-plane displacements. This would only be possible in the case in which 
\begin{equation}
\label{eq:adh-add}
\int_{-\frac12}^{\frac12}x_3 \rese z{}(t, x',x_3) \dd x_3=0.
\end{equation}
To see this, recall the notation in \eqref{KL-decomposition} and consider test functions   $\varphi\in \KLGD$ \EEE  with $\varphi_3=0$, so that $\varphi=(\overline{\varphi}_1,\overline{\varphi}_2,0)$ with  $\overline{\varphi}: = (\overline{\varphi}_1,\overline{\varphi}_2 )  \in  H^1_{\dir}(\omega\setminus\gC;\R^2)$. \EEE
Recalling that $\Omega\setminus \GC=(\omega\setminus\gC)\times (-\tfrac12,\tfrac12)$, we find
\begin{align*}
&\int_{\Omega{\setminus}\GC} \bbC_{\mathrm{r}} (\epl(\rese u{}(t){+}\epl(\rese w{}(t)) ) {:} \epl(\varphi)     \dd x
\\
& =
\int_{\Omega{\setminus}\GC} \bbC_{\mathrm{r}} (e_{\rm plan} (\rese {\overline{u}}{}(t){+}\rese {\overline{w}}{}(t))\EEE {-}x_3(\nabla^2_{x'}\rese u{}_3(t){+}\nabla^2_{x'}\rese w{}_3(t)))  {:} \epl(\varphi)     \dd x \\
&\quad=\int_{\omega{\setminus}\gC} \bbC_{\mathrm{r}} e_{\rm plan}(\rese {\overline{u}}{}(t){+}\rese {\overline{w}}{}(t))  {:} \epl(\overline{\varphi})     \dd x' -\int_{\Omega\setminus \GC} x_3\bbC_{\mathrm{r}}(\nabla^2_{x'}\rese u{}_3(t){+}\nabla^2_{x'}\rese w{}_3(t)):e_{\rm plan}(\overline{\varphi})\dd x,
\end{align*}
 where $\nabla^2_{x'}$ denotes the Hessian with respect to the planar variable $x'=(x_1,x_2)$. \EEE
In particular, since $\rese u{}_3$ and $\rese w{}_3$ are independent of $x_3$, the latter term on the right-hand side of the above equation satisfies
\begin{equation*}
\begin{aligned}
&    \int_{\Omega\setminus \GC} x_3\bbC_{\mathrm{r}}(\nabla^2_{x'}\rese u{}_3(t){+}\nabla^2_{x'}\rese w{}_3(t)) {:} e_{\rm plan}(\overline{\varphi})\dd x
  \\
  &  =\left(\int_{-\tfrac12}^{\tfrac12}x_3 \dd x_3\right) \int_{\omega\setminus \gC} \bbC_{\mathrm{r}}(\nabla^2_{x'}\rese u{}_3(t){+}\nabla^2_{x'}\rese w{}_3(t))
  {:} e_{\rm plan}(\overline{\varphi})\dd x'=0,
  \end{aligned}
\end{equation*}
so that equation \eqref{mom-bal-DR} yields
\begin{subequations}
\begin{equation}
\label{ful-dec-1}
\begin{aligned}
&\int_{\omega{\setminus}\gC} \bbC_{\mathrm{r}} e_{\rm plan}(\rese {\overline{u}}{}(t){+}\rese {\overline{w}}{}(t))  {:} e_{\rm plan}(\overline{\varphi})     \dd x'  
+ \int_{\GC} \kappa\rese z{}(t)\JUMP{\rese {\overline{u}(t)}{}-x_3\nabla_{x'}{\rese u{}}_3{}(t)}
{\cdot} \JUMP{\overline{\varphi}}\,\mathrm{d}\Surf (x) 
\\
&
=  \int_\omega \rese {\overline{f}}{}_{\rm plan}(t)\cdot\overline{\varphi} \dd x'   \qquad \qquad \qquad \text{for all } \overline{\varphi} \in  H^1_{\dir}(\omega{\setminus}\gC;\R^2),
\end{aligned}
\end{equation}
where $\rese {\overline{f}}{}_{\rm plan}(t):=\left(\int_{-\frac12}^{\frac12} \rese f{}_1(t)\dd x_3, \int_{-\frac12}^{\frac12} \rese f{}_2(t)\dd x_3\right)$.\\

Analogously, choosing $\varphi\in  \KLGD$ \EEE with $\overline{\varphi}=0$, 
equation \eqref{mom-bal-DR} yields 
\begin{equation}
\label{ful-dec-2}
\begin{aligned}
& 
 \frac{1}{12}\int_{\omega\setminus \gC} \bbC_{\mathrm{r}}(\nabla^2_{x'}\rese u{}_3(t)+\nabla^2_{x'}\rese w{}_3(t)):\nabla^2_{x'}\varphi_3\dd x' \\&\qquad - \int_{\GC} \kappa\rese z{}(t)\JUMP{\rese {\overline{u}(t)}{} {-} x_3\nabla_{x'}{\rese u{}}_3{}(t)}\cdot x_3\JUMP{\nabla_{x'}\varphi_3}\,\mathrm{d}\Surf (x)+\int_{\GC} \kappa\rese z{}(t)\JUMP{\rese u{}_3(t)}\JUMP{\varphi_3}\,\mathrm{d}\Surf (x) 
\\
&
=  \int_\Omega \rese f{}_3(t)\varphi_3 \dd x  
- \int_\Omega (\rese f{}_1(t), \rese f{}_2(t)) {\cdot}  x_3 \nabla_{x'}\varphi_3 \dd x 
\end{aligned}
\end{equation}
\end{subequations}
for all  $\varphi_3 \in  {H^2_{\dir}(\omega{\setminus}\gamma_c})$.
 Under the additional assumption \eqref{eq:adh-add}, the two equations  \eqref{ful-dec-1} and \eqref{ful-dec-2} \EEE   would fully decouple.
In fact, using the Kirchhoff-Love structure of 
$\rese u{}$ and $\varphi$,  we would have 
\begin{equation*}
\int_{\GC} \kappa\rese z{}(t)\JUMP{\rese 
{\overline{u}(t)}{}{-} x_3\nabla_{x'}{\rese u{}}_3{}(t)}{\cdot} \JUMP{\overline{\varphi}}\,\mathrm{d}\Surf (x)= \int_{\GC} \kappa\rese z{}(t)\JUMP{\rese {\overline{u}(t)}{}}{\cdot} \JUMP{\overline{\varphi}}\,\mathrm{d}\Surf (x)
\end{equation*}
 in \eqref{ful-dec-1}  and
\begin{equation*}
    -\int_{\GC} \kappa\rese z{}(t)\JUMP{\rese {\overline{u}(t)}{}{-} x_3\nabla_{x'}{\rese u{}}_3{}(t)}{\cdot} x_3\JUMP{\nabla_{x'}\varphi_3}\,\mathrm{d}\Surf (x)=\int_{\GC} x_3^2\kappa\rese z{}(t)\JUMP{\nabla_{x'}{\rese u{}}_3{}(t)}{\cdot} \JUMP{\nabla_{x'}\varphi_3}\,\mathrm{d}\Surf (x)
\end{equation*}
 in \eqref{ful-dec-2}\,. 

In general, though, the map $\rese z{}$ will retain a nontrivial   dependence \EEE  on $x_3$, so that planar and out-of-plane components of the displacements will be subject to the coupled conditions above. This is a classical phenomenon arising for inelastic dimension reduction problems. We refer to \cite[Section 5]{DavoliMora15} for an analogous observation in the setting of perfect plasticity.
\end{remark}

\subsection{\RCO Our second convergence result:  retaining the damping}
\label{ss:mainthm-2}
We again consider a sequence  $(\eps_k)_k \subset (0,\infty)$  converging to zero as $k \to\infty$.
We will now be tackling the vanishing-thickness analysis in the following setup.
\begin{condition}
\label{cond:2}
The coefficients $(\varrho_{\eps_k})_k$ and the tensors $(\mathbb{D}_{\eps_k})_k$ satisfy
\begin{equation}
 \label{hypD-2}
\begin{aligned}
&
\exists\, \varrho>0 \, : \ \ 
\varrho_{\eps_k}\to \varrho, 
\\
&
\exists\,   \mathsf{D}  \in  \R^{3\times 3\times3\times 3}, \ \text{ symmetric  in the sense of \eqref{assCD}  and positive definite, such that } \ 
 \eps_k \bbD_{\eps_k} \to   \mathsf{D}  \,.
\end{aligned}
\end{equation}
\end{condition}
Clearly,   \eqref{hypD-2}
 ensures  that the momentum balance equation,   in the vanishing-thickness limit, retains its \emph{damped} character. Moreover, since
$ \eps_k \bbD_{\eps_k} \approx  \mathsf{D}$ with  $\mathsf{D}$ a \emph{positive definite} tensor, 
estimating the sequence
$ (\eps_k \bbD_{\eps_k} e^{\eps_k}(\resed u{\eps_k}) {:}  e^{\eps_k}(\resed u{\eps_k}))_k\EEE$ 
in some suitable space
we will gain a bound for $ (e^{\eps_k}(\resed u{\eps_k}))_k$,
 and thus  additional compactness information on the sequence $(\sfu_{\eps_k})_k$. Such compactness properties \EEE
  will be strong enough to compensate a lack of
estimates for $(\sfz_{\eps_k})_k$ in $\SBV(\GC;\{0,1\})$. Namely, we will allow the parameters $(\mathrm{b}_{\eps_k})_k$ to converge to $\mathrm{b}=0$: 
in particular, in this context we may have  $\mathrm{b}_{\eps_k} = 0$ from a certain  $\bar k$ on.  Furthermore, the  compactness properties of 
 $(\sfu_{\eps_k})_k$  will also allow us to handle the contribution to the  surface  energy functional penalizing the failure of the   non-interpenetration 
 constraint $\JUMP{\sfu}  \cdot n \geq 0 $. \EEE
  That is why, we will allow the sequence $(\nu_k)_k$ modulating that contribution (cf.\ \eqref{defEk-eps}) 
   to converge  to a constant $\nu$ which may be positive. \EEE
 All in all, in place of Hypothesis \ref{Hyp-B} we will assume the following.
\begin{hypx}[Material parameters]
\label{Hyp-D}
We suppose that
\begin{align}
\label{hypD-1}
\exists\, \lim_{k \to \infty}\mathrm{b}_{\eps_k}=\mathrm{b} \geq 0, \qquad \exists\, \lim_{k\to\infty}\nu_{\eps_k} = \nu\geq 0\EEE ;
 \end{align}
 for the sequences $(a_0^{\eps_k})_k$, $(a_1^{\eps_k})_k$,  we require \eqref{parameters-below}. 
\end{hypx}
\smallskip

\par
For technical reasons,  we will additionally need the following condition. 
\begin{hypx}[Elasticity and viscosity tensors]
\label{Hyp-E-new}
The  tensors $\bbC$ and $\mathsf{D}$ comply with  \eqref{condX-1}. Further, the matrix $\mathscr{D}_3$:=$(\mathscr{D}_3)_{ij}=\mathsf{D}_{i3j3}$ is invertible.

\end{hypx}
\noindent 
  Let us recall that 
condition \eqref{condX-1} 
(which indeed encompasses part of the material symmetry conditions required in \cite{FreParRouZan11}, cf.\ Remark \ref{comparisonFPRZ}) ensures that, whenever
$\bbC$ and $\mathsf{D}$ multiply `planar' tensors, the resulting tensor is again `planar'. The crucial role of this property \EEE will be manifest in Lemma \ref{l:useful-X} ahead.  The additional invertibility requirement
in Hypothesis \ref{Hyp-E-new}
 is closely related to the properties of a suitable time-dependent minimum problem, which we will introduce in Section \ref{s:proof-th2}. \EEE
\par
 Furthermore,  while our condition on the forces  $( \rese f{\eps_k})_k$ will stay the same as in Hyp.\ \ref{hyp:data},
 in the context of Condition \ref{cond:2} we  will 
  have to strengthen our conditions on the loadings $(\sfw^{\eps_k})_k$; \eqref{new-conds-4-w}
  below
   indeed reflect that, in  the limit problem inertia will be retained, as well. \EEE
 \begin{hypx}[External forces]
\label{Hyp-E}
The sequence $( \rese f{\eps_k})_k$ complies with 
\eqref{bounds-w}.
The loadings $\sfw^{\eps_k}$ are Kirchhoff-Love themselves, i.e.\ $(\sfw_{\eps_k})_k \subset  W_\loc^{2,2}(0,\infty; \KL)$,  
$(\sfw_{\eps_k})_k$ is bounded in   $W_\loc^{2,2}(0,\infty;  H^1(\Omega;\R^3))$, 
and enjoys the additional bounds
\begin{subequations}
\label{new-conds-4-w}
\begin{equation}
\label{new-bound-w-hypF}
\exists\, C_w'>0 \  \ \forall\, k \in \N \, : \quad
 \eps_k^2 \| \resei w{\eps_k}1\|_{W_\loc^{3,1}(0,\infty;L^2(\Omega))} +
 \eps_k^2\| \resei w{\eps_k}2\|_{W_\loc^{3,1}(0,\infty;L^2(\Omega))} 
 \leq C_w'. 
 \end{equation}
 Moreover, in addition to convergences \eqref{convergences-w}, there holds
 \begin{equation}
  \label{convergences-w-3}
 \resei  w{\eps_k}3\to \sfw_3  \quad    \text{ in } W_\loc^{3,1}(0,\infty;L^2(\omega) .
 \end{equation}
 \end{subequations}
\end{hypx}
\noindent Thanks to this assumption, we clearly have $e^{\eps_k}(\sfw^{\eps_k})  = e(\sfw^{\eps_k}) $, thus, in view of the  first of \eqref{convergences-w} we have 
\begin{equation}
\label{KL-loadings-needed}
e^{\eps_k}(\sfw^{\eps_k}) \weakto e(\sfw) \quad \text{in } W_\loc^{2,2}(0,\infty;L^2(\Omega;\R^{3\times 3}))\,.
\end{equation}
\par
 Finally, let us complement the conditions on the initial data from Hyp. \ref{Hyp-C}, with a requirement on the initial velocities.
 \begin{hypx}[Initial data]
\label{Hyp-G}
The sequences $( \rese u{\eps_k}_0)_k$ and $ (\rese z{\eps_k}_0)_k$ comply with Hypothesis \ref{Hyp-C}.
Moreover,  there exists $\resed u{}_0 \in L^2(\Omega;\R^3)$ such that 
\begin{equation}
\label{e:initial-velocities}
 \resed u{\eps_k}_0 \to \resed u{}_0  \text{  in } L^2(\Omega;\R^3)\,.
\end{equation}
\end{hypx}

\par
 Under the scaling  prescribed by   Condition \ref{cond:2}  we will prove convergence of a (sub)sequence of \emph{Balanced} $\SE$ solutions 
of  the damped inertial systems $\ingrsysrepsk$ to a solution of   the   damped inertial system 
$\ingrsyslim$, with  the functionals  $\ingrsyslim$      (the subscript  {\footnotesize \textsc{VE}}  stands for 'viscoelastic') 
 specified in the following lines. Indeed,  kinetic energy will  be   given by
\begin{equation}
\label{kinetic-energy-lim}
\rese K{}(\dot{\sfu}): = \int_\Omega  \tfrac{\varrho}{2} |\dot{\sfu}_3|^2 \dd x
 = \int_\omega  \tfrac{\varrho}{2} |\dot{\sfu}_3|^2 \dd x'.
 \end{equation}
Note that,   in fact,  the integral is over $\omega$ because  the corresponding displacement $\sfu$ is Kirchhoff-Love and thus its third component only depends on the variable $x'
\in \omega$. 
The $1$-homogeneous dissipation $\rese R{}$ is given by  \eqref{R-resc-limit}. 
A key feature  of $\ingrsyslim$ will be  that  \EEE
 the viscous and the elastic bilinear forms in the
weak momentum balance only involve the planar minors $\epl(\sfu)$ and $\epl(\varphi)$  of the displacement and of the test function, like for the undamped momentum balance equation  \eqref{mom-bal-DR}. 
   Accordingly, the  dissipation due to damping will be encoded by the functional 
 \begin{subequations}
 \label{VE-diss}
\begin{equation}
\label{VE-diss-1}
 \Veve \colon H_{\GD}^1(\Omega{\setminus}\GC;\R^3) \to [0,\infty)\,,\; 
\Veve (\dot{\sfu}) : =  \int_{\Omega{\setminus}\GC} 
 \tfrac12  \mathsf{D} e(\dot \sfu){:}  e(\dot \sfu)  \dd x =   \int_{\Omega{\setminus}\GC} 
 \tfrac12  \Dr \epl(\dot \sfu){:}  \epl(\dot \sfu)  \dd x 
\end{equation}
where, for notational consistency with  Sec.\ \ref{ss:Thm1}, we have used the \emph{reduced viscosity tensor} 
 $\Dr:   \R_\sym^{2 \times 2} 
\to \R_\sym^{3\times 3} $ given by
 \begin{equation}
 \label{reduced-viscosity-tensor}
 \Dr\Xi: = \mathsf{D} 
 \,\begin{pNiceArray}{cw{r}{0.1cm}c}[margin]
\Block{2-2}{\Xi} & &  0 \\
& &  0 \\
0 & 0& 0 
\end{pNiceArray} \EEE
\qquad\text{for all } \Xi \in \R_\sym^{2 \times 2} \,.\EEE
 \end{equation}
\end{subequations}
Likewise, throughout this section, with slight abuse we will 
stick with the notation $\Cr$ for the 
\emph{reduced elasticity tensor} 
 $\Cr:   \R_\sym^{2 \times 2} 
\to \R_\sym^{3\times 3} $ given by
 \begin{equation}
 \label{reduced-elasticity-tensor-stick}
 \Cr\Xi: = \bbC
  \,\begin{pNiceArray}{cw{r}{0.1cm}c}[margin]
\Block{2-2}{\Xi} & &  0 \\
& &  0 \\
0 & 0& 0 
\end{pNiceArray} \EEE
\qquad\text{for all } \Xi \in \R_\sym^{2 \times 2} \,.\EEE
 \end{equation}
 Hence, 
the  energy functional  $\Eve : [0,\infty) \times H_{\GD}^1(\Omega{\setminus}\GC;\R^3) \times L^1(\GC)  \to\R\cup\{\infty\}$
  driving the evolution of the limit system
   will be given by 
 $ \Eve =  \Eve^{\mathrm{bulk}} +  \Eve^{\mathrm{surf}}$,  where the bulk energy is given by \EEE  
 \smallskip
 \begin{subequations}
 \label{defEve-resc-limit}
\begin{align}
 \label{defEve-resc-bulk-limit}
 &
 \Eve^\bulk  (t,\sfu) :=
 \int_{\Omega{\setminus}\GC}\tfrac{1}{2}\Cr \epl(\sfu) \epl(\sfu) \dd x  - \pairing{}{\Spvr}{\rese F{} (t)}{\sfu} 
 \intertext{with 
 $\rese F{} \colon [0,T] \to \Spu^*$ defined  by }
 &
 \nonumber
 \begin{aligned}
&
 \pairing{}{\Spvr}{\rese F {}(t)}{\sfu} :=   \int_{\Omega} \rese f{}(t)\cdot\sfu \dd x -
 \int_{\Omega{\setminus}\GC} \Cr \epl(\sfw){:} \epl(\sfu) \dd x 
 -\int_{\Omega{\setminus}\GC} \Dr \epl (\dot \sfw){:}  \epl(\sfu)  \dd x 
 \\ & \qquad \qquad \qquad \qquad \qquad \qquad  
- \int_\omega \varrho \reseidd w{} 3(t)  \sfu_3 \dd x'
\end{aligned}
\intertext{and the surface energy is given by}
& 
 \label{defEve-resc-surf-limit}
 \Eve^{\surf}(\sfu,\sfz) = 
\rese{H}{}(\sfu) +\rese {J}{} (\sfu,\sfz)  + \int_{\GC} \left(I_{[0,1]}(\sfz){-}a_0 \sfz\right) \dd \Surf (x) +\mathrm{b}\calG(\sfz)  \qquad \text{with} 
\\
& \nonumber
\begin{cases}  \rese{H}{}(\sfu) = \nu  \int_{\GC}\widehat{\alpha}_\lambda(\JUMP{\sfu_1,\sfu_2,0}) \dd \Surf(x)\,,
\smallskip
\\
 \rese {J}{} (\sfu,\sfz) = \int_{\GC}  \tfrac \kappa 2   \sfz Q({\JUMP{\sfu}})\dd \Surf(x)\,,
 \end{cases}
 \end{align}
  where now $\mathrm{b}\geq 0$ and  $\nu\geq 0$  in accordance with Hypothesis \ref{Hyp-D}.
 \end{subequations}

We are now in a position to state the second main result of the paper.
\begin{maintheorem}
\label{mainthm-2}
Under Condition \ref{cond:2}, assume Hypotheses \ref{Hyp-D}, \ref{Hyp-E-new}, \ref{Hyp-E}, and \ref{Hyp-G}, 
 with 
 \begin{equation}
\label{convergences-initial-data-34}
 \rese  E{\eps_k}(0,\rese u{\eps_k}_0, \rese z{\eps_k}_0)\to   \Eve (0,\sfu_0,  \sfz_0)\,.
 \end{equation}
 Let 
$(\sfu_{\eps_k}, \sfz_{\eps_k})_k$ a sequence of Balanced $\SE$ solutions 
to the damped inertial systems $\ingrsysrepsk_k$.
\par
Then, there exist a (not relabeled) subsequence, and a pair  $(\sfu, \sfz)$ with \EEE
\begin{equation}
\label{reg-uz-DR-2}
\begin{aligned}
&
\sfu \in H_{\loc}^1(0,\infty; \Spu) \text{ with } \sfu(t), \,  \dot \sfu(t)  \in  \KLGD \EEE    \ \foraa t \in (0,\infty) 
\text{ and }
\\
& 
 \sfu_3 \in H_\loc^1(0,\infty;H_{\dir}^2(\omega{\setminus}\gC)) \cap   W_\loc^{1,\infty} (0,\infty;L^2(\omega)), 
\\
&
\sfz \in L_{\loc}^\infty(0,\infty; L^\infty(\GC)) \cap  \BV_{\loc}(0, \infty ;L^1(\GC)), \EEE 
\end{aligned}
\end{equation}
such that the following convergences hold as $k\to\infty$:
\begin{equation}
\label{ptwse-cvgs}
\begin{aligned}
& \rese u{\eps_k}(t)  \BREV \weakto \EEE \sfu(t) &&  \text{in }  \Spu  && \text{for all }  t \in [0,\infty), \EEE 
\\
&
\rese z{\eps_k}(t) \weaksto \sfz(t)  &&  \text{in }   L^\infty(\GC) && \text{for all }   t \in [0,\infty), \EEE
\end{aligned}
\end{equation}
and 
$(\sfu,\sfz)$ comply with
\begin{subequations}
\label{prob-form-Thm34}
\begin{enumerate}
\item
the momentum balance equation 
\begin{equation}
\label{mom-bal-DR-34}
\begin{aligned}
&
\ -\int_0^t \int_\omega \varrho \reseid u{}3 \dot{\varphi}_3 \dd x' \dd r +\int_\omega \varrho \reseid u{}3(t) \varphi_3(t)  \dd x'
-\int_\omega \varrho  \resed u{}_{0,3} \varphi_3(0)  \dd x'
\\
&
 +\int_0^t  \int_{\Omega{\setminus}\GC} \Dr \epl(\resed u{})  {:} \epl(\varphi)     \dd x   \dd r
+\int_0^t  \int_{\Omega{\setminus}\GC} \bbC_{\mathrm{r}} \epl(\rese u{})  {:} \epl(\varphi)     \dd x  \dd r 
\\
&
 +   \nu  \int_0^t  \int_{\GC} \alpha_\lambda(\JUMP{\rese u{}_1, \rese u{}_2, 0}){\cdot} \JUMP{\varphi_1,\varphi_2,0} \dd \Surf(x) \dd r 
+   \int_0^t \int_{\GC} \kappa\rese z{}\JUMP{\rese u{}}\cdot\JUMP{\varphi}\,\mathrm{d}\Surf (x) \dd r 
\\
&
 =   \int_0^t \int_\Omega \rese f{}{\cdot}\varphi \dd x \dd r  - \int_0^t  \langle  \varrho  \reseidd w{}3,  \varphi_3  \rangle_{H^1(\omega{\setminus}\gC)} \dd r 
 - \int_0^t \int_{\Omega{\setminus}\GC} \Dr\epl(\resed w{})  {:} \epl(\varphi)     \dd x \dd r 
 \\ &  
 - \int_0^t  \int_{\Omega{\setminus}\GC} \bbC_{\mathrm{r}} \epl(\rese w{})  {:} \epl(\varphi)     \dd x  \dd r     \quad \text{for all }
\varphi \in\mathfrak{V} \text{ and almost all  } t\in (0,\infty)
 \\
 &  \qquad \qquad \text{with }
\mathfrak{V} : = \{ \varphi \in   L_\loc^2(0,\infty; \KLGD) \cap  W_\loc^{1,1}(0,\infty; L^2(\Omega;\R^3)) \}\,;
\end{aligned}
\end{equation}
\item the semistability condition
 for all $t\in [0,\infty)$
\begin{equation}
\label{semistab-DR-34}
\Eve(t,\sfu(t),\sfz(t)) \leq \Eve(t,\sfu(t),\tilde z) +
\rese R{}(\tilde z {-}\sfz(t)) \qquad \text{for all } \tilde z \in \Spz,
\end{equation}
featuring the $1$-homogeneous dissipation potential $\rese R{}$ from \eqref{R-resc-limit};
\item 
 the energy-dissipation 
 inequality  
\begin{equation}
\label{EDI-lim-43}
\begin{aligned}
 & 
 \rese K{}(\dot{\sfu}(t)) 
+ \int_0^t \int_{\Omega{\setminus}\GC} \Dr \epl(\resed u{})  {:} \epl(\resed u{})
\dd x  \dd r
 + \Var_{\mathsf{R}}(\sfz, [0,t])
 \\
 & \quad 
+ \Eve(t,\sfu(t),\sfz(t))     \leq  
 \rese K{}(\dot{\sfu}(0))
+    \Eve(0,\sfu(0),\sfz(0)) 
+ \int_0^t \partial_t
  \Eve(r,\sfu(r),\sfz(r)) \,  \mathrm{d}r
 \end{aligned}
\end{equation}
 for almost all    $t  \in (0,\infty)$.  
%
\end{enumerate}
\end{subequations}
\end{maintheorem}
\begin{remark}[Analysis of problem formulation \eqref{prob-form-Thm34}]
\label{rmk:despite}
\upshape In the present  setup,   the limit pair $(\sfu,\sfz)$  is a  
 $\SE$ solution  of the \emph{damped} adhesive contact system
$(\rese K{}, \Veve,\rese R{}, \Eve)$.
 The weak  formulation of the momentum balance reflects the lack of estimates for the second derivative $\reseidd u{}3$, hence the term 
$ \int_0^t  \langle  \varrho  \reseidd u{}3 ,  \varphi_3  \rangle_{H^1(\omega{\setminus}\gC)} \dd r $ needs to be integrated by part. In turn, 
 we emphasize  that,  unlike in Theorem \ref{mainthm-1} \EEE, 
 here we have succeeded in proving the semi-stability condition \emph{for all}  $t\in [0,\infty)$. 
 \par
  Nonetheless, even in this setup we have not succeeded in obtaining the energy-dissipation \emph{balance}. The latter would have stemmed 
 from testing the momentum balance equation by $\dot \sfu$ (cf.\  Remark \ref{rmk:failure-ENID}), which is however not admissible because  we do not have 
 $\dot \sfu \in \mathfrak{V}$ by lack of time regularity of $\dot{\sfu}$.  For the same reason, testing 
  the momentum balance by difference quotients does not seem to lead to the desired result, either. 
 \EEE
\end{remark}
 \begin{remark}
\label{rk:mom-balance-pb2}
\upshape As in Remark \ref{rk:mom-balance-pb1}, 
we emphasize 
that the  limit momentum balance \eqref{mom-bal-DR-34} is again characterized by a non-trivial coupling between the in-plane and out-of-plane components of the displacement $\rese u{}$. 
 Moreover, in this case only  \EEE
a \emph{partial} decoupling of in-plane and out-of-plane contributions would  be possible   under the additional structure condition \eqref{eq:adh-add}. 
\end{remark}

\section{Preliminary estimates}
\label{s:prelim-ests}
The main result of this section, Proposition \ref{prop:a-PRIO},   collects the a priori estimates available for the sequences $(\rese u{\eps_k})_k$ and $(\rese z{\eps_k})_k$. 
It is 
 formulated in such a way as to lay the ground for the compactness arguments both for Thm.\ \ref{mainthm-1} and for Thm.\ \ref{mainthm-2}.
\par
 First of all, we show that the energy functionals $(\rese E{\eps_k})_k $ satisfy, on all (bounded) sub-intervals of
 $[0,\infty)$,
  the 
 analogues of the \RCO coercivity and power-control estimates \EEE  \eqref{est-lemma}, uniformly w.r.t.\ $k \in \N$.
%
\begin{lemma}
\label{l:coercivity}
Let 
$(\mathrm{b}_{\eps_k})_k$ and $(\nu_{\eps_k})_k$ be sequences in $[0,+\infty)$ and  $(\rese f{\eps_k})_k$ comply with  \eqref{convergences-f}.
 Suppose that the  sequences $(\varrho_{\eps_k})_k \subset [0,+\infty)$,  $(\bbD_{\eps_k})_k\subset \R^{3{\times}3{\times}3{\times}3}$  comply with 
\begin{enumerate}
    \item   Condition \ref{cond:1}: in that case, let $(\rese w{\eps_k})_k$  satisfy Hypothesis \ref{hyp:data}; 
\item
Condition \ref{cond:2}: in that case, let $(\rese w{\eps_k})_k$  satisfy Hypothesis  \ref{Hyp-E}. 
\end{enumerate} 
 Then,
\begin{subequations}
\label{est-lemma-unif}
\begin{align}
&
\label{est-coer-E-eps}
\begin{aligned}
&
\forall\, \Temp>0  \ \exists\, c_{\Temp},\,  c_{\Temp}',\,  C_{\Temp}>0 \ \forall\, k \in \N \ 
 \forall\, (t,\sfu,\sfz) \in [0,\Temp]\times H_{\Dir}^1(\Omega{\setminus}\GC;\R^3)  {\times}  L^\infty (\GC;\{0,1\})  :  \EEE
\\
& 
\ \rese E{\eps_k}\EEE(t,\sfu,\sfz)   \geq c_{\Temp} \left(\|e^{\eps_k}(\sfu)\|_{L^2(\Omega;\R^{3\times 3})}^2 {+}  \mathrm{b} \EEE \|\sfz\|_{\SBV(\GC)} \right) -C_{\Temp}
 \geq c_{\Temp}' \left(\RNEW \|\sfu\|_{H_{\Dir}^1(\Omega{\setminus}\GC)}^2 \EEE {+}  \mathrm{b} \EEE \|\sfz\|_{\SBV(\GC)} \right) -C_{\Temp},
\end{aligned}\\
&
\label{Gronwall-estimate-eps}
\begin{aligned}
&
\forall\, \Temp>0  \ \exists\, \ell_{\Temp}\EEE
\in L^1(0,\Temp)  \ \forall\, k \in \N   \  \foraa t \in  (0,\Temp) \  \forall\, (u,z) \in   H_{\Dir}^1(\Omega{\setminus}\GC;\R^3)  {\times} L^\infty (\GC;\{0,1\}) :  \EEE
\\
 & \  |\partial_t \rese E{\eps_k}(t,\sfu,\sfz) | \leq |\ell_{\Temp}(t)\EEE|  \left(  \rese E{\eps_k}(t,\sfu,\sfz) {+}1\right)\,.
 \end{aligned}
\end{align}
\end{subequations}
\end{lemma}
\begin{proof}
We adapt the calculations from the proof of Lemma \ref{l:properties-en}. Indeed, by Hypothesis \ref{hyp:data} we find that 
\[
\exists\, C_{\Temp}>0 \ \ \forall\, k \in \N \  \forall\, (t,\sfu) \in [0,\Temp]\times H_{\Dir}^1(\Omega{\setminus}\GC;\R^3)  : 
\ \  |\pairing{}{H^1(\Omega{\setminus}\GC)}{\rese F{\eps_k}(t)}{\sfu} | \leq C_T  (\| \sfu\|_{H^1(\Omega{\setminus}\GC)} {+}   \| e^{\eps_k}(\sfu) \|_{L^2(\Omega)})\,. \EEE
\]
On the other hand,  
by the positive-definiteness of $\bbC$, \RCO the definition of $e^\eps$, \EEE and  Korn's inequality
we infer \EEE 
\begin{equation}
\label{due-to-Korn}
\int_{\Omega{\setminus}\GC} \tfrac12 \bbC e^{\eps_k}(\sfu){:} e^{\eps_k}(\sfu) \dd x  \geq   c' \|  e^{\eps_k}(\sfu) \|_{L^2(\Omega)}^2 \geq c \RNEW \|\sfu \|_{H^1(\Omega{\setminus}\GC)}^2\,.\EEE
\end{equation}
Combining these two estimates we deduce \EEE \eqref{est-coer-E-eps}.
\par
 Inequality \EEE \eqref{Gronwall-estimate-eps} follows upon observing that for every $(\sfu,\sfz) \in  H_{\Dir}^1(\Omega{\setminus}\GC;\R^3)  \times\SBV(\GC;\{0,1\})$
  and for almost all $t\in (0,T)$ there holds \EEE
\[
\begin{aligned}
&
|\partial_t \rese E{\eps_k}(t,\sfu,\sfz)| 
\\
& \leq C\Big(\|\resed f{\eps_k} (t)\|_{L^2(\Omega)}{+} \|e^{\eps_k} (\resed w{\eps_k}(t))\|_{L^2(\Omega)}{+} 
  \eps_k \| \mathbb{D}_{\eps_k} e^{\eps_k} (\resedd w{\eps_k}(t))\|_{L^2(\Omega)} \EEE
\\
& \qquad \qquad \qquad 
{+}\eps_k^2\sum_{i=1}^2  \varrho_{\eps_k}  \|\reseiddd w{\eps_k} i(t)\|_{L^2(\Omega)} {+} \varrho_{\eps_k} \|\reseiddd w{\eps_k} 3(t)\|_{L^2(\Omega)}\Big)
 \times
 (\| \sfu\|_{H^1(\Omega{\setminus}\GC)} {+}   \| e^{\eps_k}(\sfu) \|_{L^2(\Omega)})\,. \EEE
\end{aligned}
\] 
Now,  the first term on the right-hand side can be controlled via \eqref{convergences-f}, while the second term is estimated by the first of 
 \eqref{bounds-w}. 
 Under Condition \ref{cond:1}, we estimate 
 \[
 \| \eps_k \mathbb{D}_{\eps_k} e^{\eps_k} (\resedd w{\eps_k}(t))\|_{L^2(\Omega)} \leq \eps^{\beta} |\mathbb{D}_{\eps_k}| \eps^{1-\beta}\| e^{\eps_k} (\resedd w{\eps_k}(t))\|_{L^2(\Omega)}  \longrightarrow 0 
 \]
 thanks to   \eqref{bounds-w}. Alternatively, if we only have that $(\eps_k \mathbb{D}_{\eps_k})_k$ is bounded as under Condition \ref{cond:2}, we need to resort to
 Hypothesis 
 \ref{Hyp-E} for $(\rese w{\eps_k})_k$. 
 Finally, the  terms premultiplied by
 $\varrho_{\eps_k}$
   are non-null only in the setup of Condition \ref{cond:2}, when also 
  Hypothesis  \ref{Hyp-E} is in force.  \EEE
\end{proof}

 In view of the previous lemma, we establish some a priori bounds for the sequence $(\rese u{\eps_k},\rese z{\eps_k})_k$, 
 that will be valid both for Thm.\ \ref{mainthm-1} and for Thm.\ \ref{mainthm-2}. 
 \EEE

\begin{proposition}[A priori estimates]
\label{prop:a-PRIO}
 Under the assumptions of Lemma \ref{l:coercivity}, 
  suppose in addition that the initial data  either  comply with  Hyp.\ \ref{Hyp-C} 
 under Condition \ref{cond:1}, 
  or 
  with Hyp.\ \ref{Hyp-G} under Condition
 \ref{cond:2}. 
\par
Then, \EEE
there exists a constant $C>0$ such that the following bounds 
hold uniformly w.r.t.\ $k\in \N$:
\begin{subequations}
\label{est-rescaled-sols}
\begin{align}
&
\label{est-rescaled-sols-1}
  \eps_k\|  \varrho_{\eps_k}^{1/2}     \reseid u{\eps_k}1\|_{L_\loc^\infty(0,\infty;L^2(\Omega))} + 
   \eps_k \|   \varrho_{\eps_k}^{1/2}      \reseid u{\eps_k}2\|_{L_\loc^\infty(0,\infty;L^2(\Omega))}+
 \|   \varrho_{\eps_k}^{1/2}     \reseid u{\eps_k}3\|_{L_\loc^\infty(0,\infty;L^2(\Omega))} \leq C,
\\
& 
\label{est-rescaled-sols-2}
\| \rese u{\eps_k}\|_{L_\loc^\infty(0,\infty;\RNEW H_{\Dir}^1(\Omega{\setminus}\GC) \EEE)} \leq C,
\\
& 
\label{est-rescaled-sols-3}
 \frac1{\eps_k^2}\left\|\partial_3 \rese u{\eps_k}_3\right\|_{L_\loc^\infty(0,\infty;L^2(\Omega))} \leq C,
\\
& 
\label{est-rescaled-sols-4}
 \frac1{{\eps_k}} \left\| \partial_1 \rese u{\eps_k}_3{+}\partial_3 \rese u{\eps_k}_1 \right\|_{L_\loc^\infty(0,\infty;L^2(\Omega))}
+ \frac1{{\eps_k}}\left\| \partial_2 \rese u{\eps_k}_3{+} \partial_3 \rese u{\eps_k}_2\right\|_{L_\loc^\infty(0,\infty;L^2(\Omega))} \leq C,
\\
& 
\label{est-rescaled-sols-4-bis}
\|e^{\eps_k}(\rese u{\eps_k})\|_{L_\loc^\infty(0,\infty;L^2(\Omega))}+
 {\eps_k} \|\bbD_{\eps_k} e^{\eps_k}(\resed u{\eps_k}) {:}  e^{\eps_k}(\resed u{\eps_k})\|_{L_\loc^1(0,\infty;L^1(\Omega))} \leq C, 
\\
& 
\label{est-rescaled-sols-5}
 \mathrm{b}\, \EEE \| \rese z{\eps_k}\|_{L_\loc^\infty(0,\infty;\SBV(\GC;\{0,1\})} \leq C\,.
\end{align}
Furthermore, 
\begin{equation}
\label{est-varR}
\forall\, \Temp>0 \ \exists\, C_{\Temp}'>0 \ \forall\, {\eps_k}>0 \, : \quad \Var_{\rese R{\eps_k}}(\rese z{\eps_k}; [0,\Temp]) \leq C_{\Temp}'.
\end{equation}
\end{subequations}
\end{proposition}
\begin{proof}
We mimick the arguments from the proof of Proposition \ref{prop:2.5}. Indeed, we start from the energy-dissipation balance \eqref{EDB-eps}: 
combining
\eqref{Gronwall-estimate-eps}  with the 
 Gronwall Lemma we obtain that 
\begin{equation}
\label{energy-bound-eps-rescal}
\forall\, \Temp>0 \ \exists\, C_{\Temp}'>0 \ \forall\, k\in \N \, :  \ \  \sup_{t\in [0,\Temp]}
 |\rese E{{\eps_k}}(t,\rese u{\eps_k}(t),\rese z{\eps_k}(t))|\leq C.
\end{equation}
On account of \eqref{est-coer-E-eps},
we then infer 
\[
\| e^{\eps_k}(\rese u{\eps_k})\|_{L_\loc^\infty(0,\Temp;L^2(\Omega))} \leq C,
\]
whence estimates \eqref{est-rescaled-sols-2} (due to \eqref{due-to-Korn}),
\eqref{est-rescaled-sols-3}, and 
\eqref{est-rescaled-sols-4}, as well as \eqref{est-rescaled-sols-5}. 
\par
Furthermore, again arguing as for Prop.\  \ref{prop:2.5}, from the bound for the kinetic energy and dissipation terms we deduce estimates
\eqref{est-rescaled-sols-1}, 
\eqref{est-rescaled-sols-4-bis}
 \eqref{est-varR}. 
 \end{proof}
\section{Proof of Theorem \ref{mainthm-1}}
\label{s:proof-th1}
 The proof is split into the following steps.
\medskip

\noindent
\textbf{Step $0$: compactness.}
It follows from estimates \eqref{est-rescaled-sols} 
and standard weak compactness results that there exists  \RNEW $\sfu \in L_\loc^\infty(0,\infty;H_{\Dir}^1 (\Omega{\setminus}\GC;\R^3))$  such that
\begin{subequations}
\begin{equation}
\label{convs-k-u1}
\rese u{\eps_k} \weaksto \sfu \text{ in } L_\loc^\infty(0,\infty;H_{\Dir}^1 (\Omega{\setminus}\GC;\R^3)).
\end{equation}
Now, by \eqref{est-rescaled-sols-3} and \eqref{est-rescaled-sols-4}, we have that 
\begin{equation}
\label{towards-KL-structure}
\partial_3 \rese u{\eps_k}_3 \to 0, \quad  \left( \partial_1 \rese u{\eps_k}_3{+}\partial_3 \rese u{\eps_k}_1 \right)  \to 0, \quad 
 \left( \RCO  \partial_3 \rese u{\eps_k}_2{+}  \EEE \partial_2 \rese u{\eps_k}_3 \right) \to 0 \qquad \text{ in }  L_\loc^\infty(0,\infty;L^2(\Omega)).
\end{equation}
Hence, we deduce that 
\[
\sfu(t) \in  \KLGD \EEE \qquad \foraa t  \in (0,+\infty)\,.
\]
Therefore, $\sfu$ admits  \EEE
the representation  \eqref{KL-decomposition} \EEE
with  two functions $\overline \sfu \in L_\loc^\infty(0,\infty; H_{\dir}^1(\omega{\setminus}\gC;\R^2))$ and $u_3 \in L_\loc^\infty(0,\infty; H_{\dir}^2(\omega{\setminus}\gC)) $. 
 By \eqref{est-rescaled-sols-4-bis}, \EEE there exists
$\sfe \in L_\loc^\infty (0,\infty;L^2(\Omega;\R^{3\times 3}))$ such that 
\begin{equation}
\label{conv-stress-u}
e^{\eps_k}(\rese u{\eps_k}) \weaksto \sfe \quad \text{ in } L_\loc^\infty (0,\infty;L^2(\Omega;\R^{3\times 3})) 
\end{equation}
and
 a triple $(\sfd_{13},\sfd_{23}, \sfd_{33})$ such that 
\begin{equation}
\label{conv-dii}
\begin{aligned}
&
\frac1{\eps_k} (e(\rese u{\eps_k}))_{i3} \weaksto \sfd_{i3}  &&  \text{ in } L_\loc^\infty (0,\infty;L^2(\Omega)) \qquad \text{for  } i=1,2,
\\
&
\frac1{\eps_k^2} (e(\rese u{\eps_k}))_{33} \weaksto \sfd_{33} &&  \text{ in } L_\loc^\infty (0,\infty;L^2(\Omega)),
\end{aligned}
\end{equation}
so that (recall the notation in Subsection \ref{subs:rescaled}) \EEE
\begin{equation}
\label{identification-eplan-u}
\sfe = 
 \,\begin{pNiceArray}{cw{r}{0.3cm}c}[margin]
\Block{2-2}{\epl(\sfu) } & &  \sfd_{13} \\
& &  \sfd_{23} \\
\sfd_{13} & \sfd_{23}& \sfd_{33}
\end{pNiceArray}\,. \EEE
%
\end{equation}
\par
As for $(\rese z{\eps_k})_k$, there exist  $\sfz\in L_\loc^\infty(0,\infty;\SBV(\GC;\{0,1\})) \cap \BV([0,T];L^1(\GC))$ such that, along a (not relabeled) subsequence, 
\begin{align}
&
\label{conv-k-z-1}
\rese z{\eps_k} \weaksto  \sfz && \text{in } L_\loc^\infty(0,\infty;\SBV(\GC;\{0,1\})),
\\
&
\label{convs-k-z-5}
 \rese z{\eps_k}(t)\weaksto \sfz(t) && \text{in }  \SBV(\GC;\{0,1\})  \EEE 
  \text{ for all } t \in [0,\infty),
\\
& 
\label{convs-k-z-6}
 \rese z{\eps_k}(t) \to \sfz(t) && \text{in } L^q(\GC) \text{ for every } 1\leq q <\infty \text{ and  for all } t \in [0,\infty).
\intertext{ \RCO  Thus, by the dominated convergence theorem we have}
& 
\label{convs-k-z-7}
 \RCO \ \rese z{\eps_k} \to \sfz &&   \RCO  \text{in } L^p_\loc(0,\infty;L^q(\GC) ) \text{ for every } p,\, q \in [1,\infty).
\end{align}
\end{subequations}
\EEE

\medskip

\noindent
\textbf{Step $1$: enhanced properties of $\sfu$.} 
 In this step, we will show that the tensor-valued function $\sfe$ from \eqref{conv-stress-u} satisfies
\begin{equation}
\label{toSHOW-MM}
 \sfe = \bbM  \epl(\sfu)  
\qquad \aein \  \Omega\times(0,\infty)\,, \EEE
\end{equation}
 where $\bbM$ is the operator introduced  in \eqref{def-oper-M}. \EEE
For \eqref{toSHOW-MM}, we borrow the argument from the proof of  \cite[Thm.\ 4.1]{Maggiani-Mora}.
 Fix  $(a,b)\subset \RCO  (-\tfrac 12,\tfrac 12)$ and an open set $U\subset\omega$. Let $(\ell_n)_n\subset  \RCO \mathrm{C}^1([-\tfrac 12,\tfrac 12])$, \EEE  and for  every
  $\zeta^i\in \R$,  $i=1,2,3$, \EEE
let $(\zeta_n^i)_n\subset \mathrm{C}^1_{\rm c}(\omega)$ fulfill
\begin{align*}s
     \ell'_n \EEE \to \chi_{(a,b)}\quad\text{strongly in }L^4(-\tfrac12,\tfrac12),\qquad 
    \zeta_n^i\to\zeta^i\chi_{U}\EEE\quad\text{strongly in }L^4(\omega),
\end{align*}
as $n\to +\infty$. 
Consider the maps
\begin{equation}
\label{test-psi-added}
\psi^{\eps_k}_n(x):=
\begin{pNiceArray}{c}
2\eps_k \zeta_n^1(x')\ell_n(x_3)\\
2\eps_k \zeta_n^2(x')\ell_n(x_3)\\
\eps_k^2 \zeta_n^3(x')\ell_n(x_3)\\
\end{pNiceArray} \quad\text{for a.a. }
 x=(x',x_3)\in \Omega\EEE
\end{equation}
as test functions in \eqref{weak-mom-eps}, and integrate the momentum equation  over a generic interval $[s,t] \subset [0,\infty)$. 
\EEE
Since for every $k, \, n \in \mathbb{N}$  we have  $\JUMP{\psi^{\eps_k}_n}\equiv 0$, 
 the fourth and fifth terms in \eqref{weak-mom-eps} are identically equal to zero.
  Therefore, we obtain
 \begin{equation}
 \label{lim-n-fixed}
 \begin{aligned}
&
\int_s^t  \int_{\Omega{\setminus}\GC}  \left( {\eps_k} \mathbb{D}_{\eps_k} e^{\eps_k}(\resed u{\eps_k}) {:} e^{\eps_k}( \psi^{\eps_k}_{n} ) 
{+} \bbC e^{\eps_k}(\rese u{\eps_k} )  {:} e^{\eps_k}( \psi^{\eps_k}_{n} )  \right) \dd x \dd r 
\\
&
= \int_s^t  \int_\Omega \rese f{\eps_k}\psi^{\eps_k}_{n} \dd x \dd r 
-  \int_s^t   \int_{\Omega\setminus \GC} \bbC e^{\eps_k}(\rese w{\eps_k}) : e^{\eps_k}(\psi^{\eps_k}_{n}) \dd x \dd r - {\eps_k}  \int_s^t   \int_{\Omega\setminus \GC} \bbD_{\eps_k}
 e^{\eps_k}(\resed w{\eps_k}) : e^{\eps_k}(\psi^{\eps_k}_{n}) \dd x \dd r \,.
\end{aligned}
\end{equation}
We now take the limit in \eqref{lim-n-fixed} 
 as $\eps_k \downarrow 0$, for fixed $n$. 
It is easy to check that 
$ \lim_{k\to\infty} \int_s^t  \int_\Omega \rese f{\eps_k}\psi^{\eps_k}_{n} \dd x \dd r =0 $, as well. 
In order to evaluate the terms involving $ e^{\eps_k}( \psi^{\eps_k}_{n} )  $, we explicitly compute it to find
\[
\begin{aligned}
&
e^{\eps_k}( \psi^{\eps_k}_{n})=  
 \,\begin{pNiceArray}{cw{r}{0.7cm}c}[margin]
\Block{2-2}{ \epl( \psi^{\eps_k}_{n}) } & & \mathfrak{e}_{13}^{n,k}   \\
& & \mathfrak{e}_{23}^{n,k}   \\
\mathfrak{e}_{13}^{n,k}  & \mathfrak{e}_{23}^{n,k} & \mathfrak{e}_{33}^{n,k} 
\end{pNiceArray} \EEE
\qquad \text{with }
\medskip
\\
& 
\begin{cases}
&
\displaystyle
\epl( \psi^{\eps_k}_{n})  = \eps_k \left( \begin{array}{lll} 
 & 2 \ell_n  \partial_{x_1} \zeta_n^1 &  \ell_n (\partial_{x_1} \zeta_n^2 {+}   \partial_{x_2}\zeta_n^1  )
 \smallskip
\\
 &  \ell_n (\partial_{x_1} \zeta_n^2 {+}   \partial_{x_2}\zeta_n^1  ) &  2 \ell_n  \partial_{x_2} \zeta_n^2 
 \end{array}
 \right) \,,
  \smallskip
 \\
 &
 \displaystyle
  \mathfrak{e}_{13}^{n,k} = \tfrac{\eps_k}2  \ell_n   \partial_{x_1} \zeta_n^3 + \zeta_n^1 \ell_n'\,,
   \smallskip
  \\
  & \displaystyle
  \mathfrak{e}_{23}^{n,k} = \tfrac{\eps_k}2  \ell_n   \partial_{x_2} \zeta_n^3 + \zeta_n^2 \ell_n' \,,
   \smallskip
  \\
  &  \displaystyle
 \mathfrak{e}_{33}^{n,k}  =  \zeta_n^3 \ell_n'  \,,
 \end{cases}
\end{aligned}
\]
whence we have, for fixed $n\in \N$, 
\begin{equation}
\label{e-psi-w}
e^{\eps_k}( \psi^{\eps_k}_{n}) \longrightarrow  
 \,\begin{pNiceArray}{cw{r}{0.5cm}c}[margin]
\Block{2-2}{ 0 } & &  \zeta^1_n  \ell_n'   \\
& &  \zeta^2_n  \ell_n'   \\
 \zeta^1_n  \ell_n' &  \zeta^2_n  \ell_n' &  \zeta^3_n  \ell_n' 
\end{pNiceArray} \EEE
\qquad \text{in } L^2(\Omega;\R^{3\times 3}) \quad \text{as $k\to \infty $.}
\end{equation}
Now,  for the third term on the right-hand side of \eqref{lim-n-fixed} we have 
\begin{equation}
\label{viscosity-tends-0}
\begin{aligned}
&
\lim_{k\to\infty}\int_s^t  \int_{\Omega{\setminus}\GC}   \eps_k \mathbb{D}_{\eps_k} e^{\eps_k}(\resed u{\eps_k}) {:} e^{\eps_k}( \psi^{\eps_k}_{n} )  \dd x 
\\
&
 =\lim_{k\to\infty} \int_s^t  \int_{\Omega{\setminus}\GC}     [\eps_k^{1/2} \mathbb{D}_{\eps_k}^{1/2} 
e^{\eps_k}(\resed u{\eps_k}) ] {:}   [\eps_k^{1/2} \mathbb{D}_{\eps_k}^{1/2}  e^{\eps_k}( \psi^{\eps_k}_{n} ) ] \dd x 
=0 \,,
\end{aligned}
\end{equation}
 (recall the definition of the square root of a tensor, cf.\ Remark \ref{rmk:sqrt}). 
 For \eqref{viscosity-tends-0}, 
  we have combined  that $\eps_k \mathbb{D}_{\eps_k} \to 0 $ by Condition \ref{cond:1} with \eqref{e-psi-w}  and the bound on 
$( \eps_k^{1/2} \mathbb{D}_{\eps_k}^{1/2}
e^{\eps_k}(\resed u{\eps_k}))_k$ in $ L^2_{\loc} (0,\infty;L^2(\Omega;\R^{3\times 3}))$ due to  \eqref{est-rescaled-sols-4-bis}.
From \eqref{e-psi-w} with \eqref{conv-stress-u}  it also follows \EEE 
\[
\lim_{k\to\infty}\int_s^t  \int_{\Omega{\setminus}\GC} \bbC e^{\eps_k}(\rese u{\eps_k} )  {:} e^{\eps_k}( \psi^{\eps_k}_{n} )   \dd x \dd r  = \int_s^t  \int_{\Omega{\setminus}\GC} \bbC  \sfe {:} 
 \,\begin{pNiceArray}{cw{r}{0.3cm}c}[margin]
\Block{2-2}{ 0 } & &  \zeta^1_n     \\
& &  \zeta^2_n     \\
 \zeta^1_n   &  \zeta^2_n   &  \zeta^3_n  
\end{pNiceArray} \, \EEE
\ell'_n \dd x \dd r=0.
\]
Analogously, we take the limit in the second and in the third limit on the right-hand side of \eqref{lim-n-fixed}, recalling
that  $e^{\eps_k}(\rese w{\eps_k}) \weakto \bbM \epl(\sfw) $ in $W_\loc^{1,2}(0,\infty;L^2(\Omega;\R^{3\times 3}))$ by 
\eqref{further-cvg-w} \EEE
and using that 
$ \eps_k \bbD_{\eps_k}
 e^{\eps_k}(\resed w{\eps_k}) \to 0$ in $L^2(0,T;L^2(\Omega;\R^{3\times 3}) )$ thanks to  \eqref{further-cvg-w}  and, again,  Condition \ref{cond:1}. 
 All in all, we have proven that 
\begin{equation}
\label{limit-passage-n}
\int_s^t \int_{\Omega\setminus\GC}\bbC 
(\sfe{+}  \bbM \epl(\sfw) \EEE ){:}
 \,\begin{pNiceArray}{cw{r}{0.3cm}c}[margin]
\Block{2-2}{ 0 } & &  \zeta^1_n     \\
& &  \zeta^2_n     \\
 \zeta^1_n   &  \zeta^2_n   &  \zeta^3_n  
\end{pNiceArray} \EEE \, 
\ell'_n \dd x \dd r=0.
\end{equation}
Then,  we take the limit of \eqref{limit-passage-n} as $n\to\infty$, obtaining \EEE
\[
\int_s^t \int_{U\times (a,b)}\bbC (\sfe{+}  \bbM \epl(\sfw) \EEE){:}
 \,\begin{pNiceArray}{cw{r}{0.3cm}c}[margin]
\Block{2-2}{ 0 } & &  \zeta^1     \\
& &  \zeta^2     \\
 \zeta^1   &  \zeta^2   &  \zeta^3  
\end{pNiceArray} \EEE \, 
 \dd x \dd r=0.
\]
Since the intervals  $(a, b)$, $[s, t]$, and the set $U$ are arbitrary, we deduce  the orthogonality relation \EEE
\begin{equation}
\label{final&crucial}
\bbC (\sfe {+} \bbM \epl(\sfw) \EEE){:}
 \,\begin{pNiceArray}{cw{r}{0.3cm}c}[margin]
\Block{2-2}{ 0 } & &  \zeta^1     \\
& &  \zeta^2     \\
 \zeta^1   &  \zeta^2   &  \zeta^3
\end{pNiceArray} \EEE \, =0 
 \qquad \text{for all } \zzeta=(\zeta^1,\zeta^2,\zeta^3) \in \R^3 \text{ and }  \aein\, (0,\infty)\times \Omega\,.
\end{equation}
Recalling the  characterization \eqref{properties-CM1}
 we then infer 
 $\sfe {+} \bbM \epl(\sfw)   = \bbM (\epl(\sfu){+}\epl(\sfw))$, whence \EEE
   \EEE
 \eqref{toSHOW-MM}. 
\medskip

\noindent
\textbf{Step $2$: limit passage in the weak momentum balance.}
We tackle here the limit passage in \eqref{weak-mom-eps},  integrated over a generic interval $[s,t]\subset [0,\infty)$, 
 by confining the discussion to test functions
$\varphi \in   \KLGD$,
so that 
\begin{equation}
\label{struct-eeps-varphi}
 e^{\eps_k}(\varphi)  \equiv 
  \,\begin{pNiceArray}{cw{r}{0.4cm}c}[margin]
\Block{2-2}{ \epl(\varphi)  } & &  0     \\
& &  0    \\
0   &  0   &  0
\end{pNiceArray} \,, 
\end{equation}
and individually
 addressing each  integral term. 

 With the very same argument as for \eqref{viscosity-tends-0}, we have  that  for all $ \varphi \in  \KLGD$
\begin{equation}
\label{lim-pass-visc-terms-a}
\begin{aligned}
&  \lim_{k\to\infty} \int_s^t \int_{\Omega{\setminus}\GC}   {\eps_k} \mathbb{D}_{\eps_k} e^{\eps_k}(\resed u{\eps_k}(r))
{:} e^{\eps_k}(\varphi) \dd x \dd  r =0,
\\
&   \lim_{k\to\infty}     \int_s^t \int_{\Omega{\setminus}\GC}   {\eps_k} \mathbb{D}_{\eps_k} e^{\eps_k}(\resed w{\eps_k}(r))
{:} e^{\eps_k}(\varphi) \dd x \dd r =0 \,.
\end{aligned}
\end{equation}
As for the terms involving the elasticity tensor $\bbC$, 
we exploit the  structure \eqref{struct-eeps-varphi}
of $e^{\eps_k}(\varphi)$ 
and combine   \eqref{further-cvg-w},   \EEE \eqref{conv-stress-u},  \eqref{toSHOW-MM} and  \eqref{reduced-elasticity-tensor}, to conclude that
 \EEE 
\begin{equation}
\label{lim-pass-el-terms}
\lim_{k\to \infty} \int_s^t \int_{\Omega{\setminus}\GC}   \bbC e^{\eps_k}(\rese u{\eps_k} (r){+}\rese w{\eps_k} (r))  {:} e^{\eps_k}(\varphi) \dd r \dd x =  \int_s^t \int_{\Omega{\setminus}\GC}   \bbC_\red \left(\epl(\sfu(r)){+} \epl(\sfw(r)) \right) {:} \epl(\varphi) \dd r \dd x\,. 
\end{equation}
 Now, the mapping $\alpha_\lambda:\R^3 \to \R^3$ is Lipschitz continuous with Lipschitz constant $\tfrac 2{\lambda}$. 
Hence, taking into account that $\alpha_\lambda(0)=0$, we infer that 
\begin{equation}
\label{est-alpha-jump}
\|\alpha_\lambda(\JUMP{\resei u{\eps_k}1, \resei u{\eps_k}2, 0})\|_{L^\infty (0,\infty;L^4(\GC))} \leq\frac{2}\lambda
\| \JUMP{\sfu^{\eps_k}} \|_{L^\infty (0,\infty;L^4(\GC))} \leq C,
 \end{equation}
where the last estimate follows from \eqref{est-rescaled-sols-2}. Therefore,  on account of the second of \eqref{hypB-1}, we have,   in fact for all $\varphi \in \Spu$, that \EEE
\begin{equation}
\label{lim-pass-alpha-terms}
\nu_{\eps_k}\int_s^t \int_{\GC}\alpha_\lambda(\JUMP{\resei u{\eps_k}1(r), \resei u{\eps_k}2(r), 0})\cdot \JUMP{\varphi_1,\varphi_2,0} \dd \Surf(x) \dd r
\longrightarrow  0 \,.
\end{equation}
Combining \eqref{convs-k-u1} and  \eqref{convs-k-z-7} we infer,   again  for all $\varphi \in \Spu$  \EEE
\begin{equation}
\label{lim-pass-adh-terms}
\int_s^t \int_{\GC} \kappa\rese z{\eps_k}(r)\JUMP{\rese u{\eps_k}(r)}\JUMP{\varphi}\,\mathrm{d}\Surf (x) \dd r\longrightarrow \int_s^t \int_{\GC} \kappa\sfz(r)\JUMP{\sfu(r)}\JUMP{\varphi}\,\mathrm{d}\Surf (x)  \dd r.
\end{equation}
Finally, by   \eqref{convergences-f} we obtain
 \begin{subequations}
 \begin{equation}
\label{lim-pass-inert-terms-f-phi}
 \int_s^t \int_\Omega \rese f{\eps_k}(r)\varphi \dd x \dd r
 \longrightarrow \int_s^t \int_\Omega \sff(r)\varphi \dd x \dd r \,.
\end{equation}
\end{subequations}
\par
\RCO All in all, we conclude the \emph{integrated momentum balance}
\[
\begin{aligned}
&
 \int_s^t 
\int_{\Omega{\setminus}\GC} \bbC_{\mathrm{r}} \epl(\rese u{}{+} \rese w{})  {:} \epl(\varphi)     \dd x  \dd r 
+\int_s^t \int_{\GC} \kappa\rese z{}\JUMP{\rese u{}}\JUMP{\varphi}\,\mathrm{d}\Surf (x) \dd r 
= \int_s^t \int_\Omega \rese f{}\varphi \dd x  \dd r 
\end{aligned}
\]
 for every $\varphi \in  \KLGD$.  \EEE
  Since  $[s,t]$ is an arbitrary sub-interval in $ [0,\infty)$,  the momentum balance \EEE \eqref{mom-bal-DR} follows.  
\medskip

\noindent 
\textbf{Step $3$: Improved convergences.}
 In this step we aim at improving the convergences
of some  of the  terms  contributing to $ \rese{E}{\eps_k} (\cdot,\rese u{\eps_k}(\cdot))$. In particular, we shall obtain the pointwise convergence \eqref{eq:uep-better} ahead.
\par
With
 this aim, we consider once again \eqref{weak-mom-eps},   choose as test function   $\rese u{\eps_k} $ 
 and integrate on a generic interval $[0,t]$.
Integrating by parts in time, we have 
\begin{equation}
\label{integration-by-parts}
\begin{aligned}
&
 \int_0^t \int_{\Omega{\setminus}\GC}  \eps_k \mathbb{D}_{\eps_k} e^{\eps_k}
 (\resed u{\eps_k}(r))
 {:} e^{\eps_k}(\rese u{\eps_k}(r))  \dd x \dd r 
= \frac{\eps_k} 2 \int_{\Omega{\setminus}\GC} \{  \mathbb{D}_{\eps_k} e^{\eps_k}(\rese u{\eps_k}(t)) {:} e^{\eps_k}(\rese u{\eps_k}(t))  
{-} \mathbb{D}_{\eps_k} e^{\eps_k}(\rese u{\eps_k}_0) {:} e^{\eps_k}(\rese u{\eps_k}_0) \}  \dd x \,.
\end{aligned}
\end{equation}
Thus,  from the momentum balance \eqref{weak-mom-eps} tested by     $\rese u{\eps_k}$,  we infer
\begin{equation}
\label{crucial-limsup}
\begin{aligned}
 \limsup_{k\to\infty} \Big(   & 
 \frac{\eps_k} 2 \int_{\Omega{\setminus}\GC} \mathbb{D}_{\eps_k} e^{\eps_k}(\rese u{\eps_k}(t)) {:} e^{\eps_k}(\rese u{\eps_k}(t))  \dd x 
{+} 
\int_0^t \int_{\Omega{\setminus}\GC}\mathbb{C}\rese e{\eps_k}(\rese u{\eps_k}){:}\rese e{\eps_k}(\rese u{\eps_k})\dd x\dd r
\\
& \quad \EEE
+
  \nu_{\eps_k}  \int_0^t \int_{\GC}\alpha_\lambda(\JUMP{\rese u{\eps_k}_1(r), \rese u{\eps_k}_2(r), 0})\cdot \JUMP{\rese u{\eps_k}_1(r), \rese u{\eps_k}_2(r), 0} \dd \Surf(x) \dd r
\\
& \quad  \EEE +   \int_0^t \int_{\GC} \kappa\rese z{\eps_k}(r)\JUMP{\rese u{\eps_k}(r)}\JUMP{\rese u{\eps_k}(r)}\,\mathrm{d}\Surf (x)  \dd r 
\Big) 
\leq   I_1 +I_2 +I_3+I_4  \,, 
\end{aligned}
\end{equation}
 where the integral terms  $(I_j)_{j=1}^4$   are discussed below:
\begin{enumerate}
 \item 
 Since  $\mathsf{u}_0^{\eps_k} \to \mathsf{u}_0$ 
in $H^1(\Omega{\setminus}\GC;\R^3)$, 
we have 
\[
\begin{aligned}
 I_1 : = \lim_{k\to\infty}  \frac {\eps_k} 2  \int_{\Omega{\setminus}\GC} \mathbb{D}_{\eps_k} e^{\eps_k}(\rese u{\eps_k}_0)
{:} e^{\eps_k}(\rese u{\eps_k}_0) \dd x    = 0.
 \end{aligned}
\]
\item 
By combining \eqref{convergences-f}  with 
\eqref{convs-k-u1}
we have 
\[
I_2 : = \lim_{k\to\infty} \int_0^t \int_\Omega \rese f{\eps_k}(r) \rese u{\eps_k}(r) \dd x  \dd r = \int_0^t \int_\Omega \rese f{}(r) \rese u{}(r) \dd x  \dd r\,.
\]
\item
 Combining the boundedness of  $( e^{\eps_k}
 (\resed w{\eps_k}))_k$
 and 
 $(e^{\eps_k}(\rese u{\eps_k}))_k$ in $ L_\loc^2 (0,\infty;L^2(\Omega;\R^{3\times 3})) $ with the fact that $\eps_k  \mathbb{D}_{\eps_k} \to0$ we infer
\[
\begin{aligned}
I_3   : = -\lim_{k\to\infty}  \int_0^t \int_{\Omega{\setminus}\GC}  \eps_k \mathbb{D}_{\eps_k} e^{\eps_k}
 (\resed w{\eps_k}(r))
 {:} e^{\eps_k}(\rese u{\eps_k}(r))  \dd x \dd r =0\,.
 \end{aligned}
\]
\item  
Finally, by combining \eqref{further-cvg-w}  with 
\eqref{conv-stress-u}
we have 
\[
I_4:  = -\lim_{k\to\infty} \int_0^t \int_{\Omega{\setminus}\GC}\mathbb{C}\rese e{\eps_k}(\rese w{\eps_k}){:}\rese e{\eps_k}(\rese u{\eps_k})\dd x\dd r  = -   \int_0^t 
\int_{\Omega{\setminus}\GC} \bbC_{\mathrm{r}}  \epl(\rese w{})  {:}  \epl(\rese u{})     \dd x  \dd r 
\,.
\]
\end{enumerate}
All in all, from \eqref{crucial-limsup} we conclude that \emph{for almost all}  $t\in (0,T)$
\[
\begin{aligned}
&
\limsup_{k\to\infty} \Big(
 \frac{\eps_k} 2 \int_{\Omega{\setminus}\GC} \mathbb{D}_{\eps_k} e^{\eps_k}(\rese u{\eps_k}(t)) {:} e^{\eps_k}(\rese u{\eps_k}(t))  \dd x  {+} 
\int_0^t \int_{\Omega{\setminus}\GC}\mathbb{C}\rese e{\eps_k}(\rese u{\eps_k}){:}\rese e{\eps_k}(\rese u{\eps_k})\dd x\dd r
\\
& \quad 
+ \nu_{\eps_k}  \int_0^t \int_{\GC}\alpha_\lambda(\JUMP{\rese u{\eps_k}_1, \rese u{\eps_k}_2, 0})\cdot \JUMP{\rese u{\eps_k}_1, \rese u{\eps_k}_2, 0} \dd \Surf(x) \dd r
+   \int_0^t \int_{\GC} \kappa\rese z{\eps_k}\JUMP{\rese u{\eps_k}}\JUMP{\rese u{\eps_k}}\,\mathrm{d}\Surf (x)  \dd r 
\Big)
\\
& \leq 
\int_0^t \int_\Omega \rese f{}  \rese u{} \dd x  \dd r-
\int_0^t 
\int_{\Omega{\setminus}\GC} \bbC_{\mathrm{r}}  \epl(\rese w{})  {:}  \epl(\rese u{})     \dd x  \dd r
  \\
  & \stackrel{(1)} =  \int_0^t 
\int_{\Omega{\setminus}\GC} \bbC_{\mathrm{r}}  \epl(\rese u{})  {:}  \epl(\rese u{})  \EEE   \dd x  \dd r 
+\int_0^t \int_{\GC} \kappa\rese z{}\JUMP{\rese u{}}\JUMP{\rese u{}}\,\mathrm{d}\Surf (x) \dd r\,,
\end{aligned}
\]
where {\footnotesize (1)} follows from testing the momentum balance \eqref{mom-bal-DR} by $\rese u{}$,  (which is an admissible test function since  $\rese u{}(t) \in  
\KLGD$).   
\par
In turn,
\begin{equation}
\label{all-liminfs}
\begin{cases}
\displaystyle  
 \liminf_{k\to\infty}  
 \frac{\eps_k} 2 \int_{\Omega{\setminus}\GC} \mathbb{D}_{\eps_k} e^{\eps_k}(\rese u{\eps_k}(t)) {:} e^{\eps_k}(\rese u{\eps_k}(t))  
 \dd x \dd r 
\geq 0  \EEE
\\
\displaystyle  
\liminf_{k\to\infty} \int_0^t \int_{\Omega{\setminus}\GC}  \mathbb{C} (\rese e{\eps_k}(\rese u{\eps_k}){:}\rese e{\eps_k}(\rese u{\eps_k}) \EEE
\dd x\dd r 
\\ \displaystyle  
\stackrel{(1)}\geq   \int_0^t \int_{\Omega{\setminus}\GC}   \mathbb{C} \rese e{} {:} \rese e{} \dd x \dd r \stackrel{(2)}= \int_0^t \int_{\Omega{\setminus}\GC}  \bbC_{\mathrm{r}} \epl(\rese u{})  {:} \epl(\rese u{}) \EEE     \dd x  \dd r \,,
\\
\displaystyle \liminf_{k\to\infty} \nu_{\eps_k}  \int_0^t \int_{\GC}\alpha_\lambda(\JUMP{\rese u{\eps_k}_1, \rese u{\eps_k}_2, 0})\cdot \JUMP{\rese u{\eps_k}_1, \rese u{\eps_k}_2, 0} \dd \Surf(x) \dd r \geq 0\,,
\\
\displaystyle 
\liminf_{k\to\infty}  \int_0^t \int_{\GC} \kappa\rese z{\eps_k}\JUMP{\rese u{\eps_k}}\JUMP{\rese u{\eps_k}}\,\mathrm{d}\Surf (x)  \dd r  
 \stackrel{(3)}\geq \EEE  \int_0^t \int_{\GC} \kappa\rese z{}\JUMP{\rese u{}}\JUMP{\rese u{}}\,\mathrm{d}\Surf (x) \dd r\,,
\end{cases}
\end{equation}
where    {\footnotesize (1)} is due to \eqref{conv-stress-u},  {\footnotesize (2)} is due to  \eqref{toSHOW-MM}
 \EEE while {\footnotesize (3)} follows from combining the weak convergence  \eqref{convs-k-u1} for $(\rese u{\eps_k})_k$ with the strong convergence \eqref{convs-k-z-7} for  $(\rese z{\eps_k})_k$, via the Ioffe Theorem (cf., e.g., \cite{Ioff77LSIF,Valadier90}). 
\par
Therefore, all inequalities in \eqref{all-liminfs} in fact hold with equalities, with $\liminf_{k\to\infty}$ replaced by $\lim_{k\to\infty}$, for almost all $t\in (0,\infty)$. 
In particular, 
\begin{equation}
\label{crucial-for-later}
\begin{aligned}
 \lim_{k\to\infty} \int_0^t \int_{\Omega{\setminus}\GC}  \mathbb{C}\rese e{\eps_k}
(\rese u{\eps_k}){:}\rese e{\eps_k}(\rese u{\eps_k})\dd x\dd r
& =  \int_0^t \int_{\Omega{\setminus}\GC}\mathbb{C}  \rese e{}{:} \rese e{}\dd x \dd r
= \int_0^t \int_{\Omega{\setminus}\GC} \bbC_{\mathrm{r}}   \epl(\rese u{})  {:} \epl(\rese u{})     \dd x  \dd r\,.
 \end{aligned}
\end{equation}
This strengthens the   weak convergence \eqref{conv-stress-u} 
  to 
\[
\rese e{\eps_k}(\rese u{\eps_k} )\to \rese e{}  \qquad \text{strongly in } L_\loc^2(0,\infty; L^2(\Omega;\R^{3\times 3})\,.
\]
Therefore, we have
\begin{equation}
\label{faticosa}
\epl(\rese u{\eps_k}) = (\rese e{\eps_k}(\rese u{\eps_k}))_{\mathrm{plan}}  \longrightarrow (\rese e{})_{\mathrm{plan}} 
= \epl(\sfu)
\end{equation} \EEE
strongly in $ L_\loc^2(0,\infty; L^2(\Omega;\R^{2\times 2})) $,
where with slight abuse we have used the notation
$
\rmA_{\mathrm{plan}}  = \left( \begin{array}{cc}
a_{11} & a_{12}
\\
a_{21} & a_{22} 
\end{array}
\right) $  for the $(2{\times}2)$-minor of  given $ \rmA \in \R^{3\times 3}\,.
$
 We combine \eqref{faticosa} with the previously obtained
 \eqref{towards-KL-structure} and ultimately obtain that 
 $e(\rese u{\eps_k} )
  \to e(\rese u{})$ in $ L_\loc^2(0,\infty; L^2(\Omega;\R^{3\times 3})$. Then,
 via Korn's inequality we  conclude  
\begin{equation}
\label{key-step-4th1}
\rese u{\eps_k}\to \rese {u}{}  \quad\text{strongly in }L_\loc^2(0,\infty;H^1(\Omega{\setminus}\GC;\R^3))\,.
\end{equation}
In particular, we have
 the \emph{pointwise} convergence
\begin{equation}
\label{eq:uep-better}
\rese u{\eps_k}(t)\to \rese {u}{}(t) \quad\text{strongly in }H^1(\Omega{\setminus}\GC;\R^3)) \qquad \foraa\, t \in (0,\infty)\,.
\end{equation}
\EEE
\medskip

\noindent
\textbf{Step $4$: limit passage in the semistability condition.}
 We now take the limit as $k\to\infty$ in the rescaled semistability condition \eqref{reduced-semistab-eps} \emph{at every $t\in (0,\infty)$, out of a negligible set, for which 
\eqref{eq:uep-better} holds}. 
 \EEE
 Arguing as in Step 2 of the proof of Theorem \ref{thm:pass-lim-nu},  for all $ \widetilde z\in \SBV(\GC;\{0,1\})$ with $\widetilde{z}\leq z(t)$  a.e. in $\GC$ we need to construct a sequence
 $(\widetilde{z}_{\eps_k})_k$ such that 
 \begin{equation}
 \label{limsup-MRS-epsk}
 \begin{aligned}
 &
 \limsup_{k\to\infty} \Big(   \int_{\GC}  \tfrac \kappa 2   (\widetilde{z}_{\eps_k} {-} \rese z{\eps_k}(t)) |\JUMP{\rese u{\eps_k}(t)}|^2 \dd \Surf(x) 
 +\mathrm{b}_{\eps_k} ( P(\widetilde{Z}_{\eps_k},\GC) {-}  P(\rese Z{\eps_k}(t),\GC) )  
 \\
&\qquad \qquad 
+\int_{\GC}  (a^{\eps_k}_0 {+}a^{\eps_k}_1)  |\widetilde{z}_{\eps_k}{-}\rese z{\eps_k}(t)| \dd \Surf(x)  \Big)
\\
&
\leq  \int_{\GC}  \tfrac \kappa 2   (\widetilde{z} {-} \rese z{}(t)) |\JUMP{\rese u{}(t)}|^2 \dd \Surf(x) 
 +\mathrm{b} ( P(\widetilde{Z},\GC) {-}  P(\rese Z{}(t),\GC) )  
+\int_{\GC}  (a_0 {+}a_1)   |\widetilde{z}{-}\rese z{}(t)| \dd \Surf(x) 
\end{aligned}
 \end{equation}
 (as usual, we denote by $\widetilde{Z}_{\eps_k}$, $\rese Z{\eps_k}(t)$, $\widetilde{Z}$, and $
\rese Z{}(t)$, the finite-perimeter sets associated with $\widetilde{z}_{\eps_k}$, $\rese z{\eps_k}(t)$, $\widetilde{z}$, and $
\rese z{}(t)$, respectively). 
To obtain \eqref{limsup-MRS-epsk}, we repeat verbatim the construction from \eqref{MRS}. Thus, we obtain  a sequence $(\widetilde{z}_{\eps_k} )_k \subset \SBV(\GC;\{0,1\})$ such that 
 $\widetilde{z}_{\eps_k} \to \widetilde z$ in $L^q(\GC)$  for all $1\leq q <\infty$. 
 Combining this with the fact that 
 \[
 \JUMP{\rese u{\eps_k}(t)} \to \JUMP{\rese u{}(t)} \qquad \text{strongly in } 
  L^{4}(\GC), \EEE 
 \]
 we readily infer 
  \begin{equation}
  \label{semistab-adh-en}
 \lim_{k\to\infty}  \int_{\GC}   \tfrac \kappa2 ( \widetilde{z}_{\eps_k}{-} \rese z{\eps_k}(t)) |\JUMP{\rese u{\eps_k}(t)}|^2 \dd \Surf(x)  = 
   \int_{\GC} \tfrac \kappa 2  ( \widetilde{z}{-} \rese z{}(t)) |\JUMP{\rese u{}(t)}|^2 \dd \Surf(x) \,.
 \end{equation}
 We handle the other terms in \eqref{limsup-MRS-epsk} exactly in the same way as in  Step 2 of the proof of Theorem \ref{thm:pass-lim-nu}. Ultimately, we conclude  that for almost all $t\in (0,T) $ and for all $ \widetilde z\in \SBV(\GC;{0,1})$ with $ \widetilde{z}\leq z(t)$   a.e.\ in $\GC$
there holds
\begin{equation}
\label{true-semistability}
\begin{aligned}
    &
\int_{\GC} \tfrac \kappa 2  \sfz(t) |\JUMP{\sfu(t)}|^2 \dd \Surf(x) + 
\mathrm{b} P(\rese Z{}(t),\GC)  -   \int_{\GC}  a_0  \sfz(t) \dd \Surf(x)
\\ &  \leq  \int_{\GC} \tfrac \kappa 2   \widetilde{z} |\JUMP{\sfu(t)}|^2 \dd \Surf(x) 
+\mathrm{b} P(\widetilde{Z},\GC) - \int_{\GC}   a_0 \widetilde z \dd \Surf(x)
+\int_{\GC}  a_1   |\widetilde{z}{-}\sfz(t)| \dd x \,.
\end{aligned}
\end{equation}
Hence, we have obtained  the semistability condition \eqref{semistab-DR}. \EEE

\medskip

\noindent
\textbf{Step $5$: limit passage in the energy-dissipation inequality.}
 We now address the   limit  passage as $k\to \infty$ in  \eqref{EDB-eps}, written on the interval $[0,t]$.
Let us first tackle the terms on the left-hand side of  \eqref{EDB-eps}:
combining convergences   \eqref{crucial-for-later},  and \eqref{eq:uep-better} with conditions  \eqref{convergences-f}--\eqref{cond-rescal-data} we have 
\begin{equation}
\label{liminf-inequality-bulk-energy}
\lim_{k\to\infty} \resei{E}{\eps_k}{\mathrm{bulk}}(t, \rese u{\eps_k}(t)) =  \resei{E}{}{\mathrm{bulk}}(t, \rese u{}(t))  \qquad \foraa\, t \in (0,\infty)\,.
\end{equation}
Additionally, by \eqref{convs-k-z-5} \& \eqref{convs-k-z-6} we have
\begin{equation}
\label{liminf-inequality-surface-energy}
\liminf_{k\to\infty} \resei{E}{\eps_k}{\mathrm{surf}}(t, \rese u{\eps_k}(t),  \rese z{\eps_k}(t)) \geq  \resei{E}{}{\mathrm{surf}}(t, \rese u{}(t), \rese z{}(t))  \qquad \foraa\, t \in (0,\infty)\,.
\end{equation}
By the very same convergences we have  for every $t\in [0,\infty)$ 
\begin{equation}
\label{liminf-inequality-R-Var}
 \Var_{\mathsf{R}^{\eps_k}}(\rese z{\eps_k}, [0,t]) = \int_{\GC} a_{\eps_k}^1 (\rese z{\eps_k}(0){-}\rese z{\eps_k}(t)) \dd \Surf(x) \longrightarrow  \int_{\GC} a^1 (\rese z{}(0){-}\rese z{}(t)) \dd \Surf(x)  =  \Var_{\mathsf{R}}(\rese z{}, [0,t])\,.
\end{equation}
All in all, we conclude that 
\[
 \Var_{\mathsf{R}}(\sfz, [0,t])+
\rese E{}(t,\sfu(t),\sfz(t)) \leq \liminf_{k\to\infty} \left( \text{l.h.s.\ of \eqref{EDB-eps}} \right)  \qquad \foraa\, t \in (0,\infty)\,.
\]
As for the right-hand side, we have 
\begin{equation}
\label{liminf-inequality-power}
\begin{aligned}
&
\lim_{k\to\infty}\int_0^t \partial_t  \rese{E}{\eps_k}(r, \rese u{\eps_k}(r), \rese z{\eps_k}(r)) \dd r 
\\
&
=  \lim_{k\to\infty}\int_0^t  \Big( {-}\int_{\Omega} \resed f{\eps_k}(r) \cdot \rese u{{\eps_k}} (r) \dd x +
 \int_{\Omega{\setminus}\GC} \bbC e^{\eps_k}(\resed w{\eps_k}(r)): e^{\eps_k}(\rese u{\eps_k}(r)) \dd x   
\\
& \qquad \qquad  +{\eps_k}\int_{\Omega{\setminus}\GC} \bbD_{\eps_k} e^{\eps_k}(\resedd w{\eps_k}(r)): e^{\eps_k}(\rese u{\eps_k}(r)) \dd x 
\Big) \dd r 
\\
& \stackrel{(1)}= \int_0^t  \Big( {-}\int_{\Omega} \resed f{}(r) \cdot \rese u{}(r) \dd x +
 \int_{\Omega{\setminus}\GC} \bbC_{\mathrm{r}} \epl(\resed w{}(r)): \epl(\rese u{}(r)) \dd x \Big) \dd r
 = \int_0^t \partial_t  \rese{E}{}(r, \rese u{}(r), \rese z{}(r)) \dd r\,.
\end{aligned}
\end{equation}
Indeed, 
\begin{itemize}
\item[-]
 thanks to 
\eqref{convergences-f} and \eqref{convs-k-u1} we have  
 $\displaystyle \int_0^t  \int_{\Omega} \resed f{\eps_k} \cdot \rese u{{\eps_k}} \dd x \dd r \to \int_0^t  \int_{\Omega} \resed f{} \cdot \rese u{} \dd x \dd r  $;
\item[-] 
 by  \eqref{further-cvg-w},  we have \EEE
\[
\lim_{k\to\infty}\int_0^t  \int_{\Omega{\setminus}\GC} \bbC e^{\eps_k}(\resed w{\eps_k}(r)): e^{\eps_k}(\rese u{\eps_k}(r)) \dd x  \dd r = \int_0^t  \int_{\Omega{\setminus}\GC} \bbC_{\mathrm{r}} \epl(\resed w{}(r)): \epl(\rese u{}(r)) \dd x  \dd r;
\]
\item[-]   we use that 
\[
\begin{aligned}
&
\lim_{k\to\infty} \int_0^t \int_{\Omega{\setminus}\GC}\eps_k \bbD_{\eps_k} e^{\eps_k}(\resedd w{\eps_k}(r)): e^{\eps_k}(\rese u{\eps_k}(r)) \dd x  \dd r
\\
&= \lim_{k\to\infty} \int_0^t \int_{\Omega{\setminus}\GC}\eps_k^{\beta} \bbD_{\eps_k} \eps_k^{1-\beta}e^{\eps_k}(\resedd w{\eps_k}(r)): e^{\eps_k}(\rese u{\eps_k}(r)) \dd x  \dd r
 =0,
 \end{aligned}
\]
combining the second bound in  \eqref{bounds-w} and the boundedness of 
$(e^{\eps_k}(\rese u{\eps_k})_k$ 
 with the fact that $\eps_k^{\beta} \bbD_{\eps_k}\to 0$.
\end{itemize}
Ultimately, we conclude \eqref{EDI-lim}.

%

\medskip

This finishes the proof of Theorem \ref{mainthm-1}. 
\QED

\begin{remark}
\upshape
\label{rmk:why-we-cant}
Strengthening the weak convergence
\eqref{convs-k-u1}
 of $(\rese u{\eps_k})_k$ to $\rese u{}$ to the strong convergence \eqref{key-step-4th1} 
 has marked  a
 crucial point in the proof. In fact,  it has allowed us to carry out the limit passage in the semistability condition, because it has led to \eqref{semistab-adh-en},
 necessitating the \emph{strong}, pointwise-in-time convergence of the jumps $(\JUMP{(\rese u{\eps_k}(t)})_k$. 
 \par
 Since \eqref{key-step-4th1}  is solely an \emph{integral} convergence, it implies pointwise convergence   except on a negligible set of $(0,\infty)$: this is the reason why,
 for the limiting system $(\rese R{},\rese E{})$ we have obtained the semistability condition and, ultimately, the enegy-dissipation inequality only for \emph{almost all} $t\in (0,\infty)$. 
 \par
 As for our strategy for
 obtaining  \eqref{key-step-4th1}, we recall that, 
due to the loss of the damping term in the limit problem we have  been missing the estimates on $(e(\resed u{\eps_k}))_k$ that would have been instrumental in proving 
 \eqref{key-step-4th1}. We have been able to obtain it only via the argument  in Step $3$ of the above proof.  In turn, to carry  out such argument
  we have had to get rid of the inertial terms even on the level of the approximate problem by setting $\varrho_{\eps_k} \equiv 0$ in Condition \ref{cond:1}.
\end{remark}
\EEE
%
%
\section{Proof of Theorem   \ref{mainthm-2}} 
\label{s:proof-th2}
 In the proof of  Thm.\   \ref{mainthm-2}, a major role will be played by 
an `extended' version of the operator $\mathbb{M}: \R_{\mathrm{sym}}^{2\times2} \to  \R_{\mathrm{sym}}^{3\times3} $
from \eqref{def-oper-M}. 
We are going to introduce it in Section \ref{ss:6.1} ahead.
Then, in Sec.\ \ref{ss:6.2} we will carry out the proof of  Thm.\   \ref{mainthm-2}. \EEE
\subsection{The operator $\mathcal{M}_{\kern-1pt{\tiny\textsc{VE}}}$}
\label{ss:6.1}
Recall that  the operator $\mathbb{M}: \R_{\mathrm{sym}}^{2\times2} \to  \R_{\mathrm{sym}}^{3\times3} $ from
\eqref{def-oper-M} could be either defined via a minimum problem or through an orthogonality condition.
In the following lines we will 
 introduce  an operator
  $\Mextname$ given by means of an orthogonality indentity involving \emph{both} the elasticity tensor $\bbC$ \emph{and} the viscosity tensor $\mathsf{D}$ from Hypothesis \ref{Hyp-D}. We will  then show how $\Mextname$ is equivalently characterized by an implicit minimum problem expressed
in a \emph{temporally nonlocal} fashion.

  \par In order to formulate our definition and the implicit minimum problem 
in a compact form and avoid cumbersome notation, preliminarily we need to  settle the following notation: with a  matrix \EEE  $\Xi = (\xi_{ij}) \in   \R_{\mathrm{sym}}^{2\times2} $
 and  a vector $\eeta  = (\eta_1,\eta_2,\eta_3)\in \R^3$ we associate the  symmetric  $(3{\times}3)$-matrix 
 \begin{equation}
 \label{construction-mix}
 \mix \Xi \eeta: = \left( 
 \begin{array}{llll}
 \xi_{11} & \xi_{12} & \eta_1
 \smallskip
 \\
 \xi_{12}  &  \xi_{22} & \eta_2
  \smallskip
 \\
  \eta_1 &  \eta_2 &  \eta_3
 \end{array} 
 \right)\,.
 \end{equation}
 We will also use the notation 
 \[
 \mix {\mathbf{O}} \eeta: =  \left(
\begin{array}{llll}
  0 & 0  &   \vspace{-0.1 cm}   \eta_1 \EEE
 \smallskip
  \\
0& 0   &    \eta_2 \EEE
\\ 
 \smallskip  \eta_1 \EEE &    \eta_2  \EEE &  \eta_3 \EEE
\end{array} \right) \qquad \text{for all } \eeta \in \R^3. 
 \]
Moreover, in addition to the previously introduced quadratic form
$\Lambda_{\bbC}$ defined by $\Lambda_{\bbC}(A): = \tfrac12 \bbC A: A$ for all $A\in \R_\sym^{3\times 3}$, we also bring into play the quadratic form
$\Lambda_{\mathsf{D}}$
associated with the viscosity tensor $\mathsf{D}$ from Hyp.\ \ref{Hyp-D}. We further recall the definition of the matrix $\mathscr{D}_3$=$(\mathscr{D}_3)_{ij}=\mathsf{D}_{i3j3}$.

We now introduce \EEE 
the mapping 
\[
\Mextname: \R^{3}\times H^1_{\loc}(0,\infty; \R_{\mathrm{sym}}^{2\times2})  \to   H^1_{\loc}(0,\infty;\R_{\mathrm{sym}}^{3\times3})  
\]
  as follows. \EEE 
\begin{definition}
\label{def-mextname-new}
Assume that  $\mathscr{D}_3$ is invertible.
 For all  $\boldsymbol{m}\in \R^{3}$ and $\Xi \in  H^1_{\loc}(0,\infty; \R_{\mathrm{sym}}^{2\times2}) $ we define $\Mextname (\boldsymbol{m},\Xi): = \Upsilon$, where 
 the map 
 $\Upsilon \in H^1_{\loc}(0,\infty; \R_{\mathrm{sym}}^{3\times3})$,  
fulfills 
\begin{subequations}
\label{Upsilon-definition}
\begin{align}
\label{initial-condition-m}
&
\Upsilon_{i3}(0): =\boldsymbol{m}_{i}, &&  i=1,2,3 &&
\\
&
\label{identification-minors}
\Upsilon_{\mathrm{plan}}(t)  =\Xi(t) 
 && &&  \foraa\, t \in (0,\infty), 
\intertext{as well as }
&
\label{charact-M-mixed}
  (\bbC\Upsilon(t){+}\mathsf{D}\dot{\Upsilon}(t)) :  \mix{\mathbf{O}}{\zzeta} = 0  &&  \text{ for all }  \zzeta \in \R^3  &&  \foraa t \in (0,\infty). 
  \end{align}
 \end{subequations}
 \end{definition} \EEE
The next lemma ensures that the operator $\Mextname:   \R^3  
\times H^1_{\loc}(0,\infty; \R_{\mathrm{sym}}^{2\times2})  \to   H^1_{\loc}(0,\infty;\R_{\mathrm{sym}}^{3\times3})  
$ is well defined.
\begin{lemma}
Assume that  $\mathscr{D}_3$ is invertible.  
 Then, for every $\boldsymbol{m}\in \R^{3}$ and $\Xi \in  H^1_{\loc}(0,\infty; \R_{\mathrm{sym}}^{2\times2})$ there exists a unique 
 $\Upsilon\in H^1_{\loc}(0,\infty;\R_{\mathrm{sym}}^{3\times3})$  satisfying all the conditions in Definition \ref{def-mextname-new}.
\end{lemma}
\begin{proof}
We preliminary observe that The map $\Upsilon$ solves $\Upsilon(0)=\mix \Xi {\boldsymbol m}$ and,  rewriting \eqref{charact-M-mixed} in components,  
\begin{equation*}    
    \sum_{\alpha,\beta=1}^2 \left(\bbC_{i3\alpha\beta}\Xi_{\alpha,\beta}(t)+\mathsf{D}_{i3\alpha\beta}\dot{\Xi}_{\alpha,\beta}(t)\right)+2\sum_{k=1}^3 \left(\bbC_{i3k3}\Upsilon_{k3}(t)+\mathsf{D}_{i3k3}\dot{\Upsilon}_{k3}(t)\right)=0,
\end{equation*}
for every $i=1,2,3$, for a.a. $t\in (0,\infty)$. We denote by $\mathscr{C}_3$ the matrix $(\mathscr{C}_3)_{ij}=\bbC_{i3j3}$, and we consider the vector functions $ t \mapsto \uupsilon_3(t) : = (\Upsilon_{13}(t), \Upsilon_{23}(t), \Upsilon_{33}(t)) \in H^1_{\loc}(0,\infty;\R^3)$ and  $ t \mapsto \boldsymbol{y}(t)
\in H^1_{\loc}(0,\infty;\R^3)$ fulfilling 
\[
y_i(t)= -\frac12\sum_{\alpha,\beta=1}^2 \left(\bbC_{i3\alpha\beta}\Xi_{\alpha,\beta}(t)+\mathsf{D}_{i3\alpha\beta}\dot{\Xi}_{\alpha,\beta}(t)\right)  \text{ 
for every $i=1,2,3$, for a.a. $t\in (0,\infty)$.}
\]
We find that $\uupsilon_3$ solves the linear ODE-system
\begin{equation*}
    \begin{cases}    \mathscr{C}_3\uupsilon_3(t)+\mathscr{D}_3\dot{\uupsilon}_3(t)=
 \boldsymbol{y}(t) 
     &\\
    \uupsilon_3(0)=\boldsymbol{m},
    \end{cases}
\end{equation*}
which, owing to the invertibility of $\mathscr{D}_3$, rewrites as
\begin{equation}
\label{eq:ODE-system}
    \begin{cases}    \mathscr{D}_3^{-1}\mathscr{C}_3\uupsilon_3(t)+\dot{\uupsilon}_3(t)=
    \mathscr{D}_3^{-1} \boldsymbol{y}(t)
    &\\
    \uupsilon_3(0)=\boldsymbol{m}.
    \end{cases}
\end{equation}
The existence and uniqueness of $\uupsilon_3$ and of $\Upsilon$ follow then by the Picard-Lindel\"of Theorem for linear systems of ODEs.
\end{proof}

In Lemma \ref{l:charact-M} below we provide an alternative characterization of $\Mextname$ as the outcome of an implicit minimization procedure.

\begin{lemma}
\label{l:charact-M}
 Assume that  $\mathscr{D}_3$ is invertible. Let $\boldsymbol{m}\in  \R^3 \EEE$ 
 and $ \Xi \in  H^1_{\loc}(0,\infty; \R_{\mathrm{sym}}^{2\times2})$ be given. Then,
 $\Upsilon = 
  \Mext(\boldsymbol{m},\Xi)$
 if and only if  
 $\Upsilon(0)=\mix \Xi {\boldsymbol m}$, \EEE
 and
 \begin{subequations}
\label{def-mextname}
\begin{align}
&
\label{def-mextname-1}
 \Upsilon(t) = \mix{\Xi(t)}{\llambda_{\Xi}(t)} \qquad  \text{ with } \llambda_{\Xi} \in  H^1_{\loc}(0,\infty;\R^3) \text{ fulfilling  for a.a. } t \in (0,+\infty)
 \\
 & 
 \nonumber
 \Lambda_{\bbC}( \mix{\Xi(t)}{\llambda_{\Xi}(t)} ) +  \Lambda_{\mathsf{D}}\left( \mix{\dot{\Xi}(t)}{\dot{\llambda}_{\Xi}(t)}\right)\\
 &\qquad\label{def-mextname-2}= \mathrm{Min}_{\eeta \in \R^3} \left\{ \Lambda_{\bbC}( \mix{\Xi(t)}{\eeta}) +   \Lambda_{\mathsf{D}}\left( \mix{\dot\Xi(t)}{(\dot{\llambda}_{\Xi}(t){-}\llambda_\Xi(t){+}\eeta)}\right)\right\},
\end{align}
\end{subequations}
where we specify that  
\[
 \mix{\dot\Xi}{(\dot{\llambda}_{\Xi}{-}\llambda_\Xi{+}\eeta)} =  \left( 
 \begin{array}{llll}
 \dot{\xi}_{11} & \dot{\xi}_{12} & \dot{\lambda}_1 - \lambda_1+ \eta_1
 \smallskip
 \\
\dot{\xi}_{12}  &  \dot{\xi}_{22} &  \dot{\lambda}_2 - \lambda_2+  \eta_2
  \smallskip
 \\
 \dot{\lambda}_1 - \lambda_1+   \eta_1 &  \dot{\lambda}_2 - \lambda_2+   \eta_2 &  \dot{\lambda}_3 - \lambda_3+   \eta_3
 \end{array} 
 \right)
\]
(we have dropped  \EEE the $t$-dependence and the subscript $\Xi$
on the right-hand side
 for better readability). 
 
\end{lemma}
\color{black}
\begin{proof}
With any  $\eeta =(\eta_1,\eta_2,\eta_3) \in \R^3$ we associate the tensor
$\mix \Xi \eeta$  as in \eqref{construction-mix}.  
Let \GIO 
$\Upsilon = 
  \Mext(\boldsymbol{m},\Xi)$
\EEE
 with $\Upsilon(t) = (\Upsilon_{ij}(t))$.
We consider the vector function $ t \mapsto \uupsilon_3(t) : = (\Upsilon_{13}(t), \Upsilon_{23}(t), \Upsilon_{33}(t)) \in H^1_{\loc}(0,\infty;\R^3)$ and the time-dependent  tensor
(from now on, we omit the time variable for notational simplicity)
\[
 \mix{\dot\Xi}{(\dot{\uupsilon}_{3}{-}\uupsilon_3{+}\eeta)} =  \left( 
 \begin{array}{llll}
 \dot{\xi}_{11} & \dot{\xi}_{12} & \dot{\upsilon}_{13} - \upsilon_{13}+ \eta_1
 \smallskip
 \\
\dot{\xi}_{12}  &  \dot{\xi}_{22} &  \dot{\upsilon}_{23} - \upsilon_{23}+  \eta_2
  \smallskip
 \\
 \dot{\upsilon}_{13} - \upsilon_{13}+   \eta_1 &  \dot{\upsilon}_{23} - \upsilon_{23}+   \eta_2 &  \dot{\upsilon}_{33} - \upsilon_{33}+   \eta_3
 \end{array} 
 \right)\,.
\]
By elementary calculations, taking into account \eqref{identification-minors} we have 
\[
\begin{aligned}
&
\bbC \mix \Xi \eeta : \mix \Xi \eeta 
 &&  = &&  \bbC \Upsilon : \Upsilon + \bbC \mix{\mathbf{O}}{\eeta{-}\uupsilon_3} : \mix{\mathbf{O}}{\eeta{-}\uupsilon_3}  + 2\bbC \Upsilon : \mix{\mathbf{O}}{\eeta{-}\uupsilon_3}  \,;
\\
&
\mathsf{D} \mix{\dot \Xi}{(\dot{\uupsilon}_{3}{-}\uupsilon_3{+}\eeta)}: \mix{\dot \Xi}{(\dot{\uupsilon}_{3}{-}\uupsilon_3{+}\eeta)}&&  = &&  \mathsf{D} \dot\Upsilon : \dot\Upsilon + \mathsf{D}  \mix{\mathbf{O}}{\eeta{-}\uupsilon_3}:   \mix{\mathbf{O}}{\eeta{-}\uupsilon_3}  + 2\mathsf{D} \dot{\Upsilon} : \mix{\mathbf{O}}{\eeta{-}\uupsilon_3}\,.
\end{aligned}
\]
Therefore, by the positive-definiteness of the tensors $\bbC$ and $\mathsf{D}$, and recalling the definition of the operator $\Mextname$, we have that 
\[
\Lambda_{\bbC} (\mix \Xi \eeta) +\Lambda_{\mathsf{D}}\left( \mix{\dot \Xi}{(\dot{\uupsilon}_{3}{-}\uupsilon_3{+}\eeta)}\right)
\geq \Lambda_{\bbC}(  \Upsilon) + \Lambda_{\mathsf{D}}( \dot\Upsilon),
\]
and the statement
 follows. 
\end{proof}
 
\par
 The next result shows that, under the symmetry  condition \eqref{condX-1}
for the elasticity and viscosity tensors $\bbC$ and $\mathsf{D}$,
the analogue of  Lemma \ref{lm:identif} \EEE
holds for the
operator $\Mextname$. 
\begin{lemma}
\label{l:useful-X}
Assume that, in addition to  \eqref{assCD} and \eqref{viscosityTensor},  the 
tensors
 $\bbC$ and $\mathsf{D}$
 satisfy \eqref{condX-1} and the matrix $\mathscr{D}_3$ is invertible. \EEE 
Then, for every   $u \in W_{\loc}^{1,2}(0,\infty;\KL)$ we have
\begin{equation}
\label{identification-Mext-KL}
\Mext[\boldsymbol{0},\epl(u)] = e(u) = \mix{\epl(u)}{\boldsymbol{0}} \,.
\end{equation}
\end{lemma}
\begin{proof}
We use that, 
 for $\boldsymbol{m} = \boldsymbol{0}$ and $\Xi = \epl(u),$
conditions \eqref{initial-condition-m} and \eqref{identification-minors} respectively read 
$
\Upsilon_{i3}(0) =0$ for $  i=1,2,3$ and
 $\Upsilon_{\mathrm{plan}}  = \epl(u)$ a.e.\ in $(0,\infty)$. In turn, due to 
 \eqref{condX-1}, \EEE  system \eqref{eq:ODE-system} rewrites as
\begin{equation*}
    \begin{cases}    \mathscr{D}_3^{-1}\mathscr{C}_3\uupsilon_3(t)+\dot{\uupsilon}_3(t)
     = \boldsymbol{0}\EEE&\\
    \uupsilon_3(0)=\boldsymbol{0}.
    \end{cases}
\end{equation*}
 Then, $\uupsilon_3(t) \equiv \boldsymbol{0}$,  i.e., $\Upsilon_{i3}(t) \equiv 0$  for $i=1,2,3$.  Thus, \eqref{identification-Mext-KL} ensues. \EEE
\end{proof}
  \EEE

%
\subsection{Proof of Theorem   \ref{mainthm-2}}
\label{ss:6.2}
 In carrying out the proof of  Thm.\  \ref{mainthm-2}
we shall  revisit the steps of the proof of Theorem \ref{mainthm-1}, dwelling on the main differences. In what follows, for notational simplicity we will  write integrals,
in place of duality pairings, for the inertial terms in the momentum balance equation. 
\\\textbf{Step $0$: compactness.} 
Our starting point is again provided by the a priori estimates from Prop.\ \ref{prop:a-PRIO}.  Recall that, by  Hypothesis D (cf.\ in particular 
 \eqref{hypD-2}), we have $ \eps_k \bbD_{\eps_k} \approx  \mathsf{D}$ with  $\mathsf{D}$ a \emph{positive definite} tensor. Thus. 
 the bound for 
$ (\eps_k \|\bbD_{\eps_k} e^{\eps_k}(\resed u{\eps_k}) {:}  e^{\eps_k}(\resed u{\eps_k})_k$ in 
$L^1_\loc(0,\infty; L^1(\Omega;\R^{3\times 3}))$
 in
\eqref{est-rescaled-sols-4-bis}  now  \EEE ensures that  \emph{also} the sequence  $(e^{\eps_k}(\resed u{\eps_k}))_k$ is bounded in 
$L^2_\loc(0,\infty; L^2(\Omega;\R^{3\times 3}))$. A fortiori, we gather that $(\rese u{\eps_k})_k$
 is bounded in $H^1_\loc(0,\infty;H_{\Dir}^1 (\Omega{\setminus}\GC;\R^3))$. Therefore, 
  we now have \RREV (cf.\ \eqref{towards-KL-structure}) \EEE
 \[
\partial_3 \resed u{\eps_k}_3 \to 0, \quad  \left( \partial_1 \resed u{\eps_k}_3{+}\partial_3 \resed u{\eps_k}_1 \right)  \to 0, \quad 
 \left( \partial_3 \resed u{\eps_k}_2{+}  \partial_2 \resed u{\eps_k}_3 \right) \to 0 \qquad \text{ in }  L_\loc^2(0,\infty;L^2(\Omega)).
\]
Hence,
there exists   $\sfu \in H_\loc^{1}(0,\infty;  \KLGD)$ \EEE  such that
\begin{subequations}
\label{convs-TH2-u}
\begin{equation}
\label{TH2convs-k-u1}
\rese u{\eps_k} \weaksto \sfu \text{ in } H_\loc^1(0,\infty;H_{\Dir}^1 (\Omega{\setminus}\GC;\R^3))
\end{equation}
and, by the compactness results from \cite{Simon87}, we then conclude
\begin{equation}
\label{TH2convs-k-u1-bis}
\rese u{\eps_k} \to \sfu \text{ in } \mathrm{C}^0([a,b];H_{\Dir}^1 (\Omega{\setminus}\GC;\R^3)_{\mathrm{weak}}) \quad \text{for all } [a,b]\subset [0,\infty)\,.
\end{equation}
We now improve \eqref{conv-stress-u} to
\begin{equation}
\label{conv-stress-u-new}
\begin{cases}
e^{\eps_k}(\rese u{\eps_k}) \weakto \sfe  & \text{ in } H_\loc^1(0,\infty;L^2(\Omega;\R^{3\times 3})),
\\
e^{\eps_k}(\rese u{\eps_k}(t)) \weakto \sfe (t) & \text{ in } L^2(\Omega;\R^{3\times 3}) \text{ for all } t \in (0,\infty)
\end{cases}
\end{equation}
for some $\sfe  \in  H_\loc^1(0,\infty;L^2(\Omega;\R^{3\times 3})) $ for which 
\eqref{identification-eplan-u} holds.
\par
In this setup, 
 $\sfu$ admits  
the Kirchhoff-Love representation \eqref{KL-decomposition}
with  two functions $\overline \sfu \in H_\loc^1(0,\infty; H_{\Dir}^1(\Omega{\setminus}\GC;\R^2))$ and 
 $\sfu_3 \in H_\loc^1(0,\infty;  H_{\Dir}^2(\omega{\setminus}\gC)) $
(recall that, for Kirchhoff-Love displacements, the component $\sfu_3$ only depends on the variable $x'
\in \omega$).  We observe 
 that 
\begin{equation}
\label{TH2-convs-k-u2}
\resei u{\eps_k}3 \weaksto \sfu_3 \text{ in }  H_\loc^1(0,\infty;H_{\Dir}^2(\Omega{\setminus}\GC))  \cap   W_\loc^{1,\infty} (0,\infty;L^2(\Omega)) \,.
\end{equation}
 Taking into account that $\eps_k \rese u{\eps_k} \to 0 $ in $ L_\loc^\infty(0,\infty; L^2(\Omega)) $,
 we deduce 
 that   
\begin{align}
\label{convs-k-u3}
&
\eps_k \reseid u{\eps_k} i \weaksto 0 \text{ in } L_\loc^\infty(0,\infty; L^2(\Omega)) \quad \text{and}\quad 
 \eps_k^2 \reseid u{\eps_k} i \to 0 \EEE \text{ in } L_\loc^\infty(0,\infty; L^2(\Omega)) 
 \text{ for } i=1,2.
\end{align}
\end{subequations}
\par
The compactness results for  the  sequence   $(\rese z{\eps_k})_k$ reflect the fact that, in the setting of Hypothesis
\ref{Hyp-D}
 the parameters $(\mathrm{b}_{\eps_k})_k$ may in fact converge to $\mathrm{b}=0$. Therefore, in this context  we may only infer 
that there exists  $\sfz\in L_\loc^\infty(0,\infty; L^\infty(\GC)) \cap  \BV_{\loc}(0,\infty;L^1(\GC))$ \EEE such that 
\begin{subequations}
\label{convs-TH2-uz}
\begin{align}
&
\label{TH2-conv-k-z-1}
\rese z{\eps_k} \weaksto  \sfz &&  \text{in } L_\loc^\infty(0,\infty; L^\infty(\GC)),
\\
&
\label{TH2-convs-k-z-5}
  \rese z{\eps_k}(t)\weaksto \sfz(t) && \text{in } L^\infty(\GC) \EEE
  \text{ for all } t \in [0,\infty).
  \end{align}
\end{subequations}
\medskip

\noindent
\textbf{Step $1$: enhanced properties of $\sfu$.} 
We aim to show that 
\begin{equation}
\label{toSHOW-MM-34}
\sfe = \Mextname[\boldsymbol{0}, \epl(\sfu)] 
\qquad \aein \  \Omega\times (0,\infty)
\,. 
\end{equation}
 In turn, since $\sfu \in W_\loc^{1,2}(0,\infty;\  \KLGD)$, \EEE  by Lemma \ref{l:useful-X} the identification  \eqref{toSHOW-MM-34}  will lead to 
\begin{equation}
\label{toSHOW-MM-final}
\sfe(t) = e(\sfu(t))
\qquad \aein\  \Omega \quad  \text{ for all } t   \in [0,\infty) \,.
\end{equation}
\par\noindent
 In order to prove   \eqref{toSHOW-MM-34},
as in Step $1$ of the proof of Theorem 
\ref{mainthm-1}, we take 
the functions $\psi_n^{\eps_k}$ from \eqref{test-psi-added}
as test functions in \eqref{weak-mom-eps}, which we integrate on a generic interval $[s,t]\subset [0,\infty)$. We then take the limit in \eqref{lim-n-fixed}
as $\eps_k \down 0$ with  fixed $n$.
 We now have  the inertial terms to deal with: thanks to 
\eqref{convs-k-u3}
we have 
\[
\begin{aligned}
&
\lim_{k\to\infty}\int_s^t \int_\Omega \left(\eps_k^2\varrho_{\eps_k} \reseidd u{\eps_k}1 \psi^{\eps_k}_{1,n} {+}   \eps_k^2\varrho_{\eps_k} \reseidd u{\eps_k}2  \psi^{\eps_k}_{2,n}  \right) \dd x  \dd r 
\\
&
=
\sum_{i=1}^2 \lim_{k\to\infty} \int_\Omega  \psi^{\eps_k}_{i,n}(x)   [\eps_k^2\varrho_{\eps_k} \reseid u{\eps_k}i(t,x) -\eps_k^2\varrho_{\eps_k} \reseid u{\eps_k}i(s,x) ] \dd x =0\,.
\end{aligned} \qquad \foraa\, s,t \in (0,\infty)  \text{ with } s<t\,.
\]
Analogously, by \eqref{TH2-convs-k-u2} we have 
$\eps_k \reseid u{\eps_k}3 \to \eps_k\dot{\sfu}_3$ in    $L_\loc^{\infty} (0,\infty;L^2(\Omega))$, and thus by the same argument as above  we conclude
\[
\lim_{k\to\infty}\int_s^t \int_\Omega \varrho_{\eps_k} \reseidd u{\eps_k}3 \psi^{\eps_k}_{3,n} \dd x  \dd r  =0 \qquad \foraa\, s,t \in (0,\infty)  \text{ with } s<t\,.
\]
In the same way, we show that terms premultiplied by $\varrho_{\eps_k}$ and involving $\reseidd w{\eps_k}i$, $i=1,2,3$, tend to zero.  
As for  $  \eps_k  \int_s^t  \int_{\Omega{\setminus}\GC}  \mathbb{D}_{\eps_k} e^{\eps_k}(\resed u{\eps_k}) {:} e^{\eps_k}( \psi^{\eps_k}_{n} )  \dd x  \dd r $,  we 
repeat the same arguments as in Step $1$ from Thm.\ \ref{mainthm-1},
with the only difference that, now, thanks to the first of \eqref{conv-stress-u-new} combined with the condition that $\eps_k \bbD_{\eps_k} \to \mathsf{D}$ for some  symmetric
positive definite tensor $\mathsf{D} \in \R^{3{\times}3\times 3\times 3}$, \EEE we have 
\[
\begin{aligned}
\lim_{k\to\infty} \eps_k  \int_s^t  \int_{\Omega{\setminus}\GC}  \mathbb{D}_{\eps_k} e^{\eps_k}(\resed u{\eps_k}) {:} e^{\eps_k}( \psi^{\eps_k}_{n} )  \dd x  \dd r 
= \int_s^t  \int_{\Omega{\setminus}\GC}   \mathsf{D} \dot{\sfe}  {:} 
 \,\begin{pNiceArray}{cw{r}{0.3cm}c}[margin]
\Block{2-2}{ 0 } & &  \zeta^1_n     \\
& &  \zeta^2_n     \\
 \zeta^1_n   &  \zeta^2_n   &  \zeta^3_n  
\end{pNiceArray} \, \EEE
\ell'_n \dd x \dd r
\end{aligned}
\]
in place of \eqref{viscosity-tends-0}. 
Analogously, thanks to \eqref{KL-loadings-needed}
 (which is an outcome of 
 Hypothesis \ref{Hyp-E}),  \EEE
we have
\[
\begin{aligned}
&
\lim_{k\to\infty}\int_s^t  \int_{\Omega{\setminus}\GC}   {\eps_k} \mathbb{D}_{\eps_k} e^{\eps_k}(\resed w{\eps_k}) {:} e^{\eps_k}( \psi^{\eps_k}_{n} )  \dd x 
=  \int_s^t  \int_{\Omega{\setminus}\GC}   \mathsf{D}  e(\dot \sfw)  \EEE{:} 
 \,\begin{pNiceArray}{cw{r}{0.3cm}c}[margin]
\Block{2-2}{ 0 } & &  \zeta^1_n     \\
& &  \zeta^2_n     \\
 \zeta^1_n   &  \zeta^2_n   &  \zeta^3_n  
\end{pNiceArray} \, \EEE
\ell'_n \dd x \dd r\,.
\end{aligned}
\]
All in all,  in place of \eqref{limit-passage-n} we end up with 
\begin{equation}
\label{limit-passage-n-NEW}
\begin{aligned}
\int_s^t \int_{\Omega\setminus\GC}\left\{ \mathsf{D} (\dot\sfe{+} e(\dot \sfw)  \EEE) {+}  \bbC (\sfe{+} e(\sfw)  \EEE) \right\} {:} 
 \,\begin{pNiceArray}{cw{r}{0.3cm}c}[margin]
\Block{2-2}{ 0 } & &  \zeta^1_n     \\
& &  \zeta^2_n     \\
 \zeta^1_n   &  \zeta^2_n   &  \zeta^3_n  
\end{pNiceArray} \, \EEE
\ell'_n \dd x \dd r=0\,, \qquad 
 \foraa\, s,t \in (0,\infty)  \text{ with } s<t
\,,
\end{aligned}
\end{equation}
which again leads to the orthogonality property
\begin{equation}
\label{final&crucial-new}
\left\{ \mathsf{D} (\dot\sfe{+}  e(\dot \sfw)  \EEE) {+}  \bbC (\sfe{+} e(\sfw)  \EEE) \right\}{:} 
 \,\begin{pNiceArray}{cw{r}{0.3cm}c}[margin]
\Block{2-2}{ 0 } & &  \zeta^1    \\
& &  \zeta^2     \\
 \zeta^1   &  \zeta^2   &  \zeta^3
\end{pNiceArray} \, \EEE
=0 \qquad \text{for all } \zzeta=(\zeta^1,\zeta^2,\zeta^3) \in \R^3, \quad \aein\, (0,\infty)\times \Omega\,.
\end{equation}
Therefore, recalling that $\sfw \in W_\loc^{1,2}(0,\infty;\KL)$, $\sfu_0 \in   \KLGD$, and
owing to Definition \ref{def-mextname-new}, we can conclude
\begin{equation}
\label{joint-identification-sfe-dotsfe}
\sfe+e(\sfw) = \Mextname [ \boldsymbol{0},\epl(\sfu)+\epl(\sfw)]
\qquad 
\aein \  \Omega\times (0,\infty)\,.
\end{equation}
 Now, 
by Lemma \ref{l:useful-X} we have that $\Mextname [\boldsymbol{0}, \epl(\sfw)] = e(\sfw)$. \EEE
Therefore, from \eqref{joint-identification-sfe-dotsfe} we gather 
\eqref{toSHOW-MM-34} and, ultimately, 
\eqref{toSHOW-MM-final}
 follows. \EEE
\medskip

\noindent
\textbf{Step $2$: limit passage in the weak momentum balance.}
We test \eqref{weak-mom-eps}
by  a function $\varphi$ in the space $ \mathfrak{V}$ from
 \eqref{mom-bal-DR-34}. 
 We  integrate over a generic interval $[0,\tau]\subset [0,\infty)$
and then over  $[t{-}\delta,t{+}\delta]$ for fixed $\delta>0$:  the double integration in time is motivated by the lack of pointwise-in-time convergence
for $( \resed u{\eps_k})_k$.
\par
First of all, we tackle the limit as $k\to
\infty$ of the inertial terms: integrating by parts in the inner time integral,  we have 
for $i=1,2$
 \begin{subequations}
 \begin{align}
 &
   \label{lim-pass-inert-terms-1-NEW}
 \begin{aligned}
\int_{t{-}\delta}^{t{+}\delta}
\int_0^\tau \int_\Omega  \eps_k^2\varrho_{\eps_k} \reseidd u{\eps_k} i \varphi_i \dd x \dd r \dd 
\tau 
& = -\int_{t{-}\delta}^{t{+}\delta}
\int_0^\tau  \int_\Omega  \eps_k^2\varrho_{\eps_k} \reseid u{\eps_k} i \dot{\varphi}_i \dd x \dd r \dd \tau
\\
&
\quad
+\int_{t{-}\delta}^{t{+}\delta}  \int_\Omega  \left[ \eps_k^2\varrho_{\eps_k} \reseid u{\eps_k} i(\tau) \varphi_1(\tau)  {-}    \eps_k^2\varrho_{\eps_k} \reseid u{\eps_k} i(0) \varphi_i(0) \right]  \dd x  \dd \tau 
  \longrightarrow 0
  \end{aligned}
\intertext{thanks to \eqref{convs-k-u3},
 whereas we have }
&
\begin{aligned}
&
 \int_{t{-}\delta}^{t{+}\delta}
\int_0^\tau \int_\Omega \varrho_{\eps_k}  \reseidd u{\eps_k}3 \varphi_3 \dd x \dd r   \dd \tau
\\
&= - \int_{t{-}\delta}^{t{+}\delta}
\int_0^\tau \int_\Omega \varrho_{\eps_k}  \reseid u{\eps_k}3 \dot{\varphi}_3 \dd x \dd r  \dd \tau +
\int_{t{-}\delta}^{t{+}\delta} \int_\Omega \left[ \varrho_{\eps_k}  \reseid u{\eps_k}3(\tau) \varphi_3(\tau) \dd x {-}   \varrho_{\eps_k}  \reseid u{\eps_k}3(0) \varphi_3(0) 
\right] \dd x \dd \tau
\\
& \longrightarrow - \int_{t{-}\delta}^{t{+}\delta}
\int_0^\tau \int_\omega \varrho  \reseid u{}3 \dot{\varphi}_3 \dd x' \dd r  \dd \tau
+  \int_{t{-}\delta}^{t{+}\delta} \int_\omega \left[ \varrho  \reseid u{} 3(\tau) \varphi_3(\tau) {-}   \varrho  \reseid u{}3(0) \varphi_3 (0)\right] \dd x' \dd \tau
\end{aligned}
\intertext{by \eqref{TH2-convs-k-u2},  where the latter integrals are considered over $\omega$ is due to the fact that $\sfu_3$ and $\varphi_3$ only depend on the variable $x' \in \omega$.   By Hypothesis \ref{Hyp-E}} 
&
\label{lim-pass-inert-terms-3}
\begin{cases}
\displaystyle -  \int_{t{-}\delta}^{t{+}\delta}
\int_0^\tau \int_\Omega \left(\eps_k^2\varrho_{\eps_k} \reseidd w{\eps_k}1(r) \varphi_1 {+}   \eps_k^2\varrho_{\eps_k} \reseidd w{\eps_k}2(r) \varphi_2 \right) \dd x  \dd r \dd \tau   \longrightarrow 0, 
\\
\displaystyle  -  \int_{t{-}\delta}^{t{+}\delta}
\int_0^\tau  \int_\Omega \varrho_{\eps_k}  \reseidd w{\eps_k}3(r) \varphi_3 \dd x \dd r \dd \tau  \longrightarrow -   \int_{t{-}\delta}^{t{+}\delta}
\int_0^\tau  \int_\omega  \varrho \reseidd w{}3(r) \varphi_3 \dd x' \dd r \dd \tau
\end{cases}
\end{align}
\end{subequations}
(note that, here, by-part integration is not needed due to the enhanced time regularity of $(\rese w{\eps_k})_k$. 
Relying on condition  \eqref{convergences-f} for $( \rese f{\eps_k})_k$, 
we  also obtain  the analogue of \eqref{lim-pass-inert-terms-f-phi}, i.e.
\[
 \int_{t{-}\delta}^{t{+}\delta}
\int_0^\tau\int_\Omega \rese f{\eps_k}(r)\varphi \dd x \dd r \dd \tau
 \longrightarrow  \int_{t{-}\delta}^{t{+}\delta}
\int_0^\tau \int_\Omega \sff(r)\varphi \dd x \dd r  \dd \tau\,. \EEE
 \]
The limit passage in the viscosity term now works differently: due to 
\eqref{hypD-2},   \eqref{KL-loadings-needed}, and the fact that $   e^{\eps_k}(\resed u{\eps_k}) \weakto \dot \sfe    = e(\dot \sfu)$ \EEE
 in   $L_\loc^2(0,\infty;L^2(\Omega;\R^{3\times 3})) $ by \eqref{conv-stress-u-new}, we have 
\[
\begin{aligned}
&\lim_{k\to\infty}   \int_{t{-}\delta}^{t{+}\delta}
\int_0^\tau \int_{\Omega{\setminus}\GC}   {\eps_k} \mathbb{D}_{\eps_k} \big( e^{\eps_k}(\resed u{\eps_k}(r)){+}e^{\eps_k}(\resed w{\eps_k}(r)) \big)
{:} e^{\eps_k}(\varphi) \dd x \dd  r \dd \tau
\\
& \stackrel{(1)}{=}   \int_{t{-}\delta}^{t{+}\delta}
\int_0^\tau \int_{\Omega{\setminus}\GC}  \mathsf{D} (e(\dot \sfu){+}  e (\dot \sfw))  \EEE {:}   \mix{\epl(\varphi)}{\boldsymbol{0}}  \dd x \dd  r  \dd \tau
\\
& \stackrel{(2)}{=}    \int_{t{-}\delta}^{t{+}\delta}
\int_0^\tau \int_{\Omega{\setminus}\GC}   \Dr (\epl(\dot \sfu){+} \epl(\dot \sfw)){:} \epl(\varphi)  \EEE  \dd x \dd  r \dd \tau\,, 
\end{aligned}
\]
where for {\footnotesize (1)} we have used that $ e^{\eps_k}(\varphi) = \mix{\epl(\varphi)}{\boldsymbol{0}}  $ since $\varphi \in  \KLGD $, \EEE   while    {\footnotesize (2)}
follows from the fact that $e(\dot \sfx) =  \mix{\epl(\dot \sfx)}{\boldsymbol{0}}  $ for $\sfx \in \{\sfu, \sfw\}$. \EEE
Analogously, we have 
\[
\begin{aligned}
&
\lim_{k\to \infty}  \int_{t{-}\delta}^{t{+}\delta}
\int_0^\tau \int_{\Omega{\setminus}\GC}   \bbC e^{\eps_k}(\rese u{\eps_k} (r){+}\rese w{\eps_k} (r))  {:} e^{\eps_k}(\varphi) \dd r \dd x \dd \tau
\\
&
 =    \int_{t{-}\delta}^{t{+}\delta}
\int_0^\tau \int_{\Omega{\setminus}\GC} \Cr (\epl (\sfu){+} \epl(\sfw)) {:}  \epl(\varphi) \dd x \dd  r \dd \tau \,.\EEE
\end{aligned}
\]
Furthermore, since the 
the mapping $\alpha_\lambda:\R^3 \to \R^3$ is Lipschitz continuous, relying on the fact that 
$
\JUMP{\rese u{\eps_k}} \to \JUMP{\sfu} $   in $ \mathrm{C}^0([a,b]; L^{4-\rho}(\GC;\R^3))$  for all $0<\rho<4$ and all $[a,b]\subset [0,\infty)$ thanks to 
\eqref{TH2convs-k-u1-bis},
we conclude that 
\[
\begin{aligned}
&
\nu_{\eps_k}  \int_{t{-}\delta}^{t{+}\delta}
\int_0^\tau  \int_{\GC}\alpha_\lambda(\JUMP{\resei u{\eps_k}1, \resei u{\eps_k}2, 0})\cdot \JUMP{\varphi_1,\varphi_2,0} \dd \Surf(x) \dd r \dd \tau
\\
&
= \nu  \int_{t{-}\delta}^{t{+}\delta}
\int_0^\tau  \int_{\GC}\alpha_\lambda(\JUMP{\resei u{}1, \resei u{}2, 0})\cdot \JUMP{\varphi_1,\varphi_2,0} \dd \Surf(x) \dd r \dd \tau \,.
\end{aligned}
\]
Finally,   combining \eqref{TH2convs-k-u1-bis} and
\eqref{convs-TH2-uz}
 we again have  convergence
\eqref{lim-pass-adh-terms}, now for the integrals 
\[
  \int_{t{-}\delta}^{t{+}\delta}
\int_0^\tau  \int_{\GC} \kappa\rese z{\eps_k}\JUMP{\rese u{\eps_k}}\JUMP{\varphi}\,\mathrm{d}\Surf (x) \dd r \dd \tau  \longrightarrow 
  \int_{t{-}\delta}^{t{+}\delta}
\int_0^\tau  \int_{\GC} \kappa\rese z{}\JUMP{\rese u{}}\JUMP{\varphi}\,\mathrm{d}\Surf (x) \dd r 
\dd \tau\,.
\]
 \EEE

\par
 All in all, we obtain the   \emph{integrated momentum balance}
\[
\begin{aligned}
&
- \int_{t{-}\delta}^{t{+}\delta}
\int_0^\tau \int_\omega \varrho  \reseid u{}3 \dot{\varphi}_3 \dd x' \dd r  \dd \tau
+  \int_{t{-}\delta}^{t{+}\delta} \int_\omega \left[ \varrho  \reseid u{} 3(\tau) \varphi_3(\tau) {-}   \varrho  \reseid u{}3(0) \varphi_3 (0)\right] \dd x' \dd \tau
\\
& \quad +
   \int_{t{-}\delta}^{t{+}\delta}
\int_0^\tau
\int_{\Omega{\setminus}\GC}   \Dr \epl (\dot \sfu){{:}} \epl(\varphi)     \dd x  \dd r \dd \tau
+  \int_{t{-}\delta}^{t{+}\delta}
\int_0^\tau
\int_{\Omega{\setminus}\GC}     \Cr \epl (\sfu){{:}} \epl(\varphi)     \dd x  \dd r  \dd \tau
\\
& \quad 
+\nu  \int_{t{-}\delta}^{t{+}\delta}
\int_0^\tau  \int_{\GC}\alpha_\lambda(\JUMP{\resei u{}1, \resei u{}2, 0})\cdot \JUMP{\varphi_1,\varphi_2,0} \dd \Surf(x) \dd r \dd \tau
+ \int_{t{-}\delta}^{t{+}\delta}
\int_0^\tau\int_{\GC} \kappa\rese z{}\JUMP{\rese u{}}\JUMP{\varphi}\,\mathrm{d}\Surf (x) \dd r \dd \tau
\\
&
=  \int_{t{-}\delta}^{t{+}\delta}
\int_0^\tau \int_\Omega \rese f{}\varphi \dd x  \dd r \dd \tau
-  \int_{t{-}\delta}^{t{+}\delta}
\int_0^\tau \int_{\Omega\setminus \GC}    \Dr \epl (\dot \sfw){{:}} \epl(\varphi)   \dd x \dd r \dd \tau
\\
& \quad 
-  \int_{t{-}\delta}^{t{+}\delta}
\int_0^\tau \int_{\Omega\setminus \GC}    \Cr \epl (\sfw){{:}} \epl(\varphi)   \dd x \dd r \dd \tau
-
 \int_{t{-}\delta}^{t{+}\delta}
\int_0^\tau  \int_\omega \varrho \reseidd w{}3 \varphi_3 \dd x'  \dd r \dd \tau  \qquad \text{for all } \varphi \in  \mathfrak{V}\,.
\end{aligned}
\]
Dividing by $2\delta$ and letting $\delta \down 0$, by a Lebesgue point argument we obtain
 \eqref{mom-bal-DR-34}.  \EEE
\medskip

\medskip

\noindent
\textbf{Step $3$: limit passage in the semistability condition.} It is immediate to check that the semistability condition \eqref{semistab-DR-34} reduces to \eqref{true-semistability}
in this case as well. Therefore, the very same arguments as in the proof of Theorem 
\ref{mainthm-1} yield \eqref{semistab-DR-34}, which now holds for every $t\in [0,T]$ thanks to the improved convergences \eqref{TH2convs-k-u1}--\eqref{TH2convs-k-u1-bis}. 
\medskip

\noindent
\textbf{Step $4$: limit passage in the energy-dissipation inequality.}   We tackle the  passage to the limit in  
\eqref{EDB-eps}.  Just like in Step $2$, to compensate the lack of pointwise convergence of
$(\resed u{\eps_k})$ we will need to perform a further integration in time. Thus, we consider \eqref{EDB-eps}
 on the generic interval $[0,\tau]$
 and then integrate for $\tau \in [t{-}\delta, t{+}\delta]$. 
 We have 
\begin{subequations}
\begin{equation}
\label{liminf-inequality-inertia}
\liminf_{k\to\infty}   \frac{\varrho_{\eps_k}}{2}  
 \int_{t{-}\delta}^{t{+}\delta}
\int_\Omega  |\reseid u{\eps_k}3|^2  \dd x \dd \tau  \geq   
 \frac{\varrho}{2}  \int_{t{-}\delta}^{t{+}\delta}\int_\omega   |\reseid u{}3|^2  \dd x' \dd \tau \,. 
\end{equation}
\end{subequations}\EEE 
    Let us  now  revisit Step $5$ in the proof of Theorem 
\ref{mainthm-1}.
  As for the limit passage in the bulk energy,  it suffices to observe that 
\begin{equation}
\label{thm4-3:liminf-energia-t}
\begin{aligned}
\liminf_{k\to\infty}
 \int_{t{-}\delta}^{t{+}\delta}  \int_{\Omega{\setminus}\GC}\tfrac{1}{2}\mathbb{C}e^{\eps_k}(\rese u{\eps_k}){:} e^{\eps_k}(\rese u{\eps_k}) \dd x   \dd \tau 
 & \stackrel{(1)}{\geq}   \int_{t{-}\delta}^{t{+}\delta}  \int_{\Omega{\setminus}\GC}\tfrac{1}{2}\mathbb{C}(\sfe){:} \sfe \dd x  \dd \tau
 \\
 &  \stackrel{(2)}{=}  
 \int_{t{-}\delta}^{t{+}\delta}   \int_{\Omega{\setminus}\GC}\tfrac{1}{2}  \Cr\epl(\sfu){:}\epl(\sfu)  \EEE \dd x \dd \tau
 \end{aligned}
\end{equation}
where  {\footnotesize (1)} follows from the second of \eqref{conv-stress-u-new},  while  {\footnotesize (2)} ensues from \eqref{toSHOW-MM-final}. \EEE
%
Hence, 
 also relying on Hypotheses \ref{hyp:data} and \ref{Hyp-E} for $(\rese f{\eps_k})_k$
 and  $(\rese w{\eps_k})_k$, \EEE
 we have that 
\begin{equation}
\label{liminf-inequality-bulk-energy-43}
\lim_{k\to\infty}   \int_{t{-}\delta}^{t{+}\delta} 
\resei{E}{\eps_k}{\mathrm{bulk}}(\tau, \rese u{\eps_k}(\tau))   \dd \tau \EEE 
\geq  \int_{t{-}\delta}^{t{+}\delta}     \Eve^{\mathrm{bulk}}(\tau , \rese u{}(\tau)) \dd \tau \,.
\end{equation}
We easily check that 
\begin{equation}
\label{liminf-inequality-surface-energy-43}
\liminf_{k\to\infty} \int_{t{-}\delta}^{t{+}\delta}   \resei{E}{\eps_k}{\mathrm{surf}}(\rese u{\eps_k}(\tau),  \rese z{\eps_k}(\tau)) \dd \tau \geq 
 \int_{t{-}\delta}^{t{+}\delta}   \Eve^{\mathrm{surf}}(\rese u{}(\tau), \rese z{}(\tau)) \dd \tau \,. 
\end{equation}
Clearly, we again have  \eqref{liminf-inequality-R-Var},  now integrated over the interval $[t{-}\delta,t{+}\delta]$. \EEE
We now discuss the limit of the  energy dissipated by viscosity:  by \eqref{conv-stress-u-new} and  \eqref{toSHOW-MM-final} \EEE we have 
\begin{equation}
\label{thm4-3:liminf-dissip}
\begin{aligned}
\liminf_{k\to\infty}  \eps_k \int_{t{-}\delta}^{t{+}\delta}  \int_0^\tau 
\int_{\Omega\setminus \GC} \bbD_{\eps_k} e^{\eps_k} (\resed u{\eps_k}) {:} e^{\eps_k}(\resed u{\eps_k}) \dd x  \dd r \dd \tau 
 & \geq  \int_{t{-}\delta}^{t{+}\delta}  \int_0^\tau  \int_{\Omega\setminus \GC} \mathsf{D} \dot{\sfe} {:} \dot{\sfe} \dd x  \dd r \dd \tau
\\
&=   \int_{t{-}\delta}^{t{+}\delta}  \int_0^\tau  \int_{\Omega\setminus \GC}   \Dr \epl(\dot \sfu){:} \epl(\dot \sfu)   \EEE \dd x  \dd r \dd \tau \,.
\end{aligned}
\end{equation}
As for the right-hand side, we rely on
 \eqref{convergences-initial-data-34} \EEE
for   the energy convergence \EEE $\rese  E{\eps_k}(0,\rese u{\eps_k}_0, \rese z{\eps_k}_0)\to   \Eve (0,\sfu_0, \GIO  \sfz_0 \EEE ) $.
Finally, we have
\begin{equation}
\label{liminf-inequality-power-34}
\begin{aligned}
&
\lim_{k\to\infty}  \int_{t{-}\delta}^{t{+}\delta}  \int_0^\tau   \partial_t  \rese{E}{\eps_k}(r, \rese u{\eps_k}(r), \rese z{\eps_k}(r)) \dd r  \dd \tau 
\\
&
=  \lim_{k\to\infty}  \int_{t{-}\delta}^{t{+}\delta}  \int_0^\tau   \Big( {-}\int_{\Omega} \resed f{\eps_k} \rese u{{\eps_k}} \dd x {+}
 \int_{\Omega{\setminus}\GC} \bbC e^{\eps_k}(\resed w{\eps_k}){:} e^{\eps_k}(\rese u{\eps_k}) \dd x
  {+}\eps_k\int_{\Omega{\setminus}\GC} \bbD_{\eps_k} e^{\eps_k}(\resedd w{\eps_k}){:} e^{\eps_k}(\rese u{\eps_k}) \dd x 
\\
& \qquad \qquad   \qquad   \qquad 
{+}\eps_k^2 \int_\Omega   \sum_{i=1}^2 \varrho_{\eps_k}\reseiddd w{\eps_k} i  \rese u{\eps_k}_{i}  \dd x
{+}\int_\Omega \varrho_{\eps_k}   \reseiddd w{\eps_k} 3  \rese u{\eps_k}_3 \dd x \Big) \dd r  \dd \tau 
\\
& \stackrel{(1)}=  \int_{t{-}\delta}^{t{+}\delta}  \int_0^\tau \Big( {-}\int_{\Omega} \resed f{} \rese u{} \dd x {+}
 \int_{\Omega{\setminus}\GC} 
 \{\Cr \epl(\dot\sfw) {+} \Dr \epl(\ddot \sfw) \} \EEE {:} \epl(\rese u{}) \dd x 
{+}   \int_\omega \varrho\reseiddd w{} 3  \rese u{}_3 \dd x'  \Big) \dd r \dd \tau
\\
& =  \int_{t{-}\delta}^{t{+}\delta}  \int_0^\tau \partial_t  \rese{E}{}(r, \rese u{}(r), \rese z{}(r)) \dd r \dd \tau \,.
\end{aligned}
\end{equation}
For {\footnotesize (1)}, in addition to the arguments for \eqref{liminf-inequality-power} from Step $5$ in the proof of Thm.\ \ref{mainthm-1}, we have used 
 \eqref{KL-loadings-needed} and the fact that  $ e^{\eps_k}(\rese u{\eps_k}) \weakto e(\dot \sfu) $ in   $L_\loc^2(0,\infty;L^2(\Omega;\R^{3\times3}))$  \EEE
thanks to  \eqref{conv-stress-u-new} and \eqref{toSHOW-MM-final}.  \EEE 
 Hence, we conclude the validity of the energy-dissipation inequality
\begin{equation}
\label{EDI-ineq-43}
\begin{aligned}
 & 
  \int_{t{-}\delta}^{t{+}\delta}   \int_\omega \frac{\varrho}{2}   |\reseid u{}3 (\tau)|^2  \dd x' \dd \tau 
+  \int_{t{-}\delta}^{t{+}\delta}   \int_0^\tau  \int_{\Omega{\setminus}\GC} \Dr \epl(\resed u{})  {:} \epl(\resed u{})
\dd x  \dd r \dd \tau 
 +   \int_{t{-}\delta}^{t{+}\delta} \Var_{\mathsf{R}}(\sfz, [0,\tau]) \dd \tau 
 \\
 & \quad 
+  \int_{t{-}\delta}^{t{+}\delta} \Eve(\tau,\sfu(\tau),\sfz(\tau))  \dd \tau 
\\
&   \leq 
\int_{t{-}\delta}^{t{+}\delta} \left[ \int_\omega \frac{\varrho}{2}  |\reseid u{}{3} (0)|^2  \dd x' {+}   \Eve(0,\sfu(0),\sfz(0))  \right]  \dd \tau  
+ \int_{t{-}\delta}^{t{+}\delta}   \int_0^\tau  \partial_t
  \Eve(r,\sfu(r),\sfz(r)) \,  \mathrm{d}r \dd \tau\,.
 \end{aligned}
\end{equation}
Again, we divide by $2\delta$ and let $\delta \down 0$, thus concluding the energy-dissipation inequality on
the interval 
 $[0,t]$, for almost all $t\in (0,\infty)$.  \EEE
\medskip
\par
This finishes the proof.
\QED


 \section*{Acknowledgements}
E.D.\ acknowledges the support of the Austrian Science Fund (FWF) through grants  10.55776/V662, 10.55776/F65, 10.55776/Y1292, as well as 10.55776/P35359. \EEE G.B.\ and R.R.\ acknowledge the support 
of GNAMPA (INDAM).  R.R.\ was also supported by the  PRIN project \emph{PRIN 2020: ``Mathematics for Industry 4.0"}. 
\par The authors would also like to 
thank the two anonymous referees for several very stimulating and helpful suggestions, which  have led to improvements in the results and in their presentation.
\par
For open access purposes, the authors have applied a CC BY public copyright license to any author-accepted manuscript version arising from this submission.\EEE
{\small 
\bibliographystyle{abbrv}
\bibliography{BDR_lit}
}

\end{document}